\documentclass[11pt,leqno]{amsart}
\usepackage{amsthm,amsfonts,amssymb,amsmath,oldgerm}
\usepackage{graphicx}
\usepackage{makecell}
\numberwithin{equation}{section}
\usepackage{fullpage}

\allowdisplaybreaks[2]

\newcommand{\Z}{\mathbb Z}

\renewcommand\d{\partial}

\def\g{\gamma}

\def\eps{\varepsilon }

\newcommand\tRR{\widetilde{\mathbb R}}

\renewcommand\d{\partial}

\newcommand\R{\mathbb R}
\newcommand\C{\mathbb C}

\def\g{\gamma}

\def\eps{\varepsilon}

\newcommand\br{\begin{remark}}
\newcommand\er{\end{remark}}
\newcommand\bp{\begin{pmatrix}}
\newcommand\ep{\end{pmatrix}}
\newcommand{\be}{\begin{equation}}
\newcommand{\ee}{\end{equation}}
\newcommand\ba{\begin{equation}\begin{aligned}}
\newcommand\ea{\end{aligned}\end{equation}}
\newcommand\ds{\displaystyle}

\newcommand{\bap}{\begin{app}}
\newcommand{\eap}{\end{app}}
\newcommand{\begs}{\begin{exams}}
\newcommand{\eegs}{\end{exams}}
\newcommand{\beg}{\begin{example}}
\newcommand{\eeg}{\end{example}}
\newcommand{\bpr}{\begin{proposition}}
\newcommand{\epr}{\end{proposition}}
\newcommand{\bt}{\begin{theorem}}
\newcommand{\et}{\end{theorem}}
\newcommand{\bc}{\begin{corollary}}
\newcommand{\ec}{\end{corollary}}
\newcommand{\bl}{\begin{lemma}}
\newcommand{\el}{\end{lemma}}
\newcommand{\bd}{\begin{definition}}
\newcommand{\ed}{\end{definition}}
\newcommand{\brs}{\begin{remarks}}
\newcommand{\ers}{\end{remarks}}

\newcommand{\NN}{{\mathbb N}}

\newcommand{\const}{\text{\rm constant}}
\newcommand{\Id}{{\rm Id }}

\newtheorem{theorem}{Theorem}[section]
\newtheorem{proposition}[theorem]{Proposition}
\newtheorem{corollary}[theorem]{Corollary}
\newtheorem{lemma}[theorem]{Lemma}

\theoremstyle{remark}
\newtheorem{remark}[theorem]{Remark}
\theoremstyle{definition}
\newtheorem{definition}[theorem]{Definition}
\newtheorem{assumption}[theorem]{Assumption}

\newtheorem{example}[theorem]{Example}

\newcommand{\RM}{\mathbb{R}}
\newcommand{\ZM}{\mathbb{Z}}

\newcommand{\CM}{\mathbb{C}}
\newcommand{\NM}{\mathbb{N}}

\newcommand{\beq}{\begin{equation}}
\newcommand{\eeq}{\end{equation}}
\title{Spectral stability of inviscid roll waves}
\author{ Mathew A. Johnson}
\address{University of Kansas, Lawrence, KS 66045}
\email{matjohn@ku.edu}
\thanks{Research of M.J. was partially supported under NSF grant no. DMS-1614785.}

\author{Pascal Noble}
\address{Institut de Math\'ematiques de Toulouse, UMR5219, Universit\'e de Toulouse, CNRS, INSA, F-31077 Toulouse, France}
\email{pascal.noble@math.univ-toulouse.fr}
\thanks{Research of P.N. was partially supported by the French ANR Project BoND ANR-13-BS01-0009-01}

\author{L.~Miguel Rodrigues}
\address{Univ Rennes, CNRS, IRMAR - UMR 6625, F-35000 Rennes, France}
\email{luis-miguel.rodrigues@univ-rennes1.fr}
\thanks{Research of M.R. was partially supported by the French ANR Project BoND ANR-13-BS01-0009-01 and the city of Rennes}

\author{Zhao Yang}
\address{Indiana University, Bloomington, IN 47405}
\email{yangzha@indiana.edu}
\thanks{Research of Z.Y. was partially supported by the Hazel King Thompson Summer Reading Fellowship}

\author{Kevin Zumbrun}
\address{Indiana University, Bloomington, IN 47405}
\email{kzumbrun@indiana.edu}
\thanks{Research of K.Z. was partially supported
under NSF grant no. DMS-1400555 and DMS-1700279}

\begin{document}

\begin{abstract}
We carry out a systematic analytical and numerical study of spectral stability of discontinuous roll wave solutions of the inviscid Saint Venant equations, based on a periodic Evans-Lopatinsky determinant analogous to the periodic Evans function of Gardner in the (smooth) viscous case, obtaining a complete spectral stability diagram useful in hydraulic engineering and related applications. In particular, we obtain an explicit low-frequency stability boundary, which, moreover, matches closely with its (numerically-determined) counterpart in the viscous case. This is seen to be related to but not implied by the associated formal first-order Whitham modulation equations.
\end{abstract}
\date{\today}
\maketitle



\tableofcontents

\section{Introduction}\label{s:introduction}

In this paper, we study the spectral stability of periodic ``roll wave'' solutions of the inviscid St.~Venant equations
\ba\label{sv_intro}
\d_th+\d_x q&=0, \\
\d_t q +\d_{x}\left(\frac{q^2}{h}+\frac{h^2}{2F^2}\right)&=h-\frac{|q|\,q}{h^2}
\ea
modeling inclined shallow water flow.  Here, $t$ and $x$ denote elapsed time and spatial location along the incline, $h$ and $q=hu$ denote the fluid height and total flux at point $(x,t)$, $u$ is a vertically averaged velocity, and
%\be\label{p}
$%p(h)= 
\frac{h^2}{2F^2}$
%\ee
is an effective fluid-dynamical force analogous to pressure in compressible flow\footnote{Indeed, the lefthand side of \eqref{sv_intro} may be recognized as the equations of isentropic compressible gas dynamics, with $h$ and $q$ playing the roles of density and momentum, and a polytropic ($\gamma=2$) equation of state.}. The parameter $F$ is a nondimensional {\it Froude number}, defined explicitly as
%changed PN: the Froude-->a Froude: no unique definition of Froude, depending on the choice of U0
%\be\label{F}
$F=\frac{U_0}{ \sqrt{H_0 g \cos \theta }}$,
%\ee
where $\theta$ is angle from horizontal of the incline, $g$ the gravitational constant, and $H_0$ and $U_0$ are fixed chosen reference values of height and velocity, i.e., units of measurement for $h$ and $u$. Source terms $h$ and $-|q|q/h^2=-|u|u$ on the righthand side represent opposing accelerating and resisting forces of gravity and turbulent bottom friction, the latter approximated by {\it Chezy's formula} as proportional to fluid speed squared \cite{D49,BM}.

%changed PN: nu inverse of Reynolds number.
System \eqref{sv_intro} and its viscous counterpart, obtained by including the additional term $\nu \d_x (h \d_x(q/h))$, with $\nu>0$ a non dimensional constant (the inverse of a Reynolds number), on the righthand side of \eqref{sv_intro}(ii), are both commonly used in hydraulic engineering applications, for example to model shallow fluid flow in a canal or spillway. See \cite{BM} for an interesting survey of this topic and applications, including a description of the reduction to nondimensional form \eqref{sv_intro} \cite[p. 14]{BM}.

Our particular interest here is in the phenomena of {\it roll waves} \cite{C34,Br1,Br2}, a well-known hydrodynamic instability associated  with destabilization of constant flow \cite{J25} consisting of a periodic series of shocks, or ``bores", separated by smooth monotone wave profiles, advancing down the incline with constant speed. Such waves are important due to their destructive capacity, both through overflow of a confining channel due to increased amplitude variation and through the ``water-hammer" effect of periodic shock impacts on hydraulic structures \cite[p.149]{D49}, \cite[p.1]{BM}, \cite[p.7]{Huang}. This motivates the study of both the {existence}, and, as the physical selection mechanism determining naturally occurring amplitudes and frequencies,
%time-evolutionary
dynamical {\it stability} of roll waves.

The existence problem is by now well understood. Existence of roll waves for \eqref{sv_intro} was established by Dressler by exact solution \cite{D49}, and for the viscous version of \eqref{sv_intro} by H\"arterich \cite{H03} using singular perturbation analysis in the limit as $\nu\to 0^+$.
%MR
See also \cite{BJNRZ2} for existence of a full family of small-amplitude viscous long waves in the near-onset regime $F>2$, $|F-2|\ll1$, by a Bogdanev-Takens bifurcation from cnoidal waves of the Korteweg-de Vries equation and
\cite{NM} for existence of
nearly harmonic
small-amplitude roll waves for general viscosity by a Hopf bifurcation analysis
from constant states.
In the inviscid case, it is known that there exists a 3-parameter family of 
%
%waves, 
wave profiles, 
parametrized by the Froude number $F>2$ and parameters $H_s>0$ and $H_-\in(H_{min}(F,H_s),H_s)$, where $H_s$ and $H_-$ denote, respectively, the fluid heights at a distinguished ``sonic" point
%MR: added
(where wave profile equations are degenerate)
and at the left endpoint of the (monotone increasing) continuous profiles separating jumps; see Figures~\ref{fig:roll}(a)-(b) for a typical solution and physical example.\footnote{The latter reproduced from \cite{C34}.} By scaling invariance, this may be reduced to a 2-parameter family, taking without loss of generality $H_s\equiv 1$. See Section~\ref{s:existence} for details. In the viscous case,
%MR: added
for each fixed $\nu$
there is likewise a 3-parameter family of waves, converging as $\nu\to 0^+$ to matched asymptotic expansions of the inviscid profiles \cite{H03}.

The stability problem, by contrast, despite substantial results in various asymptotic limits, remains on the whole somewhat mysterious, especially in the medium-to-large Froude number regime $2.5 \lessapprox F\lessapprox 20$ that is relevant to hydraulic engineering applications \cite{J25,Br1,Br2,A,RG1,RG2,FSMA}. As described, e.g., in \cite{YK,YKH,Kr92,BM,BN,BJNRZ2}, the near-onset regime, $F>2$, $|F-2|\ll 1$, in the viscous case is well-described by a weakly nonlinear ``amplitude equation" consisting of a singularly perturbed Korteweg-de Vries equation modified by Kuramoto-Sivashinsky diffusion, for which stability boundaries may be explicitly determined in terms of integrals of elliptic functions. This formal description (in which viscosity plays an important role), has at this point been rigorously validated at the linear and nonlinear level in \cite{JNRZ12,JNRZ13,B,BJNRZ2}, in part using rigorous computer-assisted proof. However, numerical investigations of \cite{BJNRZ2} indicate that its regime of validity is limited to approximately $2<F\lessapprox 2.3$, which is outside the regime of interest for typical hydraulic engineering applications.
\begin{figure}
\begin{center}
$
\begin{array}{lr}
(a) \includegraphics[height=5cm,width=8cm]{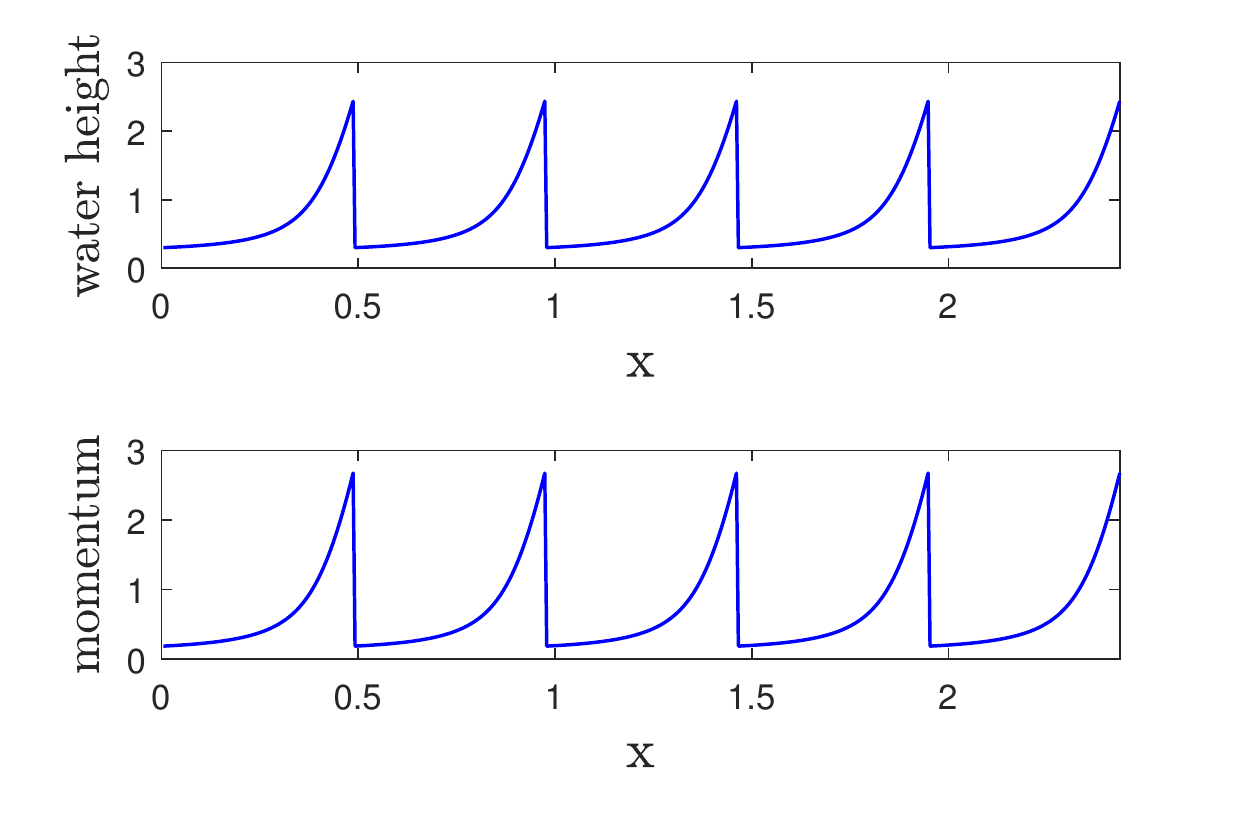}
   & (b)\ \includegraphics[height=6cm,width=6cm]{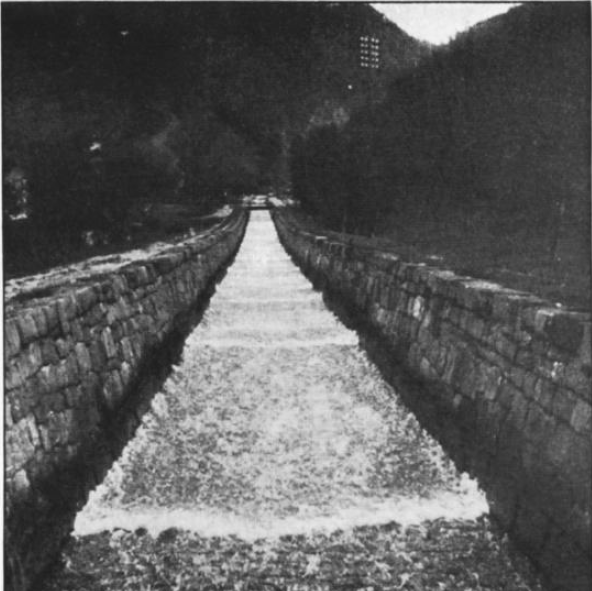}
\end{array}
$
\end{center}
\caption{(a) Simulated Dressler roll waves. (b) Roll waves in the Grunnbach conduit.}
\label{fig:roll}
\end{figure}

Away from onset, a standard approach in pattern formation is to replace the weakly nonlinear amplitude equation approximation by a formal Whitham modulation expansion \cite{W}, a multiscale expansion in similar spirit, but built around variations in the manifold of periodic solutions rather than linear perturbations of a constant state. In contrast to the near-onset case, where
%MR
%the stability of high-frequency modes follows by a continuity argument in the limit towards the constant state,
the full spectral stability is proved to be encoded in a relevant weakly-nonlinear amplitude equation,
such Whitham expansions yield only
%
%{\it low-frequency}, or "long-wave" stability information.
stability information that is low-frequency on wave parameters, in particular low-Floquet or {\it side-band} for the original equations.
See \cite{NR} for a precise discussion of what may rigorously deduced from this kind of approach, including higher-order versions thereof.
Formal analyses of this general type were carried out by Tamada and Tougou in \cite{TT,T} to obtain formal
%low-frequency
side-band
stability criteria.  
%More recently, Boudlal and Liapidevskii have proposed a different formal low-frequency stability condition based on direct spatial averaging \cite{BL02}.
More recently, Boudlal and Liapidevskii have proposed a different formal necessary stability condition based on direct spatial averaging \cite{BL02}, presumably also related to  (but not necessarily limited to) low-frequency stability.

In a different direction, Jin and Katsoulakis \cite{JK} have carried out weakly nonlinear asymptotics in the high-frequency limit, obtaining an amplitude equation consisting of a negatively damped Burgers equation, from which one may conclude formally instability of roll waves with sufficiently small period. Working by very different techniques, based on a direct linearized eigenvalue analysis, and WKB approximation, Noble \cite{N1,N2} has obtained complementary {\it high-frequency} stability criteria, from which he was able 
%MR
to conclude, for waves with sufficiently large period, high-frequency stability, that is, stability with respect to perturbations of sufficiently high 
time 
frequency. However, in the large-Froude number regime,
(i) the various low-frequency criteria do not agree, leading to confusion as to what precisely they capture;
(ii) the high-frequency condition of Noble, though theoretically conclusive, is difficult to analyze outside of the large-period limit computed by Noble; and
(iii) up to now, except in the unstable high-frequency limit studied by Jin-Katsoulakis, complete stability information has been given for no single wave away from onset. In particular,
%MR
%TODO: smooth formulation if needed
%for $|F-2|\not \ll 1$,
for any $F$ that is not close to $2$,
{\it stability has not been verified for any roll wave solution of \eqref{sv_intro}.}

In summary, the stability theory for large-Froude number roll waves remains far from clear, consisting of disparate, mainly formal, pieces with no unified whole. Our goal in this paper is to shed light on the situation by a systematic investigation combining rigorous analysis of the exact eigenvalue equations with numerical investigation {\it to obtain a complete spectral stability diagram for the family of discontinuous roll wave solutions of the inviscid system \eqref{sv_intro}}, and at the same time determining a precise connection between low-frequency stability and the formal Whitham modulation equations.

\subsection{Viscous stability, Whitham equations, and the Evans-Lopatinsky determinant}\label{s:lop2}

\subsubsection*{Viscous stability} Our main impetus for the present \emph{inviscid} stability analysis is the recent stability analysis carried out away from onset in \cite{NR,JNRZ13,RZ,BJNRZ2} for smooth roll wave solutions of the viscous version of the Saint Venant equations \eqref{sv_intro}.
%MR: precised
There, it was shown rigorously that (i) nonlinear modulated stability follows from diffusive spectral stability; (ii) in stable cases associated second-order formal Whitham modulation systems provide accurate large-time approximations; and (iii) in any case (stable or unstable),
%
%low-frequency behavior-encoded by spectral perturbation of neutral modes with respect to Bloch, or "sideband", frequency-
spectral perturbations of neutral modes with respect to Floquet/Bloch frequency
agree to second order with those predicted by an associated formal
second-order
Whitham modulation system. Finally, the spectrum was approximated numerically for the linearized problem about the waves using the periodic Evans function $D(\lambda, \xi)$, an analytic function introduced by Gardner \cite{G} whose zeros correspond to spectra $\lambda$ and associated Bloch/Floquet numbers $\xi$; see \cite{BJNRZ2} for further description.

The surprising result in the viscous case was that, away from onset, i.e., for $F\gtrapprox 2.5$, the stability diagram has a different but equally simple power-law description from the one near onset, with stable wave parameters corresponding to the region within a lens-shaped region bounded by an upper and a lower stability boundary. This observation was obtained purely numerically in \cite{BJNRZ2}, with
%MR: smooth if needed
%no formal or heuristic explanation.
no explanation of any kind, not even formal.
However, there it was noted that in the region of stability, wave profiles seemed to converge to discontinuous ``Dressler-wave" solutions for large values of $F$, suggesting that an explanation of the observed behavior might be found in the study of the inviscid equations \eqref{sv_intro} and its singular perturbation via the zero viscosity limit. Moreover, the numerical computations in the viscous case, lying in simultaneous large period and small viscosity limits, were quite delicate, with a total reported computation time exceeding 40 days (machine time) on the IU supercomputer cluster \cite[\S 9.2]{BJNRZ2}.

From the viewpoint of both physical insight and reliability/efficiency of numerical computations, these findings motivate the study of the inviscid equations as an organizing center for the observed viscous results. More, they suggest the analytical strategy that we shall follow here, of adapting to the inviscid case the tools that have proved successful in the viscous one, Whitham modulation, direct spectral expansion, Floquet-Bloch analysis, and the periodic Evans function of Gardner.

\subsubsection*{Inviscid Whitham system}
We begin by identifying a first-order Whitham modulation system analogous to that of the viscous case \cite{W,Se,NR,JNRZ13}. In the viscous case, this appears as
\begin{equation}\label{whit}
\begin{aligned}
\d_t \overline{H}+\d_x \overline{Q}&=0, \\
\d_t k - \d_x \omega&=0,
\end{aligned}
\end{equation}
where $k=1/X$ and $\omega=-kc$ are spatial and temporal wave numbers for roll wave solutions
$$
(h,q)(x,t)=(H,Q)(\omega t + kx)
$$
propagating with speed $c$, with periodic profile $(H,Q)$ of period one, upper bars denoting averages over one period. As described in \cite{Se,NR,JNRZ13}, the system \eqref{whit} may be obtained 
in the viscous case by averaging all conservative equations of the original PDE \eqref{sv_intro} (here, just the first, $h$ equation), and augmenting the resulting system with the ``eikonal equation" relating $k$ and $\omega$. For a derivation, see \cite{Se}, \cite{NR} and \cite[Appendix B.1.1]{JNRZ13}. 
%CHANGED-MJ This discussion is about the viscous case... I moved this (good) comment below to the discussion of the inviscid case.
%%MR: added
%A remarkable feature of the system~\eqref{whit}, associated with \eqref{sv_intro}, is that all its coefficients may be given explicitly, 
%since wave profile may themselves be explicitly given up to the resolution of a scalar ODE with a rational vector field.

Following the usual Whitham formalism \cite{W,OZ1,Se,NR,BNR,KR}, the system \eqref{whit} is expected to describe
%
%low-frequency, or ``long wave''
critical side-band
behavior of a perturbed periodic viscous roll wave as a
low-frequency, or ``long wave''
modulation along the 2-parameter family (holding fixed the third, physical parameter $F$) of nearby viscous roll waves. Note that, indeed, since all terms are functions of wave parameters, as follows from the ODE existence theory for the viscous roll waves, we may view \eqref{whit} as a $2\times 2$ system of first-order conservation laws governing the evolution of these two parameters. This yields via a consistency argument the formal
%low-frequency
side-band
stability condition of well-posedness, or {\it hyperbolicity} of \eqref{whit}.

Considering now the inviscid case, we note the lack of a systematic, multi-scale expansion to derive a corresponding inviscid Whitham system, due to the discontinuous ``shock" nature of the profiles (invalidating standard derivations based on smooth solutions). Nevertheless, based on the convergence of profiles proved in \cite{H03}, together with expected convergence of low-frequency behavior in the zero-viscosity limit, we propose as an inviscid Whitham system simply the same system of equations \eqref{whit}, but substituting the \emph{inviscid} dependence of $\overline{H}$, $\overline{Q}$, $k$, and $\omega$ on wave parameters $(H_-,H_s)$. While this is merely an optimistic guess at this point, we will justify this system in Section~\ref{s:comparison} by a direct comparison with the associated spectral problem.  
%CHANGED-MJ added MR's remark from above
%We note that a remarkable feature of the system~\eqref{whit}, associated with \eqref{sv_intro}, is that all of its coefficients may be given \emph{explicitly}, 
%since the wave profile may themselves be explicitly given as the resolution of a scalar ODE with a rational vector field.
We note that a remarkable feature of the system~\eqref{whit} associated with \eqref{sv_intro} is that all of its coefficients may be given \emph{explicitly},
since the wave profiles may themselves be explicitly given as the resolution of a scalar ODE with a rational vector field

For comparison, we present also the alternative low-frequency model proposed in \cite{BL02}:
\ba\label{av}
\d_t \overline{H}+\d_x \overline{Q}&=0, \\
\d_t \overline{Q}+\d_{x}\left(\overline{
\frac{Q^2}{H}+%p(H) 
\frac{H^2}{2F^2}}\right)&=0.
\ea
This model was proposed based on the observation that the average of undifferentiated terms, being equal to the jump across the shock in differentiated quantities, must vanish by the Rankine-Hugoniot jump condition at the shock. Here, again, \eqref{av} is to be interpreted as a $2\times 2$ system of conservation laws determining the evolution of wave parameters, with hyperbolicity corresponding to
%MR
%a low-frequency stability condition.
some necessary condition for spectral stability.
Like \eqref{whit}, the system \eqref{av} was proposed on heuristic grounds rather than systematic %formal 
expansion, hence likewise requires validation by external means. To rigorously justify/compare such heuristic %low-frequency
stability predictions is an important tangential goal of our investigations. We show in Section \ref{s:comparison} that \eqref{av} despite its intuitive appeal does {\it not} accurately predict stability,
%low-frequency
side-band
or otherwise, giving rise to false negative and positive results, and in some cases wrongly predicting instability of (globally) linearly stable waves.
%CHANGED-MJ
%{\bf Further, we demonstrate that the Whitham system \eqref{whit} does accurately capture...}
{\it That is, the averaged system \eqref{whit} suggested by the Whitham formalism accurately predicts low-frequency stability also in the inviscid case, while the seemingly similar averaged system \eqref{av} does not successfully predict stability on any scale.}

\subsubsection*{Periodic Lopatinsky determinant}\label{s:lop}
Our primary approach to the stability study of inviscid roll waves, bypassing issues of formal approximation as described above, is, following \cite{N1,N2}, to work directly with the exact eigenvalue equations for the linearization of \eqref{sv_intro} about the wave. Using the rigorous abstract conclusions so obtained, we then attempt to deduce useful approximations and to evaluate the validity of proposed formal stability criteria such as, for example, hyperbolicity of modulation systems \eqref{whit} and \eqref{av}. Our main contributions beyond what was done in \cite{N1,N2} are the characterization of normal modes as spectra in the usual sense of an appropriate linear operator, clarifying the connection of the framework of \cite{N1,N2} to standard Floquet/Bloch theory,\footnote{In particular, bearing on the issue of ``completeness'' of normal modes solutions.} and the introduction of a {\it stability function} or periodic Evans-Lopatinsky determinant analogous to the periodic Evans function of Gardner \cite{G} in the viscous case. The latter proves to be extremely useful for both numerical and analytical computations; indeed, it is the central object in our development.

%MR: tried to clarify further and correct
Here, we give a quick heuristic derivation of the stability function and associated eigenvalue equations. A rigorous derivation is given in Section \ref{s:evans-lop}. Consider a general system of balance laws
\be\label{ab}
\d_tw+\d_x(F(w))=R(w),
\ee
with $w$ valued in $\R^n$, smooth except at a sequence of shock positions $x_j=x_j(t)$, satisfying the Rankine-Hugoniot jump conditions at the $x_j$
\be\label{rh}
x_j'(t)[w]_j=[F(w)]_j
\ee
where $[h]_j:= h(x_j^+)-h(x_j^-)$. Let us assume that $W$ is a $X$-periodic traveling wave solution of \eqref{ab}, with 
wave speed $c$ and 
a single shock by period. 
%MR: removed the c=0 assumption for clarity
%Without loss of generality, by working in a co-moving coordinate frame, we may thus assume that $W$ is stationary with shocks located at $X_j=jX$.
Working in a co-moving coordinate frame turns $W$ in a stationary solution, with shocks located at $X_j=jX$.
Then, in this frame, 
formally linearizing \eqref{ab} about $W$ (i.e., taking a pointwise, or Gateaux, derivative not required to be uniform in $x,t$), we obtain the ``interior equations''
\be\label{abint}
\d_tv+ \d_x(Av)=Ev \quad\textrm{on}\quad\tRR\,:=\,\bigcup_{j\in Z} (jX, (j+1)X),
\ee
where here
$$
A(x):=d_wF(W(x))
-c\Id
\qquad\textrm{and}\qquad
E(x):=d_wR(W(x)),$$
while at the shock locations $X_j=jX$ we obtain linearized jump conditions
\be\label{abjump}
y_j'\,[W]_j=y_j\,[A\,W']_j+ [Av]_j,
\ee
with $v$ and $y_j$ now denoting perturbations in $w$ and $x_j$, where the second term in \eqref{abjump} is obtained due to displacement $y_j$ in the location of the jump (error terms by $H^1$ extension remaining of higher order so long as solutions remain bounded in $H^1$ and $y_j$ remains sufficiently small).

As the system \eqref{abint}-\eqref{abjump} has $X$-periodic coefficients, it may be analyzed via Floquet theory, i.e., through the use of the Bloch-Laplace transform.  Precisely, decomposing solutions of \eqref{abint}-\eqref{abjump} into ``normal mode" solutions of the form
\be\label{Bloch}
v(x,t)=e^{\lambda t}e^{i\xi x}\check{v}(x) ,\quad  y_j(t)= \chi e^{\lambda t}e^{i\xi (j-1)X},
\qquad \lambda \in \C,\quad\xi\in [-\pi/X,\pi/X],
\ee
with $\check{v}(\cdot)$ $X$-periodic and $\chi$ constant, and setting $w(x)=\check{v}(x)e^{i\xi x}$, we use the quasi-periodic structure in \eqref{Bloch} to reduce the spectral problem to a one-parameter family of Floquet eigenvalue systems
\ba\label{eval}
(Aw)'=(E-\lambda \Id) w\textrm{ on }(0,X),\\
\chi (\lambda \{W\} - \{AW'\})=\{Aw\}_\xi,\\
\ea
parametrized by $\xi\in[-\pi/X,\pi/X]$, where here $'$ denotes $\d_x$. Here above we have used periodic jump notation
\be\label{jumpnote}
\{f\}:= f|_0^X,\qquad
\{f\}_\xi:= f(X^-)- e^{i\xi X}f(0^+)\,.
\ee
Note that this reduces the whole-line problem \eqref{abint}-\eqref{abjump} to consideration of a family of problems posed on a single periodic cell $(0,X)$.

As a result, the $X$-periodic wave $W$ is said to be (spectrally) \emph{unstable} if there exists a $\xi\in[-\pi/X,\pi/X]$ such that the eigenvalue problem \eqref{eval} has a non-trivial solution $w$ for some $\lambda\in\CM$ with $\Re(\lambda)>0$. Otherwise, the wave $W$ is said to be (spectrally) \emph{stable}. For details, see Section~\ref{s:evans-lop} below.

By stationarity of $W$, we have 
%MR: removed c=0
$AW'=(F(W)-cW)'=R(W)$. 
Thus we find that existence of a non-trivial solution to the Floquet eigenvalue problem is equivalent to linear dependence of $\{\lambda W -R(W)\} $ and $\{Aw\}_\xi$, for some solution $w$ to \eqref{eval}(i) or, equivalently under Assumption~\ref{local} from Section~\ref{s:evans-lop}, vanishing of the determinant
\be\label{det}
\Delta(\lambda,\xi):=
\det \bp \{\lambda W -R(W)\} & \{Aw_1\}_\xi&\dots & \{Aw_{n-1}\}_\xi \ep \, ,
\ee
$\lambda \in \C$, $\xi \in [-\pi/X,\pi/X]$, where $w_1,\dots, w_{n-1}$ form a basis of solutions of interior eigenvalue equation \eqref{eval}(i). Here in the abstract discussion we are taking for granted minimal degeneracy of the structure of interior equations, responsible for a loss of one dimension in the space of solutions,
%KZ added
and existence of an analytic in $\lambda$ choice of basis 
%$q_1,\dots,q_{n-1}$; 
$w_1,\ldots,w_{n-1}$;
see Assumptions~\ref{local}-\ref{extension} below. 
With this respect, note that a consequence of the periodic single-shock structure together with the Lax characteristic condition, already observed in \cite{N1,N2}, is the presence of one sonic, or characteristic, point of $F$ in $W$, that is, existence of a singular point of \eqref{eval}(i); see Section~\ref{s:existence}. We refer the reader to Section~\ref{s:evans-lop} and Appendix~\ref{s:local} for further details and in particular proofs that 
%KZ clarified
on one hand the underlying Assumptions~\ref{local}-\ref{extension} hold if each periodic cell contains only one singular point and this singular point is a regular-singular point, and on the other hand this scenario includes the case of System~\eqref{sv_intro}.
%on one hand the underlying Assumption~\ref{local} holds if each periodic cell contains only one singular point and that this singular point is a regular-singular point, and that on the other hand this includes System~\eqref{sv_intro}.

%TODO: change to >0 for consistency
%\begin{definition}\label{specdef}
%We call $\Delta$ the {\it periodic Evans-Lopatinski determinant}, or ``stability function'' for $W$. We define the {\it nonstable spectrum} of $W$ as the set of roots $\Re \lambda \geq 0 $, $\lambda\neq 0$ of $\Delta(\cdot, \xi)$ for some $\xi \in [-\pi/X,\pi/X]$. Likewise, we define {\it spectral stability} of $W$ as the absence of zeros of $\Delta$ for $\Re \lambda \geq 0$ except at $\lambda=0$.
%\end{definition}
\begin{definition}\label{specdef}
	We call $\Delta$ the\footnote{Note that we slightly abuse the terminology here, since ``the" determinant is not canonically defined but depends on the choice of basis $w_1,\dots, w_{n-1}$ 
	in Assumption~\ref{extension} below. However, this
	does not affect our notion of spectra or stability.}
	{\it periodic Evans-Lopatinsky determinant}, or ``stability function'' for $W$. We define the {\it nonstable spectrum} of $W$ as the set of roots $\lambda$ with $\Re \lambda >0 $ of $\Delta(\cdot, \xi)$ for some $\xi \in [-\pi/X,\pi/X]$. Likewise, we define {\it spectral stability} of $W$ as the absence of zeros of $\Delta$ with $\Re \lambda > 0$.
\end{definition}

In Section~\ref{s:evans-lop}, we show by a study of the associated resolvent equation that the spectrum as defined above agrees with the usual $H^1(\tRR)\times \ell^2(\Z)$-spectrum of the linearized problem about $W$.

\br\label{comprmk}
For comparison, the periodic Evans function of Gardner may be written as
$$
D(\lambda,\xi):= \det \bp \{Aw_1 \}_\xi & \dots & \{Aw_n \}_\xi\ep,
$$
where $w_1, \dots, w_n$ are a basis of solutions to the eigenvalue ODE. Meanwhile, the Evans-Lopatinsky determinant associated with a detonation-type solution of \eqref{ab}, featuring a single shock discontinuity at $x=0$, appears \cite{E,JLW,Z1}, with $\{h\}$ now denoting $h|^{+\infty}_{-\infty}$, as
%MR: precised
%TODO: check
$$
\delta(\lambda):= \det \bp \{\lambda W -R(W)\} & Aw_1(0^+)&\dots & Aw_{j_0}(0^+)&Aw_{j_0+1}(0^-)&\dots & Aw_{n-1}(0^-)\ep,
$$
where $w_1$, $\dots$, $w_{j_0}$ are functions on $(0,\infty)$ that form a basis of solutions to the eigenvalue ODE decaying to zero at $\infty$ and $w_{j_0+1}$, $\dots$, $w_{n-1}$ are functions on $(-\infty,0)$ that form a basis of solutions to the eigenvalue ODE decaying to zero at $-\infty$. In particular, the periodic Evans-Lopatinsky determinant interpolates between the periodic Evans function of the viscous case and the Lopatinsky determinant of inviscid shock/detonation theory.
\er

\subsection{Main results}\label{s:results}
We now describe our main results, both analytical and numerical.
%TODOkz: or?
%We now describe our main analytical and numerical results.

\subsubsection*{Low-frequency expansion and Whitham modulation equations}\label{s:lf}
Let $W$ be an $X$-periodic roll wave solution of \eqref{sv_intro} with speed $c$. Following Section~\ref{s:lop}, we make the change of coordinates $x\to x-ct$ to a co-moving frame in which $W$ is stationary, and define the periodic Evans-Lopatinsky determinant $\Delta(\lambda,\xi)$ as in \eqref{det}.  Our first observation is that $\Delta$ possesses the following surprisingly special structure.

\bpr\label{factor}
$\lambda=0$ is a root of $\Delta(\cdot,\xi)$ for all $\xi\in [-\pi/X,\pi/X]$.
Equivalently, $\Delta$ factors as
\be\label{fact}
\Delta(\lambda,\xi)=\lambda \hat \Delta(\lambda, \xi)
\ee
for all $\lambda\in\CM$ and $\xi\in[-\pi/X,\pi/X]$, where $\hat \Delta$ is an analytic function in both $\lambda$ and $\xi$. Moreover, there exist coefficients $\alpha_0$ and $\alpha_1$ such that uniformly in $\xi\in[-\pi/X,\pi/X]$,
\be\label{fact2}
\hat\Delta(\lambda,\xi)
=\alpha_0\,\lambda+\frac{e^{i\xi X}-1}{iX}\,\alpha_1\,+O(|\lambda|\,(|\lambda|+|\xi|))\,.
\ee
% $\hat \Delta(0,\xi)$ is calculable in closed form, for all $\xi\in [-\pi/X,\pi/X]\subset \R$.
\epr

In particular in \eqref{fact2}
$$
\alpha_0=\d_\lambda\hat\Delta(0,0)\,,\qquad\qquad
\alpha_1=\d_\xi\hat\Delta(0,0)=-\frac{iX}{2}\hat\Delta(0,\pi/X)\,,
$$
and, for any $\xi\in[-\pi/X,\pi/X]$,
$$
\hat\Delta(0,\xi)
=\frac{e^{i\xi X}-1}{iX}\,\alpha_1\,.
$$

From \eqref{fact}-\eqref{fact2}, we obtain quite detailed information on low-frequency spectral expansions.

\bc\label{alternative}
$\lambda=0$ is a double root of $\Delta(\cdot,\xi)$ for some non zero $\xi\in[-\pi/X,\pi/X]$ if and only if it is a double root for any $\xi\in[-\pi/X,\pi/X]$.
\ec

\bc\label{exp}
Assume that $\Delta(\cdot,0)$ has a root of multiplicity exactly two at $\lambda=0$, i.e., $\alpha_0\neq 0$. In particular, there are two spectral curves $\lambda_j(\cdot)$, $j=1,2$, that are analytic in $\xi$ near $\xi=0$ and such that, for $(\lambda,\xi)$ near $(0,0)$, $\Delta(\lambda,\xi)=0$ is equivalent to $\lambda=\lambda_j(\xi)$ for some 
%$j\in \{1,2\}$. Moreover, they may be chosen to expand as
$j\in \{1,2\}$, where
\be\label{specexp}
\lambda_1(\xi)=- i\xi \alpha \big( 1 - i\xi \gamma +O(\xi^2)\big),
\qquad
\lambda_2 (\xi)\equiv 0 \quad \hbox{\rm for $|\xi|$ sufficiently small,}
\ee
%
%\be\label{specexp}
%\lambda_1(\xi)=- i\xi \alpha \big( 1 - i\xi \gamma  +O(\xi^2)\big),
%\qquad
%\lambda_2 (\xi)\equiv 0,
%\ee
for some real constants $\alpha,\gamma\in\RM$. Moreover, $\alpha$ is obtained from \eqref{fact2} by
\begin{equation}\label{alpha}
\alpha=\frac{\alpha_1}{i\alpha_0}
=\frac{\d_\xi\hat \Delta(0,0)}{i\d_\lambda\hat \Delta(0,0)}
=-\frac{X \hat \Delta(0,\pi /X)}{2 \d_\lambda\hat \Delta(0,0)}.
\end{equation}
\ec

Our numerical investigations suggest that $\alpha_0$ never vanishes so that indeed expansions \eqref{specexp} are directly relevant everywhere. The first-order expansions of $\lambda_j(\xi)$ by \eqref{specexp} are purely imaginary, or neutral, hence {\it do not yield directly low-frequency stability information}. However, the special structure of \eqref{specexp} allows us to conclude that second-order low-frequency stability information does depend on the sign of the first-order coefficient. Indeed, the second-order term $-\alpha \gamma \xi^2$ in $\lambda_1(\xi)$, determining low-frequency stability or instability according as $\Re \alpha \gamma>0$ or $<0$, has sign depending on that of the first-order coefficient $\alpha$. Moreover the term $\gamma$ in \eqref{specexp} is found numerically not to vanish on the closure of the set of stable waves, hence numerical evidence that {\it the transition stability/instability for roll waves through low-frequency stability boundary is the curve in parameter space on which $\alpha=0$.}
\br
The above argument relies partly on numerical observations. However it is consistent with the intuition that since when $\alpha=0$ a full loop of spectrum is present at $\lambda=0$ we expect that a full loop passes through zero when $\alpha$ changes sign; see Figure~\ref{spectrum}(a). In contrast when only the curvature changes sign, see Figure~\ref{spectrum}(b), one expects that some 
%KZ
%mean-frequency 
medium-frequency 
part of the spectrum, far from zero, has already gone a while ago from the stable to the unstable half-plane; see Figure~\ref{medfrespectrum}. We will give another independent supporting argument below, based on a subharmonic stability index, that does not require information on $\gamma$ and proves that when $\alpha$ changes sign the number of real positive roots of $\Delta(\cdot,\pi/X)$ changes.
\er
Recall {\it Serre's Lemma} \cite{Se,NR,JNRZ13} in the viscous case, that the characteristics of the first-order Whitham modulation system, after a coordinate change $x\to x-ct$ to comoving frame (taking characteristics $\tilde \alpha_j$ to $\tilde \alpha_j-c$), agree with the coefficients of the first-order expansion at $(\xi, \lambda) =(0,0)$ of the exact spectrum of the linearized operator about the wave, encoding
%low-frequency
side-band
behavior. 
For a general first-order system
$f^0(w)_t + f(w)_x=0$, $w\in \R^n$, define the associated {\it dispersion function }
$$D^W(\lambda, \xi):=\det \left(\lambda \frac{df^0}{dw} +i\xi \frac{df}{dw}\right),$$
an $n$-degree homogeneous polynomial in $\xi$, $\lambda$.  The roots $\lambda=\lambda_j(\xi)$ of $D^W(\lambda,\xi)=0$ are then the dispersion
relations associated with the system at a given background point $w$ (typically not mentioned), themselves first-order homogeneous
The following results give a generalization to the inviscid case, {\it rigorously justifying the formal Whitham modulation system \eqref{whit} as a predictor of low-frequency behavior.}

\bpr[Inviscid Serre's Lemma]\label{serreprop}
There exists an explicit non zero constant $\Gamma_0$ such that as $\lambda\to0$, uniformly in $\xi\in[-\pi/X,\pi/X]$,
$$
\Delta(\lambda,\xi)
\,=\,
\Gamma_0\ 
D^W\left(\lambda-ic\frac{e^{i\xi X}-1}{iX},\frac{e^{i\xi X}-1}{iX}\right)
+O(|\lambda|^2\,(|\lambda|+|\xi|))
$$
where $D^W$ denotes the dispersion 
%relation
function
associated with system~\eqref{whit}.
\epr

\bc\label{calculable}
Coefficients $\alpha_0$ and $\alpha_1$ in \eqref{fact2}, thus also $\alpha$ in  \eqref{specexp}, are explicitly calculable\footnote{Daunting, though. See \eqref{alpha_1} for the expression of $\alpha_1$, which is by far the nicest one.}.
\ec

\bc\label{doubleroot}
$\lambda=0$ is exactly a double root of $\Delta(\cdot,0)$ if and only if System~\eqref{whit} is evolutionary.
\ec

\bc\label{groupvelocities}
Assume that $0$ is a double root of $\Delta(\cdot,0)$. Then the characteristics of \eqref{whit} are given by the values $c+ \alpha$ and $c$, with $\alpha$ as in \eqref{alpha}, or, after a change $x\to x-ct$ to comoving coordinates, the first-order coefficients $\alpha$ and $0$ of $\lambda_1$ and $\lambda_2$, respectively, in \eqref{specexp}.
\ec

\br
Taylor expanding $\frac{e^{i\xi X}-1}{iX},\frac{e^{i\xi X}-1}{iX}$ for small $\xi$, we obtain the simpler but less detailed formulation
$\Delta(\lambda,\xi) = \Gamma_0\ D^W( \lambda-ic \xi , \xi ) +O( (|\lambda|+|\xi|)^3) $ familiar from \cite{Se}, etc.
\er

\br Though valid, {\it the first-order Whitham modulation system does not directly yield instability information}, being hyperbolic except on the boundary $\alpha=0$, where, numerically, hyperbolicity is seen to fail due to presence of a nontrivial Jordan block. Nevertheless, recall that due to the special structure \eqref{specexp}, it turns out that the non-zero characteristic $\alpha$ for the first order modulation system gives also second order side-band and mod-$2$ subharmonic stability information, a fact not seen at the ``first-order" level of the modulation system \eqref{whit}. This is reminiscent of the case of viscous shock theory (see, e.g., \cite[\S 6]{Z3}), where failure of a first-order low-frequency stability condition (in this case an inviscid Lopatinsky determinant), marks a boundary for full viscous stability.
\er

\subsubsection*{High- and 
%KZ: ``medium,'' right?
%mean-frequency 
medium-frequency stability indices}\label{s:index}
We now present our two main analytical results, comprising rigorous high- and low-frequency stability conclusions. The first is directly based on 
%MR: I do not think so
%a WKB-type approximation 
%KZ: the below is exactly WKB, see LWZ1... but never mind, it seems distracting apparently.
an approximate high-frequency diagonalization 
as in \cite{N1}; see \cite{Z1,Z2,LWZ1} and \cite{BJRZ,BJNRZ2} for related analyses in respectively detonation and viscous roll wave stability.

\bt[High-frequency stability criteria]\label{hfthm}
For any roll wave $W$ of \eqref{sv_intro}:\\
(i) (potential) unstable spectra has bounded real part.\\
(ii) there is a {\it stability index} $\bf I$, explicitly calculable, such that when ${\bf I}<1$ (potential) unstable spectra is bounded, whereas when ${\bf I}>1$ there is an unbounded curve of unstable Floquet eigenvalues asymptotic to $\Re \lambda=\eta$ as $|\lambda|\to \infty$, for some $\eta>0$.
\et

The high-frequency stability index ${\bf I}$ is seen numerically to be identically $<1$, across the entire parameter-range of existence. Thus, we have the important conclusion, generalizing observations of \cite{N1,N2}, that
%{\it all roll waves are stable for sufficiently high frequencies of perturbation.}
{\it for any roll wave unstable modes have both bounded time growth rates and bounded time frequencies.} In short, we have excluded high-frequency instabilities, with the convention used above and from now on that high-, mid-, and low-frequency refer to the size of the spectral 
%KZ: this footnote is hard to parse: reword?
parameter\footnote{Note that we abuse terminology by calling $\lambda$ instead of $\Im(\lambda)$ a time frequency. Besides, note that the relevance of this distinction is in contrast with the fact that the size of spatial frequencies is not readily accessible in periodic spectral problems since in any given Floquet parameter are grouped together spatial frequencies of arbitrary large size.} $\lambda$.
Such a result is important in numerical stability investigations, truncating frequency space to a compact domain on which computations can be carried out with uniform error bounds, in a theoretically justified way.
See \cite{BJRZ,BJNRZ2} for similar bounds and their numerical use in the viscous case.
Theorem \ref{hfthm} is a corollary of the more detailed Proposition \ref{hfprop} below.

\br\label{loprmk} Intuitively, the upper bound on growth rates from Theorem \ref{hfthm}(i), associated with local well-posedness of the linearized equations, follows from the expectation that $\Delta$ converges as $\Re(\lambda)\to +\infty$ to a nonvanishing multiple of the shock Lopatinsky condition for the component shock at endpoints $jX$ of $W$, which is nonvanishing by Majda's theorem \cite{M} giving stability of shock waves for isentropic gas dynamics with any polytropic equation of state. That is, {\it well-posedness involves only the behavior near component shocks} \cite{N3}. See \cite{Z2} for a similar argument in the detonation case.
\er

The second result gives a rigorous justification of $\alpha=0$ as a potential stability boundary.

%KZ: changed order for readibility
\bt[Nonoscillatory co-periodic and subharmonic instabilities]\label{lfthm}
Let $W$ be an $X$-periodic roll wave of \eqref{sv_intro} and $\alpha_0$ and $\alpha_1$ be as in \eqref{fact2} with a normalization ensuring that $\Delta(\lambda,0)$ and $\Delta(\lambda,\pi/X)$ are real when $\lambda$ is real. Assume that $\alpha_0\neq0$ and let $\alpha$ be as in \eqref{specexp}.  Then:
\begin{enumerate}
\item[(i)] The sign of $\alpha_0=\d_\lambda^2\Delta(0,0)/2$ determines the parity of the number of real positive roots of $\Delta(\cdot,0)$.
\item[(ii)] $\alpha_1$ is purely imaginary and if $\alpha_1\neq0$ then the sign of $\alpha_1/i=\d_{\lambda\xi}^2\Delta(0,0)/i$ determines the parity of the number of real positive roots of $\Delta(\cdot,\pi/X)$.
\item[(iii)] In particular if $\alpha<0$ then $W$ is spectrally unstable.
\end{enumerate}
\et

The proof follows by a global stability index computation adapting those for continuous waves, see for instance \cite{Z3,BJRZ,BMR}, noting that with the above normalization $\hat \Delta(\lambda,0)$ and $\hat \Delta(\lambda,\pi/X)$ are real-valued for $\lambda$ real, determining their signs when $\lambda$ is real and large and invoking the intermediate value theorem. Numerical observations suggest that only subharmonic instabilities occur through a change in the sign of $\alpha$, consistently with the low-frequency analysis.

\subsubsection*{Numerical stability analysis}\label{s:numres}
We complete our investigations with a full numerical stability analysis across the entire parameter-range of existence. This is performed following the general approach of \cite{BJNRZ2} in the viscous case by numerical approximation of the periodic Evans-Lopatinsky determinant combined with various root finding and tracking procedures. Additional difficulties in the present case are induced by the sonic point at which the eigenvalue ODE becomes singular; this is handled by a hybrid scheme in which the solution is approximated by series expansion in a neighborhood of the sonic point, then continued by Runga-Kutta ODE solver forward and backward to boundaries $x=0,X$. The resulting inviscid Evans-Lopatinsky solver proves to be both numerically well-conditioned and fast. See Section~\ref{s:numerics} for further details.

\begin{figure}[t]
\includegraphics[scale=0.65]{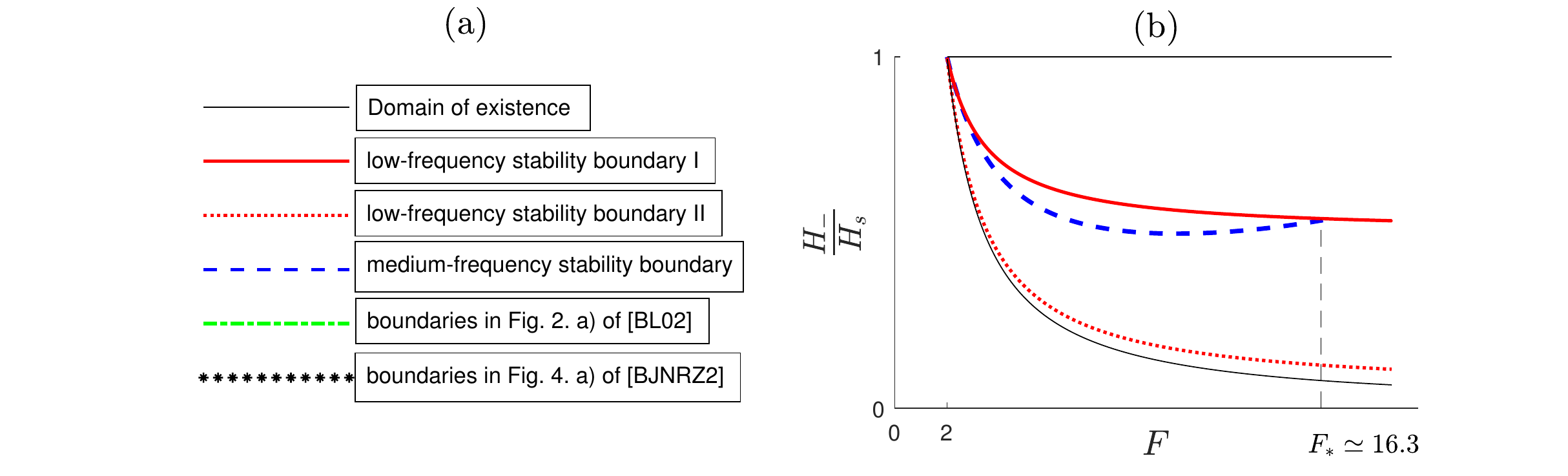}
\caption{Throughout our figures, we use the scheme in (a) for labeling the various stability boundaries. In (b), in an effort to illustrate fine details difficult to observe in the numerical results in Figure~\ref{fig_rosetta}, we present a cartoon of the stability boundaries in $H_-/H_s$ vs. $F$ coordinates.  In the actual stability boundaries in Figure~\ref{fig_rosetta}(a) below, the 
%KZ: more standard usage... is this what you mean here Miguel? (medium, or mid-frequency?)
%mean-frequency stability boundary is nearly indistinguishable from the lower existence boundary,
mid-frequency stability boundary is nearly indistinguishable from the lower existence boundary,
while the low-frequency boundary II seemingly asymptotes to this lower existence curve as $F$ increases, and so the detail in (b) above is invisible to the eye.
%MR
We stress in particular that large-amplitude transition to instability does not originate in 
%low-frequencies, 
%KZ: picky change.
low frequencies, 
	stability being encountered between low-frequency stability boundary I and 
%KZ: is this what you mean? (``mean'' is wrong word)
%mean-frequency boundary.
the mid-frequency boundary.
We note also that relevant stability boundaries meet at $H_-/H_s\approx 0.09201$ when $F=F_*$. }
\label{fig_scheme}
%Changed PN: visible-->invisible
\end{figure}

To interpret our numerical results, we introduce a general scheme for labeling the various stability boundaries in Figure~\ref{fig_scheme}(a). The outcome of our numerical investigations, displayed in Figure~\ref{fig_rosetta}, is a {\it complete and rather simple stability diagram for all inviscid roll wave solutions of the Saint Venant equations \eqref{sv_intro}}, an object of considerable interest in hydraulic engineering and related applications. In panel (a), given in coordinates $H_-/H_s$ versus $F$, the upper line $H_-/H_s=1$ corresponds to the small-amplitude limit while the lower curve corresponds to the large-amplitude (homoclinic) limit.
%Here, red denotes the stable region, blue denotes the low-frequency
%unstable region in which $\beta:=\alpha\gamma>0$, using notation from Corollary \ref{exp},
%and black the (almost indistinguishable) medium-frequency unstable region, in which there exist unstable roots far from
%the origin.
Here, the stability boundaries are labeled following the scheme in Figure~\ref{fig_scheme}(a); the region below the low-frequency boundary I, and above the mid-frequency boundary corresponds to spectrally stable roll waves; all other parameter values are spectrally unstable to either low or mid-frequency perturbations, as described in our forthcoming analysis. The numerically-determined low-frequency stability boundaries I and II agree well with the explicitly calculable boundaries
$\alpha=0$ and $\gamma=0$, a useful confirmation of numerical accuracy of the code.  In Figure~\ref{fig_scheme}(b), in an effort to emphasize the key, yet difficult to see, features in Figure~\ref{fig_rosetta}(a), we provide a cartoon version of the numerical results in Figure~\ref{fig_rosetta}(a): in Figure~\ref{fig_scheme}(b), we exaggerated the horizontal and vertical scales to emphasize the relative positions of the nearly indistinguishable stability and existence boundary curves in Figure~\ref{fig_rosetta}(a)). Panel~\ref{fig_rosetta}(b) depicts the same diagram with minimum wave height $H_-/H_s$ replaced by maximum wave-height $H_+/H_s$, addressing the question of maximum wave overflow mentioned earlier. Panel (c), given in terms of relative period $X/H_s$, addresses the ``water hammer'' issue, determining stable wavelengths $X$ and temporal frequencies $\omega= -c(H_s)(H_s/X)$. Similarly as in the viscous case \cite{BJNRZ2}, there is seen to be a transition at about $F\approx 2.75$ to a different asymptotic regime, in which the 
%KZ
%mean-frequency 
mid-frequency
stability curve has a different shape. This is displayed in the enlarged diagram of panel (d).

In Figure~\ref{fig_compare}(a), we compare the above stability results to the low-frequency stability predictions of model \eqref{av}, given in terms of average relative height $\overline{H}/H_s$ as in \cite{BL02}, where here $\overline{H}$ denotes the average of the height profile $H$ over a period. Here, the green curves are the stability boundaries of \cite{BL02} and the red and blue boundaries the stability boundaries of Figure~\ref{fig_rosetta}. We see clearly that \eqref{av}
%does not agree with the Whitham modulation equations \eqref{whit}, and hence
\emph{does not correctly predict low-frequency spectral behavior nor long-time dynamical stability of roll waves} of any kind. Indeed, it gives both false positive and false negative predictions of spectral stability, invalidating in a strong sense the model \eqref{av} proposed in \cite{BL02}. 
%This resolution demonstrates the enormous benefit of working with the exact eigenvalue equations \eqref{eval} that hinge on sound mathematical bases.
This resolution, along with the analytical and numerical validation of the inviscid Whitham system \eqref{whit} , demonstrates the enormous benefit of working with the exact eigenvalue equations \eqref{eval} derived from sound mathematical bases.

In Figure~\ref{fig_compare}(b), we compare to a viscous stability diagram obtained through intensive numerical computations in \cite{BJNRZ2}. Upper and lower inviscid stability boundaries are %depicted in red and blue, while viscous stability boundaries are depicted in black.
again depicted according to the scheme in Figure~\ref{fig_scheme}(a). Both inviscid boundaries are seen to lie near the lower viscous boundary, with the upper viscous boundary deviating substantially from the upper inviscid curve. A closeup view, given in Figure~\ref{fig_blowup}, shows that lower viscous and inviscid curves are in fact \emph{extremely} close, giving (since carried out by separate codes and techniques) confirmation of the numerical accuracy of both computations. An important consequence from the engineering point of view is that
%MR
for practical purposes the small-amplitude part of the explicitly calculable inviscid stability boundary $\alpha=0$ {\it appears to suffice as an excellent approximation of the small-amplitude transition to stability also in the viscous case.}
Note that the viscosity coefficient associated with Figures~\ref{fig_compare}(b)-\ref{fig_blowup} has the value $\nu=1$; that is, part of the inviscid stability diagram seems to persist beyond a ``small-viscosity'' approximation, for moderate values of $\nu$ as well.

\begin{figure}[htbp]
\begin{center}
\includegraphics[scale=0.6]{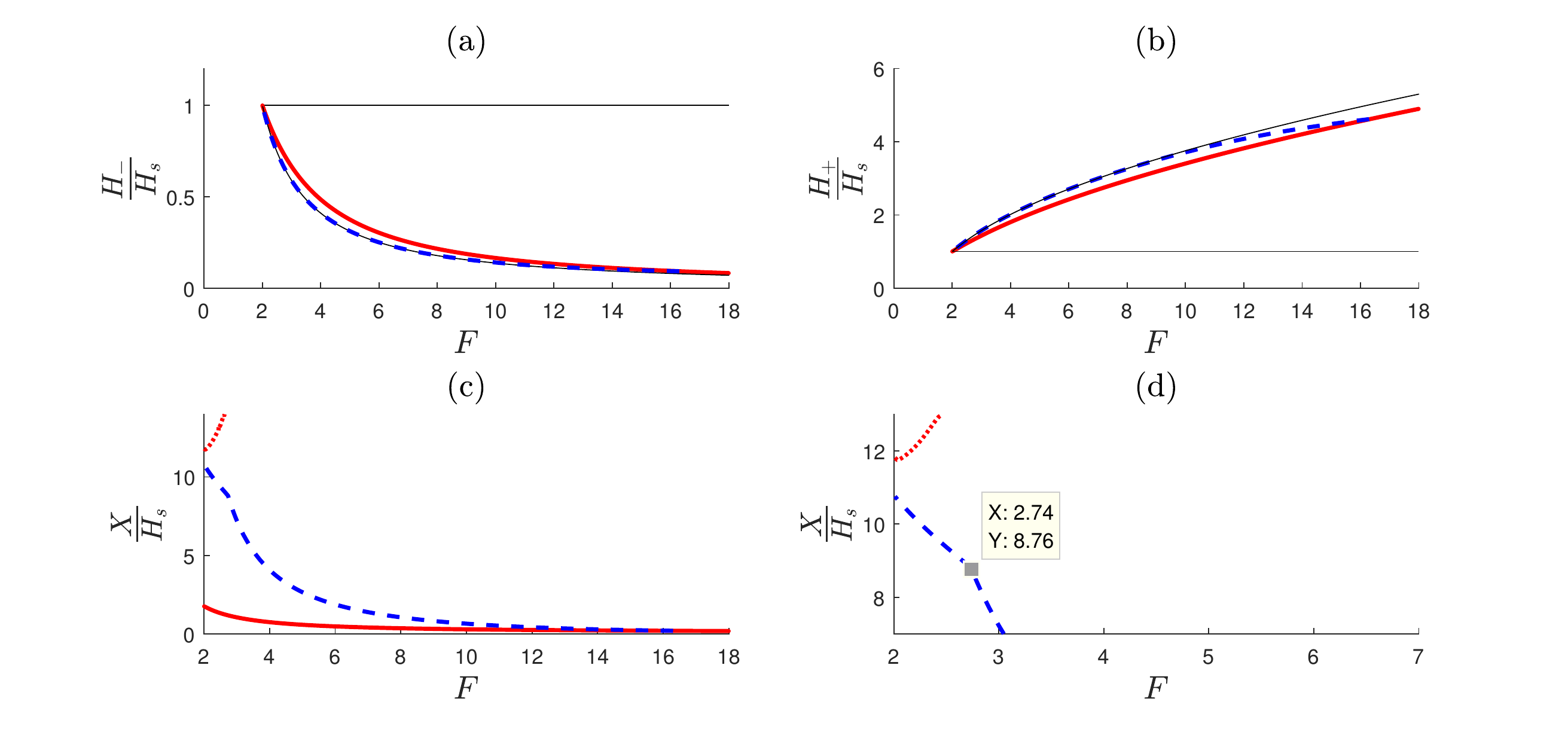}
\end{center}
\caption{Complete inviscid stability diagram: (a) $H_-/H_s$ v.s $F$. (b) $H_+/H_s$ vs. $F$. (c) $X/H_s$ vs. $F$.
(d) enlarged view of $X/H_s$ vs. $F$ near onset.}
\label{fig_rosetta}\end{figure}

\begin{figure}[htbp]
\begin{center}
\includegraphics[scale=0.6]{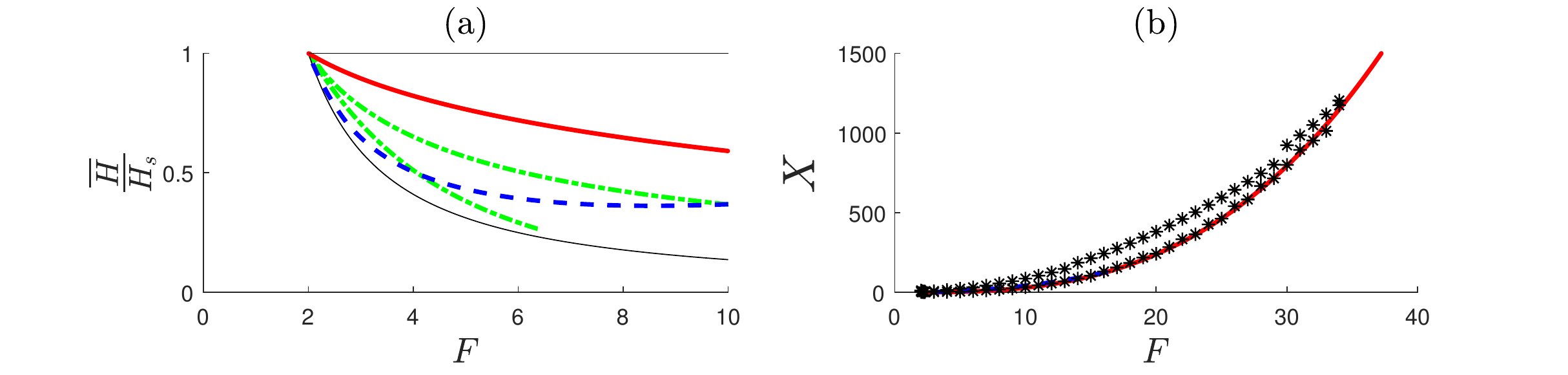}
\end{center}
\caption{Comparison of inviscid stability diagram (red and blue) with:
(a) modulational stability diagram of \cite{BL02} (green). (b) viscous stability diagram of \cite{BJNRZ2} (black).}
\label{fig_compare}\end{figure}

\begin{figure}[htbp]
\begin{center}
\includegraphics[scale=0.6]{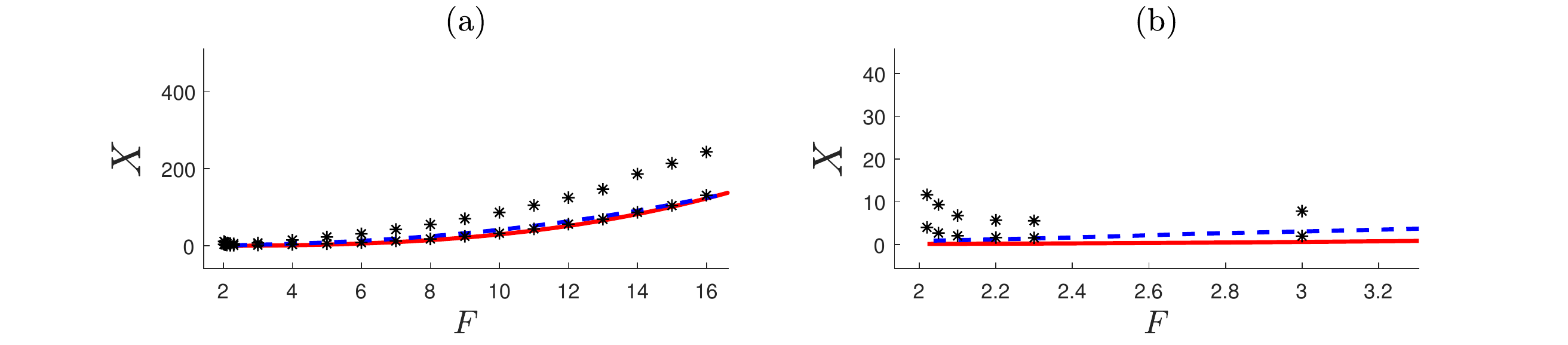}
\end{center}
\caption{Blowups of Figure~\ref{fig_compare}(b). (a) Close correspondence of lower curves away from
onset. (b) Different scaling of viscous vs. inviscid boundaries near onset.}
\label{fig_blowup}\end{figure}

\subsection{Discussion and open problems }\label{s:discussion}
We have obtained, similarly as in the viscous investigations of \cite{BJNRZ2},
%MR
%a surprisingly simple power-law description of
surprisingly simple-looking
curves bounding the region of spectral stability in parameter space from above and below, across which particular low- and intermediate-frequency stability transitions for inviscid roll waves occur. This stability region is bounded; in particular, all waves are unstable for $F\gtrapprox 16.3$. As also observed in the viscous case, there seems to be a transition between low-Froude number, or ``near-onset'' behavior for $2<F \lessapprox 2.5$ and high-Froude number behavior for $F\gtrapprox 2.5$, in the inviscid case occurring at $F\approx 2.74.$

In contrast to the viscous case, however,
%MR
in this inviscid case, the small-amplitude  transition is seen to agree with
the low-frequency stability boundary that is obtained {\it explicitly}, being given as the solution of a cubic equation in wave parameters. Moreover, numerical computations of this boundary are fast and well-conditioned, even for large $F$. Using scale-invariance of \eqref{sv_intro} these findings are compactly displayed in a single figure (Figure~\ref{fig_rosetta}), for four different choices of wave parametrizations which we hope convenient for hydraulic engineering applications.

Our results validate {\it but are not implied by} the associated formal Whitham modulation system \eqref{whit}; being obtained by direct spectral analysis via the periodic Evans-Lopatinsky determinant \eqref{det} rather than by formal WKB asymptotics. Furthermore, our results invalidate in a strong sense the alternative averaged model \eqref{av} as a predictor of stability/instability of inviscid roll waves, since it is shown to give false positive and false negative predictions of spectral stability. Part of our explicit low-frequency stability boundary appears to be accurate also in the viscous case, at least for a range of viscosity $0<\nu\leq 1$.

As noted throughout the presentation, there is a substantial analogy between roll wave stability and the more developed detonation theory, both in the structure of the equations \eqref{ab} and phenomena/mathematical issues involved. It is our hope that the periodic stability function \eqref{det} and numerical stability diagram introduced here will play a similar role for stability of roll waves as have Erpenbeck's stability function and systematic numerical investigation for detonations \cite{E}.

Our results suggest a number of directions for further investigation. For example, it should be possible to carry out rigorous asymptotics on the periodic Evans-Lopatinsky determinant in the $F\to 2$ regime where numerical computations become singular, complementing our current analysis and determining the validity of the various formal amplitude equations proposed near onset. Interestingly, the inviscid asymptotics appear to have a different scaling than the viscous ones; see Figure~\ref{fig_blowup}(b). A related problem is to derive the Whitham equations \eqref{whit} from first principles via
%MR
%rigorous
a systematic
multiscale expansion, and, continuing, to obtain a second-order expansion (similar to \cite{NR}), presumably recovering the second-order low-frequency stability condition obtained here via expansion of the Evans-Lopatinsky determinant.

A further very intriguing puzzle left by our analysis is the close correspondence of low-frequency boundaries in the inviscid and viscous case; Figure~\ref{fig_blowup}(a). This is reminiscent of the situation in the case of denotations, where it has been shown rigorously that low-frequency limits for inviscid and viscous models agree \cite{JLW}.\footnote{Curiously, this does not yield instability results in the detonation case, but only low-frequency {\it stability}.} Here, the corresponding object would appear to be the Whitham modulation equations, or low-frequency spectral expansion. However, these clearly do not agree, since numerical computations of \cite{BJNRZ2} show that characteristics of the viscous modulation systems do not vanish, whereas inviscid characteristics do. Moreover, the upper and lower stability curves clearly diverge near onset $F\to 2^+$, Figure~\ref{fig_blowup}(b), with viscous periods going to infinity while inviscid periods approach finite limits. Thus, in the present case, the approximate agreement of low-frequency stability boundaries appears to be %approximate and not exact, and
limited to the large-$F$ regime $F\gtrapprox 2.5$. Nonetheless, the correspondence seems of practical use in hydraulic engineering applications lying in this same regime, and as such this would be very interesting to shed some light on this coincidence.

More generally, the study of the singular zero-viscosity limit and viscosity-dependence of upper {\it and} lower stability boundaries appears to be the main outstanding problem in stability of roll waves. See \cite{Z2} for a corresponding study in the detonation case. The roll wave case is significantly complicated by the presence of sonic points for the inviscid profile, corresponding to loss of normal hyperbolicity in the singular limit; see the treatment of the related existence problem in \cite{H03}.

A natural further question is to what extent nonlinear stability is related to the spectral stability properties studied here. Here, we face the conundrum pointed out in \cite{JLW}, that the strongest nonlinear stability results proven to date for solutions containing shocks, are {\it short time} stability results in the sense of the original shock stability work of Majda \cite{M} (see also \cite{Majda_existence_multiD_shock-fronts,Majda_stability_multiD_shock-fronts,FrancheteauMetivier,BenzoniGavage-Serre_multiD_hyperbolic_PDEs}), yet these results can be obtained equally well assuming only stability of the component shocks, corresponding in terms of spectrum of the full wave to nonexistence of spectra with sufficiently large real part; see again Remark~\ref{loprmk} and reference \cite{N3}. To obtain a full nonlinear asymptotic stability result
%MR
%appears to
could
require an argument set not in Sobolev setting, but in a setting like BV accommodating formation of additional shocks, presumably involving a Glimm or shock-tracking scheme. This is a very interesting problem, but has not so far been carried out even in the simpler detonation setting. An alternative approach to nonlinear dynamics is to combine the rather complete nonlinear stability theory of the viscous case \cite{BJNRZ2} with the detailed spectral stability picture of the inviscid case carried out here, closing the logical loop by a comprehensive study of viscous spectra in the inviscid limit following \cite{Z1,Z2}.

Finally, we mention the very interesting recent work of Richard and Gavrilyuk \cite{RG1,RG2} introducing a refined version of \eqref{sv_intro} modeling additional vorticity effects. For roll waves, this takes the form of the full non isentropic ($3\times 3$) equations of gas dynamics plus source terms, and yields profiles matching experimental observations of \cite{C34,Br1,Br2} to an amazing degree, removing shock overshoot effects of the Dressler approximations. It would be very interesting to apply our methods toward the stability of these waves. Other natural directions for generalization are the study of multi-shock roll waves as mentioned in Remark~\ref{comprmk2} below, and the study of multidimensional stability incorporating transverse as well as longitudinal perturbations.

\medskip
{\bf Acknowledgement:} Thanks to Olivier Lafitte for stimulating discussions regarding
normal forms for singular ODE, and to Blake Barker for his generous help in sharing source computations from \cite{BJNRZ2}. The numerical computations in this paper were carried out in the MATLAB environment; analytical calculations were checked with the aid of MATLAB's symbolic processor. Thanks to Indiana Universities University Information Technology Services (UITS) division for providing the Karst supercomputer environment in which most of our computations were carried out. This research was supported in part by Lilly Endowment, Inc., through its support for the Indiana University Pervasive Technology Institute, and in part by the Indiana METACyt Initiative. The Indiana METACyt Initiative at IU was also supported in part by Lilly Endowment, Inc.

\section{Dressler's roll waves}\label{s:existence}

\subsection{Profile equations}\label{s:profiles}

We first review the derivation by Dressler \cite{D49} of periodic traveling-wave solutions of \eqref{sv_intro}. Let $(h,q)=(H,Q)(x-ct)$ denote a solution of \eqref{sv_intro} with $c$ constant and $(H,Q)$ piece-wise smooth and periodic with period $X$,
%MR
with discontinuities at $jX$, $j\in\ZM$.
In smooth regions, we have therefore
\be\label{smooth}
-cH'+(Q)'=0\,,\quad -cQ'+\left(\frac{Q^2}{H}+ \frac{H^2}{2F^2}\right)'=H-\frac{|Q|Q}{H^2}\,,
\ee
and across curves of discontinuity $(H,Q)$ are chosen to satisfy the Rankine-Hugoniot jump conditions
\be\label{jump}
-c[H]+[Q]=0\,,\quad -c[Q]+\left[\frac{Q^2}{H}+\frac{H^2}{2F^2}\right]=0,
\ee
augmented following standard hyperbolic theory \cite{La,Sm,Serre-conservation-I} with the Lax characteristic conditions
%MR: removed redundant pm
\be\label{lax}
a_1(X^-)<c,\ a_2(X^-)>c>a_2(X^+) \qquad\hbox{\rm or }\qquad
a_2(X^+)>c,\ a_1(X^-)>c>a_1(X^+) \, ,
\ee
where
%MR: extended to make explicit which one is which
$$
a_1=\frac{q}{h}-\sqrt{\frac{h}{F^2}}\,,\qquad
a_2=\frac{q}{h}+\sqrt{\frac{h}{F^2}}
$$
are the characteristics associated with \eqref{sv_intro}. Recall that the conservative part of \eqref{sv_intro} is the system of isentropic gas dynamics with velocity $u=q/h$ and pressure law $p(h)=h^2/2F^2$, thus above formulas coincide with $a_j=u\pm \sqrt{p'(h)}$.

Integrating the first equation of \eqref{smooth}
%MR: added, important to get the same constant everywhere !
jointly with the first equation of \eqref{jump}, we obtain
\be\label{UHrel}
Q-c H\equiv \const =:-q_0,
\ee
whence, substituting in the second equation of \eqref{smooth}, we obtain the scalar ODE
\be\label{altsmooth}
\left(\frac{-q_0^2}{H^2}+\frac{H}{F^2}\right)H'=H-\left|-q_0+cH\right|(-q_0+cH)/H^2
\ee
%or, equivalently,
%$$
%\left( -(c-Q/H)^2 + p'(H) \right)H' =H-\frac{\left|-q_0+cH\right|(-q_0+cH)}{H^2}\,,
%$
and, substituting in the second equation of \eqref{jump}, the scalar jump condition
\be\label{altjump}
\left[\frac{q_0^2}{H}+\frac{H^2}{2F^2}\right]=0.
\ee

%MR: simplified the argument
%From the Lax condition, [...]
From \eqref{altjump} we deduce that there is a special {\it sonic value} $H_s\in (H_-,H_+)$ such that
$$
\frac{-q_0^2}{H^2}+\frac{H}{F^2}=0
\qquad\textrm{when}\qquad
H=H_s\,,
$$
in particular, there, one of the characteristic speeds $a_j$ equals the wave speed $c$, hence the terminology. The latter argument uses in a fundamental way the scalar nature of the reduced profile equation~\eqref{smooth} but a similar conclusion may be obtained as a more robust consequence of the Lax condition, from which stems that at least one of the characteristic speeds $a_j$ must change position with respect to speed $c$ along the wave profile.
%MR: added discussion of sign and F>2
It follows that \eqref{altsmooth} is {\it singular} at the value $H_s$, from which we can draw a number of useful conclusions. First, we may check that there is indeed only one sonic value and that reciprocally we may solve the sonic equation to obtain, up to a sign indetermination, $q_0$ as a function of $H_s$ (and $F$); then, substituting this value in \eqref{altsmooth} evaluated at $H=H_s$, we obtain, again up to a sign indetermination, $c$ as a function of $H_s$ (and $F$) as well, leaving
$$
\frac{c}{H_s^{1/2}}=1\pm\frac1F\,,\qquad \frac{q_0}{H_s^{3/2}}=\pm\frac1F.
$$
At this stage, by monotonicity of solutions to \eqref{altsmooth}, one may notice that only the $+$ sign, corresponding to a $2$-shock, is compatible with the Lax condition \eqref{lax} and it requires $F>2$.

Combining information, assuming $F>2$, and setting $H_-:=H(0^+)$, $H_+:=H(X^-)$, we obtain the defining relations
\ba\label{pode}
H'= F^2\frac{H^2+(H_s-c^2)H+\frac{q_0^2}{H_s}}{H^2+H_s\,H+H_s^2}, \qquad H_-\leq H\leq H_+ ,\\
\ea
\be\label{rhrel}
\frac{q_0^2}{H_-}+\frac{H_-^2}{2F^2}\,=\,\frac{q_0^2}{H_+}+\frac{H_+^2}{2F^2},
\ee
\be\label{cq}
q_0=q_0(H_s)=\frac{H_s^{\frac32}}{F}\,,\qquad
c=c(H_s)= H_s^{\frac12}\Big(1+\frac{1}{F}\Big)\,,
\ee
where $'$ denotes $d/dx$. From the solution $H$ of \eqref{pode}, we may recover
$Q=-q_0+ cH$ using \eqref{UHrel}.
%MR: added
Note that, for any $0<H_-<H_s$, equation~\eqref{rhrel} defines a unique $H_+=H_+(H_-,H_s)>H_s$. Thus solvability reduces to the condition that there is no equilibrium of \eqref{pode} in $(H_-,H_+)$, which takes the form $H_->H_{hom}$ for some $H_{hom}(H_s)$. We make this latter condition explicit below. At last we observe that the shape of $H$ does not really depend on $H_-$ but is obtained as a piece of the maximal solution of \eqref{pode} passing through $H_s$.

\br\label{comprmk2}
Here, we have decided to consider only roll waves containing a single shock per period. By the analysis above, it is clear that we may construct multi-shock profiles consisting of arbitrarily many smooth pieces on intervals surrounding $H_s$, connected by shocks satisfying \eqref{altjump}. However these solutions do not persist as traveling waves under viscous perturbations \cite[\S 1.3.4]{N1}. A similar situation occurs in phase transitions models:  at the inviscid level one can form steady traveling patterns consisting of essentially arbitrary noninteracting (since traveling with common speed) under-compressive phase-transitional shocks switching from one phase to another. Turning on viscosity makes their ``tails'' interact, and so they do not persist as a noninteracting pattern. Numerical simulations \cite{AMPZ} show that these slowly interacting patterns can persist for a very long time, but eventually ``coarsen''  with waves overtaking and absorbing each other as happens for 
%KZ: changed this
%St Venant 
the St.~Venant equations
in some (unstable) cases \cite{BM}.
\er

\subsection{Scale-invariance}
%MR: added speeds and discharge rates
Following \cite{BL02}, we note the useful scale-invariance
\ba\label{scaleinv}
H(x)=H_s \underline{H}(x/H_s)\,,&
\qquad
X=H_s\underline{X}\,,
\qquad
c=H_s^{1/2}\underline{c}\,,\\
%\qquad
q_0=H_s^{3/2}\underline{q}_0\,,&
\qquad
Q(x)=H_s^{3/2} \underline{Q}(x/H_s)\,,
\ea
of \eqref{pode}, where $\underline{H}$ is the solution of \eqref{pode} with $H_s=1$, and correspondingly
$$
\underline{c}=1+\frac{1}{F}\,,
\qquad
\underline{q}_0=\frac1F\,,
$$
i.e.,
\be\label{rode}
\underline{H}' =
\Psi(\underline{H}):= \frac{F^2\,\underline{H}^2-(1+2F)\underline{H}+1}{\underline{H}^2+\underline{H}+1},
\qquad \underline{H}_- \leq \underline{H} \leq \underline{H}_+.
\ee
with
\be\label{zrhrel}
\frac{1}{\underline{H}_-}+\frac{\underline{H}_-^2}{2}\,=\,\frac{1}{\underline{H}_+}+\frac{\underline{H}_+^2}{2},
\qquad\textrm{\emph{i.e.}}\qquad
\underline{H}+= \mathcal{Z}_+(\underline{H}_-)=-\frac{\underline{H}_-}{2}
+\sqrt{\frac{\underline{H}_-^2}{4}+\frac{2}{\underline{H}_-}} \, .
\ee
This is quite helpful in simplifying computations; in particular, we see that all profiles are just rescaled pieces of a single solution $\underline{H}$ of the scalar ODE \eqref{rode}.
%MR: modified the formulation
Note that the the two real roots $\frac{1+2F\pm\sqrt{1+4F}}{2F^2}$ of the numerator of $\Psi$ are smaller than the sonic point $1$, so that the condition that $(\underline{H}_-,\underline{H}_+)$  avoid these stationary points of \eqref{rode} is
\be
\label{Hnrange}
\underline{H}_{hom}:=\frac{1+2F+\sqrt{1+4F}}{2F^2}<\underline{H}_-< 1.
\ee
We note in passing that the denominator of $\Psi$ never vanishes, being always positive.

\subsection{Wave numbers and averages}
%MR: reduced notation
%Parametrizing the family of Dressler waves by $(H_s, \zeta_-)$, and denoting averages over a single
%period $X$ by upper bars, we have
Denoting averages over a single periodic cell by upper bars, from equation~\eqref{rode} we have
\ba\label{zave}
\underline{X}=\ell(\underline{H}_-):=\int_{\underline{H}_-}^{\underline{H}_+} \frac{dh}{\Psi(h)}\,,&\qquad
\overline{\underline{H}}=\frac{1}{\ell(\underline{H}_-)}\int_{\underline{H}_-}^{\underline{H}_+}\frac{h\,dh}{\Psi(h)}\,,\qquad
\overline{\underline{Q}}=\underline{c}\,\overline{\underline{H}}-\underline{q}_0\,,\\
\overline{\frac{\underline{q}_0^2}{\underline{H}}+\frac{\underline{H}^2}{2F^2}}=\overline{\gamma}(\underline{H}_-)&:=\frac{1}{\ell(\underline{H}_-)}\int_{\underline{H}_-}^{\underline{H}_+} \frac{1}{F^2}\left(\frac{1}{h} + \frac{h^2}{2}\right) \frac{dh}{\Psi(h)},
\ea
with $\underline{H}_+=\mathcal{Z}_+(\underline{H}_-)$. As integrals of rational functions, all the above integrals may be computed explicitly. We record also formulas
\be\label{komega}
\underline{k}=\frac{1}{\underline{X}}\,,\qquad
\underline{\omega}= -\underline{c}\underline{k}\,,
\ee
for the (scaled) spatial and temporal wave numbers $\underline{k}$ and $\underline{\omega}$, also explicitly computable, and corresponding scaling rules
\be\label{scaleinvbis}
k=\frac{\underline{k}}{H_s}\,,\qquad
\omega=\frac{\underline{\omega}}{H_s^{1/2}}\,.
\ee

At last, note that at any $\underline{H}_{hom}<\underline{H}_-< 1$
$$
\ell'(\underline{H}_-)=\frac{1}{\Psi(\underline{H}_+)}\left(\frac{\underline{H}_+}{\underline{H}_-}\right)^2\frac{\underline{H}_+^3-1}{\underline{H}_-^3-1}-\frac{1}{\Psi(\underline{H}_-)}<0
$$
(as a sum of negative terms) so that one could alternatively parametrize wave profiles by $(H_s,X)$ or $(H_s,\underline{X})$ instead of $(H_s,H_-)$ or $(H_s,\underline{H}_-)$.

\section{Modulation systems}\label{s:modulation}
We next study modulation systems \eqref{whit} and \eqref{av}, using the computations of Section \ref{s:existence}.

\subsection{Dispersion relations and hyperbolicity}
Both of the systems \eqref{whit} and \eqref{av} are of the form
\begin{equation}\label{system}
\d_tG^0 + \d_x G^1=0,
\end{equation}
where $G^j=G^j(H_s, \underline{H}_-)$, of which the characteristics are the eigenvalues $\tilde \alpha_j$, $j=1,2$ of
$$
(A^0)^{-1}A^1\,, \qquad
A^j:=d_{(H_s,\underline{H}_-)}G^j,
$$
or, alternatively, coefficients of the dispersion relations $\lambda_j(\xi)=i\alpha_j \xi$ determined by
$$
\det\big(\lambda_j(\xi) A^0+ i\xi A^1\big)=0.
$$
The characteristics $\alpha_j$ are evidently invariant under nonsingular changes of parameters, corresponding to nonsingular changes of coordinates in the first-order system, with hyperbolicity corresponding to the $\alpha_j$ being real and semisimple\footnote{That is, the algebraic and geometric multiplicities of the real $\alpha_j$ agree, hence the eigenspaces contain no Jordan Blocks.}.
%MR:removed,  not clear to me
%By the scaling-invariance \eqref{scaleinv}, hyperbolicity depends only on $F$ and $\underline{H}_-=H_-/H_s$.
Below, we compute the characteristics $\alpha_j$ for both the Whitham system \eqref{whit} and the averaged system \eqref{av}.

\subsection{Whitham system}\label{s:Whitham}
We first compute the characteristics associated to the system \eqref{whit}.
Using \eqref{scaleinv}-\eqref{scaleinvbis}, the system \eqref{whit} may be written in the form
\be\label{Gs}
G^0=\bp H_s \overline{\underline{H}} \\ 1/(H_s\ell)\ep, \qquad
G^1=\bp H_s^{3/2}(\underline{c}\,\overline{\underline{H}}-\underline{q}_0)\\ \underline{c}/(H_s^{1/2}\ell)\ep
\ee
which, taking partial derivative with respect to $H_s,\underline{H}_-$, yields
\be\label{dGs}
A^0=\bp \overline{\underline{H}} & H_s \overline{\underline{H}}' \\ -1/(H_s^2\ell)& -\ell'/(H_s \ell^2)\ep, \qquad
A^1=\bp \frac32H_s^{1/2}(\underline{c}\,\overline{\underline{H}}-\underline{q}_0) &
H_s^{3/2}\underline{c}\,\overline{\underline{H}}'\\
-\frac12\underline{c}/(H_s^{3/2}\ell)&-\underline{c}\ell'/(H_s^{1/2}\ell^2)\ep .
\ee
Thus, so long as $A^0$ is invertible we find
$$
(A^0)^{-1}=
\frac{H_s\ell^2}{\ell\overline{\underline{H}}'-\overline{\underline{H}}\ell'}
\bp-\ell'/(H_s \ell^2)& -H_s \overline{\underline{H}}' \\
1/(H_s^2\ell)&\overline{\underline{H}}\ep
$$
hence
\begin{equation}\label{chars}
(A^0)^{-1}A^1= c\,\Id
+\frac{1}{\ell\overline{\underline{H}}'-\overline{\underline{H}}\ell'}
\bp H_s^{1/2}(\tfrac32\ell'\underline{q}_0-\tfrac12\underline{c}\,(\ell\overline{\underline{H}})')&0\\H_s^{-1/2}\ell(\underline{c}\,\overline{\underline{H}}-\tfrac32\underline{q}_0) & 0 \ep;
\end{equation}
here, we have verified numerically $\det A^0= \frac{\ell\overline{\underline{H}}'-\overline{\underline{H}}\ell'}{H_s\ell^2}\neq 0$. From \eqref{chars}, we have evidently that the characteristics of the Whitham system \eqref{whit} are
\be\label{whitchar}
\tilde \alpha_1=c\,,\qquad
\tilde\alpha_2=c+
H_s^{1/2}\frac{\tfrac32\ell'\underline{q}_0-\tfrac12\underline{c}\,(\ell\overline{\underline{H}})'}{\ell\overline{\underline{H}}'-\overline{\underline{H}}\ell'}.
\ee
As these are both real, we see the Whitham system \eqref{whit} is {\it strictly hyperbolic} whenever $\tilde \alpha_2\neq c $. On the boundary curve $\tilde\alpha_2=c$, it can be checked numerically that $\underline{c}\,\overline{\underline{H}}-\tfrac32\underline{q}_0\neq 0$, hence the system fails to be hyperbolic due to the presence of a non-trivial Jordan block; see Figure~\ref{fig_M21 BL}(a).

\subsection{Averaged system}\label{s:BLsys}
We next compute the characteristics of \eqref{av}. Using \eqref{scaleinv}-\eqref{scaleinvbis}, we may rewrite the system in the form \eqref{system} with
\be\label{BLGs}
G^0=\bp H_s \overline{\underline{H}} \\ H_s^{3/2}(\underline{c}\,\overline{\underline{H}} -\underline{q}_0)\ep,\qquad
G^1=\bp H_s^{3/2}(\underline{c}\,\overline{\underline{H}} -\underline{q}_0)\\
H_s^2(\underline{c}^2\overline{\underline{H}}-2\underline{c}\,\underline{q}_0+\bar\gamma)\ep,
\ee
yielding %taking partial derivative with respect to $H_s,\underline{\bar{H}}$ yielding in
\be\label{BLdGs}
A^0=\bp \overline{\underline{H}}&H_s\overline{\underline{H}}'\\
\frac32H_s^{1/2}(\underline{c}\,\overline{\underline{H}} -\underline{q}_0)&
H_s^{3/2}\,\underline{c}\,\overline{\underline{H}}'\ep, \quad
A^1=\bp \frac32H_s^{1/2}(\underline{c}\,\overline{\underline{H}} -\underline{q}_0)&
H_s^{3/2}\underline{c}\,\overline{\underline{H}}'\\
2H_s(\underline{c}^2\overline{\underline{H}}-2\underline{c}\,\underline{q}_0+\bar\gamma)
&H_s^2(\underline{c}^2\overline{\underline{H}}'+\bar\gamma')\ep .
\ee
Thus, so long as $A^0$ is invertible we have
$$
(A^0)^{-1}=
\frac{1}{(\tfrac32\underline{q}_0-\tfrac12\underline{c}\,\overline{\underline{H}})\overline{\underline{H}}'}
\bp \underline{c}\,\overline{\underline{H}}'&-H_s^{-1/2}\overline{\underline{H}}'\\
-\frac32H_s^{-1}(\underline{c}\,\overline{\underline{H}} -\underline{q}_0)&
H_s^{-3/2}\overline{\underline{H}}\ep
$$
hence
\ba
\mathcal{A}:=(A^0)^{-1}A^1&=c\,\Id +
\frac{1}{(\tfrac32\underline{q}_0-\tfrac12\underline{c}\,\overline{\underline{H}})\overline{\underline{H}}'}
\bp
H_s^{1/2}\overline{\underline{H}}'(\underline{c}\,\underline{q}_0-2\bar\gamma)&
-{H_s}^{3/2}\bar\gamma'\overline{\underline{H}}'
\\
H_s^{-1/2}(2\bar\gamma\overline{\underline{H}}-\tfrac14\underline{c}^2\overline{\underline{H}}^2+\tfrac12\underline{c}\,\underline{q}_0\,\overline{\underline{H}}-\tfrac94\underline{q}_0^2)&
H_s^{1/2}\overline{\underline{H}}\bar\gamma'
\ep.
\ea
Next, solving for curves corresponding to vanishing of the discriminant of the quadratic polynomial $\det(\mathcal{A}-\lambda\Id)=0$, i.e.,
$({\rm Trace}(\mathcal{A}-c\,\Id))^2= 4\det(\mathcal{A}-c\,\Id)$, we get the equation
$$
(\overline{\underline{H}}'(\underline{c}\,\underline{q}_0-2\bar\gamma)
+\overline{\underline{H}}\bar\gamma')^2
=-\bar\gamma'\overline{\underline{H}}'\left(\underline{c}\,\overline{\underline{H}}-3\underline{q}_0\right)^2\,,
$$
or, more explicitly,
\be\label{hyp}
(\overline{\underline{H}}'(F+1-2F^2\bar\gamma)
+\overline{\underline{H}}F^2\bar\gamma')^2
+F^2\bar\gamma'\overline{\underline{H}}'\left((F+1)\,\overline{\underline{H}}-3\right)^2=0
\ee
for the boundaries of the region of hyperbolicity. Tracing the roots of \eqref{hyp}, we get, up to numerical error, the same boundaries reported in \cite[Fig.2.a)]{BL02} see Figure~\ref{fig_M21 BL} (b).

We point out, in particular, that the boundaries for the regions of hyperbolicity associated  to \eqref{whit} and \eqref{av} are \emph{different}.  In order to determine which, if either, give accurate information regarding the local dynamics about a roll wave solution of \eqref{sv_intro}, we next perform a mathematically rigorous investigation of the spectral stability of roll wave solutions of \eqref{sv_intro}.
\begin{figure}[htbp]
\begin{center}
\includegraphics[scale=.6]{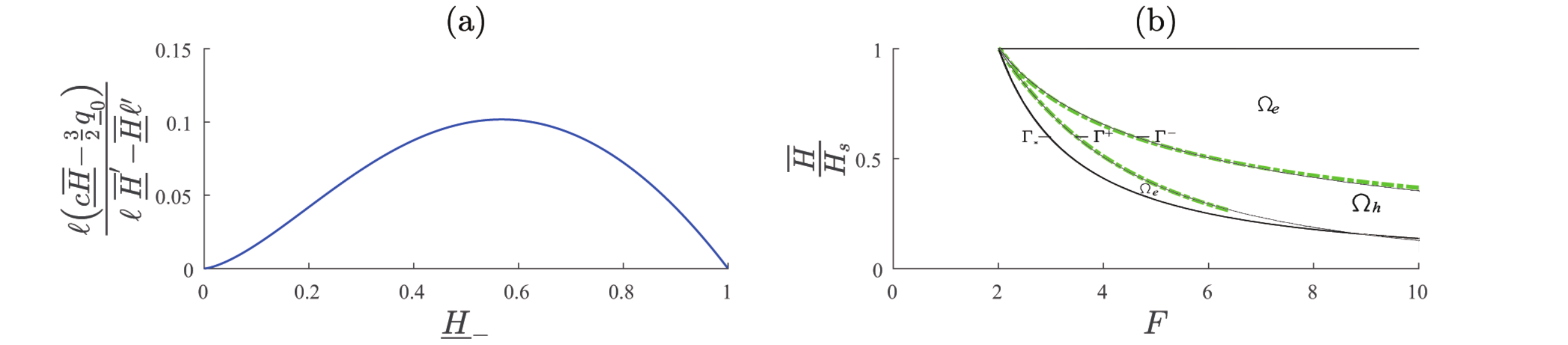}
\end{center}
\caption{(a)Value of
$\ell(\underline{c}\underline{\bar{H}}-\tfrac32\underline{q}_0)/(\ell\underline{\bar{H}}'-\underline{\bar{H}}\ell')$ on the strictly hyperbolic boundary $\tilde\alpha_2=c$. (b) Hyperbolic boundaries from \eqref{hyp} superimposed on corresponding figure from \cite{BL02}.
In (b), $\Omega_h$ and $\Omega_e$ denote domains of hyperbolicity and ellipticity, $\Gamma_\pm$ (thin grey line) the boundaries of $\Omega_h$, and $\Gamma_*$ the boundary of existence for roll wave solutions of \eqref{sv_intro}; the curves with thicker dark line (green in color plates) were computed using \eqref{hyp}. Labels $\bar{\zeta}$ and $Fr$ in \cite{BL02} correspond in our notation to $\overline{\underline{H}}$ and $F$.}
\label{fig_M21 BL}\end{figure}

%MR: rigorous->general
\section{General spectral stability framework}\label{s:evans-lop}
%MR
We now turn to the exact spectral stability problem, replacing the formal development of Section~\ref{s:lop} with a %rigorous treatment.
treatment as rigorous as possible.
Our goal is to connect as closely as possible our spectral framework with a notion of linear stability relevant also at nonlinear level. Unfortunately we cannot rely on any general nonlinear stability framework since none is known for any class of discontinuous waves of hyperbolic systems. Instead we shall argue by comparison with, on one hand, local well-posedness theory near single-shock waves pioneered by Majda \cite{M,Majda_existence_multiD_shock-fronts,Majda_stability_multiD_shock-fronts,FrancheteauMetivier,BenzoniGavage-Serre_multiD_hyperbolic_PDEs}, in particular \cite{N3} devoted to short-time persistence of roll waves, and, on the other hand, nonlinear stability of continuous periodic waves \cite{JNRZ13}.

\subsection{Linear space-modulated stability}\label{s:linearizing}

As in \eqref{ab}--\eqref{rh}, consider a general system of balance laws $\d_tw +\d_x(F(w))=R(w)$, $w\in \R^n$, $w$ piecewise smooth, with jump conditions $[F(w)]_j=x_j'(t)[w]_j$ at discontinuities $x_j$, and a traveling roll wave solution $W$ with shocks at $X_j+ct=jX+ct$. We complement those jump conditions with Lax characteristic conditions but assume, as in Section~\ref{s:profiles}, that they are satisfied in a strict sense by $W$ thus they will not appear at the linearized level.

From the analysis of continuous periodic cases, the best stability that we expect to hold in general is what was coined as \emph{space-modulated stability} in \cite{JNRZ13}. This corresponds to showing that starting close from $W$ a solution $w$ will remain close in the sense that for some $(\tilde w,\psi)$
$$
w(x-ct-\psi(x,t),t)=W(x)+\tilde w(x,t)
$$
with $(\tilde w,\d_x\psi,\d_t\psi)(\cdot,t)$ small in suitable norms. See related detailed discussions in \cite{JNRZ13,R_HDR,R_Roscoff,R_linearKdV}. Given the regularity structure of $W$ it is natural to measure the smallness of $\tilde w(\cdot,t)$ in $H^s(\tRR)$, with $s\geq0$ and
$$\tRR\,:=\,\bigcup_{j\in \Z} (jX, (j+1)X)\,.$$
Note that when $s>1/2$ this implicitly requires that $\psi(\cdot,t)$ fixes discontinuities $(x_j(t))_{j\in\Z}$ of $w(\cdot,t)$ through
\be\label{e:discont}
x_j(t)\,=\,jX-ct-\psi(jX,t)\,,\qquad j\in\Z\,.\ee
Observe that whereas for continuous waves the role of the resynchronization by $\psi$ is to allow a capture of long-time preservation of shape beyond divergence of positions, that is, to ensure that the norm of $\tilde w(\cdot,t)$ remains small, for discontinuous waves it is already necessary to ensure that it remains finite in finite time. In particular, a synchronization ensuring \eqref{e:discont} is also needed in definitions and proofs of local-in-time well-posedness in piecewise smooth settings \cite{N3}. At last note that at time $t$ smallness of $\d_x\psi(\cdot,t)$ encodes both that $\Id_\R-\psi(\cdot,t)$ is a diffeomorphism and that the distance between consecutive shocks remain bounded away from zero, hence they do not interact directly.

In the new coordinates, the perturbation of shape and phase shifts $(\tilde{w},\psi)$ evolves according to
$$
\d_t\tilde w+\d_t\psi\,W'+\d_x(A\,\tilde w)-E\,\tilde w+\d_x\psi\,R(W)
\,=\,\mathcal{N}_1(\tilde w,\d_t\tilde w,\d_x\tilde w,\d_t\psi,\d_x\psi)\qquad\textrm{on }\tRR
$$
and, for any $j\in\ZM$,
$$
\d_t\psi(jX,t)\,[W]+[A\tilde w]_j\,=\,\mathcal{N}_2(\tilde w,\d_t\tilde w,\d_x\tilde w,\d_t\psi,\d_x\psi)
$$
where, here, $[h]_j:= h((jX)^+)-h((jX)^-)$, $\mathcal{N}_1$, $\mathcal{N}_2$ are at least quadratic in their arguments and, as in Section~\ref{s:lop},
$$
A(x):=d_wF(W(x))-c\Id
\qquad\textrm{and}\qquad
E(x):=d_wR(W(x))\,.$$
To get closer to equations \eqref{abint}-\eqref{abjump}, we now introduce
$$
y_j(t):=x_j(t)+ct-jX=-\psi(jX,t)
\qquad\textrm{and}\qquad
v:=\tilde w+\psi W'
$$
and write the above equations equivalently as
$$
\d_tv+\d_x(A\,v)-E\,v
\,=\,\mathcal{N}_1(v-\psi W',\d_t(v-\psi W'),\d_x(v-\psi W'),\d_t\psi,\d_x\psi)\qquad\textrm{on }\tRR
$$
and, for any $j\in\ZM$,
$$
y_j'(t)\,[W]+y_j(t)[AW']-[Av]_j\,=\,-\mathcal{N}_2(v-\psi W',\d_t(v-\psi W'),\d_x(v-\psi W'),\d_t\psi,\d_x\psi)\,.
$$
We now drop nonlinear terms to focus on linear stability issues. 
%Note however that somehow the process leads to a derivative loss and that up to now this issue has been bypassed only as far as short-time local well-posedned is concerned.
Note however that our relsovent estimates will not gain derivatives so that the presence of derivatives in nonlinear terms will necessarialy 
induce a derivative loss in a nonlinear scheme proving stability and relying directly on linearized estimates. 
Recall that, up to now, this issue has been bypassed only as far as short-time local well-posedned is concerned;
 again see \cite{N3} and \cite{BenzoniGavage-Serre_multiD_hyperbolic_PDEs} and references therein.

In view of the foregoing discussion the natural linear stability problem consists in 
%CHANGED-MJ
%proving that solutions to
considering the bounded (continuous) solvability of the system
\be\label{e:lin-evolve}
\d_tv+\d_x(A\,v)-E\,v\,=\,f\quad\textrm{ on }\tRR
\qquad\textrm{and}\qquad
\textrm{for any }j\,,\quad
y_j'\,[W]-y_j[AW']+[Av]_j\,=\,g_j
\ee
%CHANGED-MJ
for functions $f$ and sequences $(g_j)$ belonging to an appropriate space, i.e. determining if \eqref{e:lin-evolve} has solutions such that
%are such that 
there exists a $\psi$ such that, for any $j$, $\psi(jX,t)=-y_j(t)$, and $(v-\psi W',\d_x\psi,\d_t\psi)$ may be bounded in terms of $(f,(g_j)_j,v(\cdot,0),(y_j(0)))$. The question of rigorously elucidating how this is connected to spectral properties considered here would lead us too far and we leave it for further investigation. See \cite{JNRZ13,R_linearKdV} for examples of such considerations for different classes of equations. We stress however that based on analyses of continuous cases the kind of time growth expected to arise from neglecting the dynamical role of $\psi$ - in particular trying to bound $\psi$ rather than $\d_x\psi$ - is algebraic, thus may be safely omitted when focusing, as we shall do now, on precluding exponential growths. Regarding local well-posedness, at the linearized level one only needs to exclude growths faster than exponential and $\psi$ may be chosen independently of dynamical considerations, in an essentially arbitrary way from the discontinuity positions, for instance cell-wise affine such that $\psi(jX,t)=-y_j(t)$ as in \cite{N1,N2,N3}.

For our purpose it is actually sufficient to consider solvability of \eqref{e:lin-evolve} in $H^1(\tRR)\times \ell^2(\ZM)$. As is well-known, see for instance \cite[\S 3.1.1]{BenzoniGavage-Serre_multiD_hyperbolic_PDEs}, thanks to the equation traces are still defined in lower-regularity settings, however we use $H^1(\tRR)$ setting for another purpose here: to discard algebraic or logarithmic singularities that arise from sonic points in $\tRR$. Once those are cast away, one may actually transfer conclusions from $H^1(\tRR)$ to $H^s(\tRR)$, for any $s\geq1$. To motivate the $\ell^2(\Z)$ framework for $(y_j)_j$, we add one more comment concerning its relation with $\psi$. Without loss of generality one expects to be able to enforce that the phase shift $\psi$ is low-frequency and centered - see \cite[Section~3.1]{R_linearKdV} - so that its high-regularity norms are controlled by its lower ones and
$$
\psi(x,t)\,=\,\int_{-\pi/X}^{\pi/X} e^{i\xi x}\widehat \psi (\xi, t) d\xi
$$
where $\widehat{\cdot}$ denotes the Fourier transform in the $x$-variable. Besides, under these conditions, from Parseval identities for the Fourier transform and Fourier series
$$
\|\psi(\cdot,t)\|_{L^2(\RM)}=\sqrt{X}\|(y_j(t))_j\|_{\ell^2(\ZM)}\,,\qquad
\|\d_t\psi(\cdot,t)\|_{L^2(\RM)}=\sqrt{X}\|(y_j'(t))_j\|_{\ell^2(\ZM)}\,.
$$

\subsection{Structure of the spectral problem}\label{s:lineartospectral}

Focusing on solutions to~\eqref{e:lin-evolve} that grow at most linearly in time naturally lead to the consideration of Laplace transforms in time on time frequencies $\lambda$ with $\Re(\lambda)>0$. This transforms \eqref{e:lin-evolve} into
\be\label{e:lin-line}
\lambda v+\d_x(A\,v)-E\,v\,=\,f\quad\textrm{ on }\tRR
\qquad\textrm{and}\qquad
\textrm{for any }j\,,\quad
\lambda y_j\,[W]-y_j[AW']+[Av]_j\,=\,g_j
\ee
with notational changes that the new $(v,(y_j)_j)$ is the Laplace transform of the old one at frequency $\lambda$ and that the new $(f,(g_j)_j)$ mixes Laplace transforms at $\lambda$ of old ones and initial data for former $(v,(y_j)_j)$.

Observing that \eqref{e:lin-line} is periodic-coefficient in space, it is natural to introduce the Bloch-wave representation of $v$ (and $f$) and to interpret $y=(y_j)_j$ (and $(g_j)_j$) as Fourier series of $(2\pi/X)$-periodic functions\footnote{Note that in terms of the above low-frequency
%CHANGED-MJ
assumption on
$\psi$, $\check y(\xi)=-\widehat{\psi}(\xi)$ when $\xi\in[-\pi/X,\pi/X]$.}
\ba\label{blochrep}
v(x)&=\int_{-\pi/X}^{\pi/X} e^{i\xi x}\check v(\xi,x) d\xi,\qquad\qquad
y_j=\int_{-\pi/X}^{\pi/X} e^{i\xi jX}\check y(\xi) d\xi,
\ea
where each $\check v(\xi,\cdot)$ is $X$-periodic. For sufficiently smooth $v$ and sufficiently localized $(y_j)_j$, the former transforms are defined pointwise by
$$
\check{v}(\xi,x):=\sum_{k\in\ZM}e^{i\frac{2k\pi}{X}x}\,\widehat{v}\big(\tfrac{2\pi k}{X}+\xi\big)=\sum_{k\in\ZM}e^{-i\xi(x+k)}v(x+kX),
\qquad
\check{y}(\xi):=\sum_{j\in\ZM}e^{-ijX\,\xi}\,y_j\,.
$$
General definitions follow by a density argument in $L^2$ (respectively, $\ell^2$) based on Parseval identities
$$
\|\check{v}\|_{L^2(-\pi/X,\pi/X;L^2(0,X))}
=\tfrac{1}{\sqrt{2\pi}}\|v\|_{L^2(\R)}
\,,
\qquad
\|\check{y}\|_{L^2(-\pi/X,\pi/X)}=\sqrt{\frac{X}{2\pi}}\|(y_j(t))_j\|_{\ell^2(\ZM)}\,.
$$
In particular the Bloch transform identifies $L^2(\R)$ with $L^2(-\pi/X,\pi/X;L^2(0,X))$,  and this identification may be extended to $H^s(\tRR)$ with $L^2(-\pi/X,\pi/X;H^s(0,X))$ by
observing
$$
\|(\d_x+i\xi)^k\check{v}\|_{L^2(-\pi/X,\pi/X;L^2(0,X))}
=\tfrac{1}{\sqrt{2\pi}}\|\d_x^kv\|_{L^2(\R)}\,,\qquad k\in\NM\,.
$$
For comparison, note the key distinction that $H^s(\R)$ is identified with $L^2(-\pi/X,\pi/X;H_{\rm per}^s(0,X))$ where $H_{\rm per}^s(0,X)$ is the $H^s(0,X)$-closure of smooth $X$-periodic functions on $\R$, hence is a set of $H^s(0,X)$ functions satisfying suitable periodic boundary conditions as soon as $s>1/2$.

Applying the above transformations to \eqref{e:lin-line} diagonalizes it into single-cell problems parametrized by the Floquet exponent $\xi$, namely
\be\label{e:lin-cell}
\lambda v+(\d_x+i\xi)(A\,v)-E\,v\,=\,f\quad\textrm{ on }(0,X)
\qquad\textrm{and}\qquad
y(\lambda\,[W]-[AW'])+[Av]\,=\,g
\ee
where we have dropped $\check{ }$ signs for the sake of readability. This clean characterization in \eqref{e:lin-cell} originates in \cite{N1,N2}, without discussion of underlying integral transforms. Here, we are making the new observation of ``completeness'' of the representation \eqref{blochrep}, providing a rigorous basis to the normal form analysis. 
Under the substitution $w(x)=e^{i\xi x}v(x)$, $\chi=e^{i\xi X}y$ the system \eqref{e:lin-cell} is recognized as an inhomogeneous version of \eqref{eval} from Section~\ref{s:lop}.

%CHANGED-MJ following MR's suggestion of splitting into definition and assumption.
%At this point, we need some knowledge of the structure of the interior ODEs in \eqref{e:lin-cell}. 
%To this end, we make the following assumption, verified for \eqref{sv_intro} in Appendix~\ref{s:local}, analogous to \emph{consistent splitting} in standard Evans function theory \cite{AGJ}.
At this point, we need some knowledge of the structure of the interior ODE's in \eqref{e:lin-cell}.
We begin with the following definition, analogous to \emph{consistent splitting} in standard Evans function theory \cite{AGJ}.

\begin{definition}[Local $H^1$ solvability]\label{def:local}
We say that at $\lambda\in\CM$ local $H^1$ solvability holds it there exists a constant $C$ such that for any $f\in H^1(0,X)$, the interior equations
$$
\lambda w+\d_x(A\,w)-E\,w\,=\,f
$$
have an $(n-1)$-dimensional affine space of $H^1(0,X)$ solutions whose minimum $H^1(0,X)$ norm is bounded by $C\|f\|_{H^1(0,X)}$.
\end{definition}

We will call the set $\Lambda$ of $\lambda\in\CM$ that satisfies Definition~\ref{def:local} the \emph{domain of $H^1$ local solvability}, echoing the classical Evans function terminology.  
%MR: no need to exclude 0
%\begin{assumption}[Local $H^1$ solvability]\label{local}
%For each $\lambda\in\CM$ such that $\Re \lambda \geq 0$, there exists a constant $C$ such that for any $f\in H^1(0,X)$, the interior equations 
%$$
%\lambda w+\d_x(A\,w)-E\,w\,=\,f
%$$
%have an $(n-1)$-dimensional affine space of $H^1(0,X)$ solutions whose minimum $H^1(0,X)$ norm is bounded by $C\|f\|_{H^1(0,X)}$.
%\end{assumption}
%\begin{assumption}[Local solvability]\label{local}
%For each $\lambda$ in a set $\Lambda\subset\CM$ containing $\{\Re \lambda \geq 0\}\setminus \{0\}$, there exists a constant $C$ such that for each $\xi\in[-\pi/X,\pi/X]$ and any $f\in H^1(0,X)$, the interior equations \eqref{e:lin-cell}(i) have an $(n-1)$-dimensional set of $H^1(0,X)$ solutions whose minimum $H^1(0,X)$ norm is bounded by $C\|f\|_{H^1(0,X)}$.
%\end{assumption}
%We will call the set $\Lambda$ of $\lambda\in\CM$ where Assumption \ref{local} holds the \emph{domain of $H^1$ local solvability}, echoing the classical Evans function terminology. 
It follows that for $\lambda\in\Lambda$, bounded invertibility of \eqref{e:lin-cell} depends continuously on $\xi\in[-\pi/X,\pi/X]$. 
Combined with the isometric properties of the integral transforms discussed above, this shows that bounded invertibility of \eqref{e:lin-line} is equivalent 
to the problem \eqref{e:lin-cell} being boundedly invertible for each $\xi\in [-\pi/X,\pi/X]$, justifying the  above normal form reduction; see for instance \cite[p.30-31]{R_HDR}. It follows that
for $\lambda\in\Lambda$ bounded invertibility of \eqref{e:lin-cell} is equivalent to an $n$-dimensional square matrix problem, encoded by  $\Delta(\lambda,\xi)\neq 0$ with $\Delta$ as in \eqref{det}, hence in particular bounded invertibility is equivalent to injectivity. Further, on $\Lambda$ where $D(\lambda,\xi)\neq0$ one may define resolvent-like operators as products of $D(\lambda,\xi)^{-1}$ and functions that are analytic in $\lambda$ on $\Lambda$ so that one may check that those resolvent-like operators have poles exactly where $D(\cdot,\xi)$ vanishes and that multiplicities as poles of resolvent-like operators agree with multiplicities as roots of $D(\cdot,\xi)$.

With the above in mind, we now make the following assumption, verified for \eqref{sv_intro} in Appendix~\ref{s:local}, regarding the structure of the set $\Lambda$ for the general
system \eqref{e:lin-cell}.

\begin{assumption}[Structure of $\Lambda$]\label{local}
At any $\lambda\in\CM$ such that $\Re \lambda \geq 0$ local $H^1$ solvability holds.  That is,
\[
\left\{\lambda\in\CM:\Re(\lambda)\geq 0\right\}\subset\Lambda.
\]
\end{assumption}

This justifies the Definition~\ref{specdef} of spectral instability given in Section~\ref{s:lop} as equivalent to the fact that for some $\lambda\in\CM$ with $\Re(\lambda)>0$ 
the problem \eqref{e:lin-line} is not boundedly invertible in $H^1(\tRR)\times\ell^2(\ZM)$. For the Saint Venant equations \eqref{sv_intro}, the domain of local $H^1$ solvability is shown in Appendix~\ref{s:local} to satisfy
%MR: changed, incorrect scaling in H_s
$$
\left\{\lambda\,:\,\Re(\lambda)>-\frac{F-2}{4\sqrt{H_s}}\right\}\subset \Lambda,
$$
evidently verifying Assumption~\ref{local}, hence Definition~\ref{specdef}, for the St. Venant equations.  Moreover, in Appendix \ref{s:local} we also demonstrate that the complex right-half plane is 
included in domain of local $H^s$ solvability for \eqref{sv_intro} for any $s\geq1$.

%MR: changed to yet another version. Check that it is OK.
%NOTE: here deviate a bit from Miguel's treatment, in which initial pert's in x_j allowed... fine for resolvent equation
%but muddies a bit the semigp framework, also apparent compatibility issues for data.  possibly could construct
%semigp just from resolvent equation a la \cite{Sa}... but compatibility issue stops me here... better this way?

Observe that we use local $H^1$ solvability at $\lambda\in\CM$ only to ensure that vanishing of the Evans-Lopatinsky determinant $\Delta(\lambda,\xi)$ for some $\xi$ is the only way in which the 
bounded solvability of the associated resolvent-like $H^1\times\ell^2$ problem may fail. To study and compute $\Delta$ for general systems \eqref{e:lin-cell}
%KZ, reworded slightly:
%the following assumption, that seems to be weaker, is sufficient.
the following, seemingly weaker, assumption is sufficient.

\begin{assumption}[Homogeneous local analytic solvability]\label{extension}
There exists an open connected set $\Lambda_0\subset\CM$ containing $\{\lambda\,:\,\Re(\lambda)\geq0\}$ on which one may choose a basis $w_1, \dots, w_{n-1}$ of the space of analytic solutions to
$$
\lambda w+\d_x(A\,w)-E\,w\,=\,0\,,
$$
that depends analytically on $\lambda$.
\end{assumption}

In Appendix~\ref{s:local} we check that for the special case of the Saint Venant equations \eqref{sv_intro}, one may choose
$$
\Lambda_0\ =\ \left\{\lambda\,:\,\Re(\lambda)>-\frac{F-2}{2\sqrt{H_s}}\right\}\,,
$$
hence verifying Assumption \ref{extension} in that case.

\subsection{Specialization to Saint-Venant equations}\label{s:special}

Preparatory to our further investigations, we now specialize the abstract theory of Sections \ref{s:linearizing}-\ref{s:lineartospectral} above to the case of the Saint Venant equations. System \eqref{sv_intro} can be put in form \eqref{ab} with $n=2$ as
$$
w=\begin{pmatrix}h\\q\end{pmatrix},\qquad
F(w)=\begin{pmatrix}q\\p(h)+\frac{q^2}{h}\end{pmatrix},\qquad
R(w)=\begin{pmatrix}0\\r(w)\end{pmatrix}
$$
where
$$
p(h)=\frac{h^2}{2F^2}\,,\qquad\qquad r(w)=h-\frac{q^2}{h^2}\,.
$$
Linearizing about $W:=(H,Q)^T$ in a co-moving frame and taking a Laplace transform in time, we obtain interior equations
\be\label{linearizedeq}
\displaystyle
\lambda h+\d_x(-c h+ q)=0,\quad \lambda q+\d_{x}\left( \big(\frac{2Q}{H}-c \big) q+ \big(-\frac{Q^2}{H^2}+\frac{H}{F^2} \big) h\right)= h-\frac{2Q}{H^2} q+2\frac{Q^2}{H^3} h
\ee
corresponding to \eqref{eval}(i), with $v_1=(h,q)^T$, 
$$
Av=\begin{pmatrix}-c h+ q\\(\frac{2Q}{H}-c) q+(-\frac{Q^2}{H^2}+p'(H)) h\end{pmatrix}
=\begin{pmatrix}-c h+ q\\(\frac{2Q}{H}-c) q+(-\frac{Q^2}{H^2}+\frac{H}{F^2}) h\end{pmatrix}\,,
$$
and
$$
Ev=\begin{pmatrix}0\\\d_hr(W) h+\d_qr(W)q\end{pmatrix}
=\begin{pmatrix}0\\ h-\frac{2Q}{H^2} q+2\frac{Q^2}{H^3} h\end{pmatrix}\,.
$$

To define an Evans-Lopatinsky determinant we need to choose a normalization of solutions to \eqref{linearizedeq}. It is convenient to enforce 
\be\label{eigen-normalization}
\begin{pmatrix}h\\q\end{pmatrix}(x_s)
\,=\,\frac{(F-2)}{2(F+1)}\sqrt{H_s}\begin{pmatrix}\left(\lambda+\frac{2}{3}\frac{F+1}{\sqrt{H_s}}\right)F\\\left(\lambda\sqrt{H_s}(F-1)+\frac23(F+1)^2\right)\end{pmatrix}
\ee
where $x_s$ denotes the sonic point in $(0,X)$, that is, the point in $(0,X)$ where $H(x_s)=H_s$. The facts that this parametrization is possible and that the analytic dependence of values at $x_s$ transfers to the joint analyticity of $(h,q)$ follows from the analysis in Appendix~\ref{s:local}. The  above normalization is chosen to ensure that at $\lambda=0$, $(h,q)\equiv(H',Q')$.

With such a choice in hands, as prescribed in \eqref{det}, we introduce the Evans-Lopatinsky determinant:
\begin{align}\label{lopatinskii}
\Delta(\lambda, \xi)&=
\det \begin{pmatrix}
\{\lambda H\} & \{-c h+ q\}_\xi\\
\{\lambda Q-(H-\frac{Q^2}{H^2})\}&\{(-\frac{Q^2}{H^2}+\frac{H}{F^2}) h+(\frac{2Q}{H}-c) q\}_\xi
\end{pmatrix}.%\\
%&=\det \begin{pmatrix}
%\{\lambda H\} & \{c h- q\}_\xi\\
%\{H-\frac{Q^2}{H^2}\}&\{(-\frac{Q^2}{H^2}+\frac{H}{F^2}+c^2) h+(\frac{2Q}{H}-2c) q\}_\xi
%\end{pmatrix},\nonumber
\end{align}
where here we recall the notation in $\{\cdot\}$ and $\{\cdot\}_\xi$ in \eqref{jumpnote}.
Similarly as in \eqref{scaleinv}, we may reduce the spectral problem to the case when $H_s=1$ by performing the rescaling 
$$
\lambda=\frac{\underline{\lambda}}{\sqrt{H_s}}\,,\qquad
\xi=\frac{\underline{\xi}}{H_s}\,,\qquad
h=H_s\,\underline{h}\left(\frac{\cdot}{H_s}\right)\,,\qquad
q=H_s^{3/2}\,\underline{q}\left(\frac{\cdot}{H_s}\right)\,.
$$
This results (with obvious notation) in
$$
\Delta(\lambda,\xi)\,=\,H_s^{5/2}\underline{\Delta}(\underline{\lambda},\underline{\xi})\,.
$$

\section{Low-frequency analysis}\label{s:Lowboundary}

\subsection{Comparison with modulation equations}\label{s:comparison}
The goal of this section is to obtain expansions of the Evans-Lopatinsky determinant $\Delta(\lambda,\xi)$ when $(\lambda,\xi)$ is sufficiently small.  To this end,
%We now obtain some expansions of $\Delta(\lambda,\xi)$ when $(\lambda,\xi)$ is small. To do this 
we exploit cancellations that have their robust origin in the structure of neighboring periodic traveling waves. In this way, we prove Propositions~\ref{factor} and~\ref{serreprop} and their corollaries at once.

With this in mind, on one hand we temporarily translate profiles to enforce that sonic points occur at $jX$, $j\in\ZM$. Thus jumps are now located at $x_-+jX$, $j\in\ZM$, where $x_-=-x_s$, $x_s$ being the original position of the sonic point in $(0,X)$. We also modify the definition of jump accordingly. Note that the reduced spectral problems are now posed on $(x_-,x_-+X)$. However, for readability's sake, we keep notation unchanged. The gain is that with this new normalization of invariance by translation it is clearly apparent that profiles are parametrized in terms of\footnote{Incidentally we observe that we could also use a parametrization by $(c,X)$, a more robust choice that would slightly simplify the present computations but would bring us farther from choices of other sections of the paper.} $(H_s,H_-)$ as 
$$
(H,Q)(x)=(H,Q)(x;H_s)\,,\qquad x\in(x_-,x_-+X) 
$$
with
$$
x_-=x_-(H_s,H_-)\,,\qquad X=X(H_s,H_-)\,,\qquad c=c(H_s)\,,\qquad q_0=q_0(H_s)\,.
$$
In particular the interior shape of profiles is obviously independent of $H_-$, and the quantities $x_-$ and $X$ are implicitly defined from it by
$$
H(x_-;H_s)=H_-\,,\qquad \left(\frac{H^2}{2F^2}+\frac{Q^2}{H}-cQ\right)(x_-+X)=\left(\frac{H^2}{2F^2}+\frac{Q^2}{H}-cQ\right)(x_-).
$$

On the other hand, let us denote the solutions $(h,q)$ of \eqref{linearizedeq} as $(h,q)(\cdot;\lambda)$, or just $(h,q)(\lambda)$, to mark dependence on $\lambda$.  
Noting that differentiating spatially the interior profile equations \eqref{smooth} one finds that both $(h,q)(\cdot;0)$ and $(H',Q')$ satisfy the same first-order system and that by normalization \eqref{eigen-normalization} they share the same value at the sonic point $0$, we conclude by uniqueness that 
$$
(h,q)(\cdot;0)\,=\,(H',Q')\,.
$$
Let us now introduce the analytic functions
$$
(\tilde h,\tilde q)(\cdot;\lambda)\,=\,\lambda^{-1}\,((h,q)(\cdot;\lambda)-(H',Q')),
$$
Now, first integrating \eqref{linearizedeq} then expanding with $(\tilde h,\tilde q)$ and using interior profile equations yields
\begin{align*}
\Delta(\lambda,\xi)&=\det 
\begin{pmatrix}
\{\lambda H\} & \{-c h+ q\}_\xi\\
\{\lambda Q-(H-\frac{Q^2}{H^2})\}&\{(-\frac{Q^2}{H^2}+\frac{H}{F^2}) h+(\frac{2Q}{H}-c) q\}_\xi
\end{pmatrix}\\
&=\lambda\,\det \begin{pmatrix}
\{ H\} &\ds-\lambda\int_{x_-}^{x_-+X}\tilde h\,+\,(e^{i\xi X}-1)(c \tilde h- \tilde q)(x_-)\\[1em]
\{\lambda Q-(H-\frac{Q^2}{H^2})\}
&\lambda\{(-\frac{Q^2}{H^2}+\frac{H}{F^2})\tilde h+(\frac{2Q}{H}-c) \tilde q\}_\xi
\,+\,\lambda\{Q\}\\
&\,-\,(e^{i\xi X}-1)(H-\frac{Q^2}{H^2})(x_-)
\end{pmatrix},\\
\end{align*}
where here we have used the fact that $\{f\}=\{f\}_\xi+(e^{i\xi X}-1)f(x_-)$ as well as the observation that
$$
\left(-\frac{Q^2}{H^2}+\frac{H}{F^2}+c^2\right) H'+\left(\frac{2Q}{H}-c\right) Q'
=\left(-\frac{q_0^2}{H^2}+\frac{H}{F^2}\right) H'=H-\frac{Q^2}{H^2}\,.
$$
This motivates the introduction of the factored determinant
\be\label{e:factored}
\hat{\Delta}(\lambda,\xi)
=\det \begin{pmatrix}
\{ H\} &\ds-\lambda\int_{x_-}^{x_-+X}\hspace{-0.5em}\tilde h\,+\,(e^{i\xi X}-1)(c \tilde h- \tilde q)(x_-)\\[1em]
\{\lambda Q-(H-\frac{Q^2}{H^2})\}
&\lambda\{(-\frac{Q^2}{H^2}+\frac{H}{F^2})\tilde h+(\frac{2Q}{H}-c) \tilde q\}_\xi
\,+\,\lambda\{Q\}\\
&\,-\,(e^{i\xi X}-1)(H-\frac{Q^2}{H^2})(x_-)
\end{pmatrix}\,.
\ee

%CHANGED-MJ just cleaned up exposition, trying to state what "this" and "that" are specifically referring to...
Now, note that\footnote{Here, $\tilde h(0)$ denotes the function $\tilde h(\cdot,\lambda)$ at $\lambda=0$, and similarly for $\tilde q(0)$.} $(\tilde{h}(0),\tilde{q}(0))$ satisfies the system 
\begin{align*}
&H'+\d_x(-c \tilde h(0)+ \tilde q(0))=0,\\
&Q'+\d_{x}\left( \left(\frac{2Q}{H}-c \right)\tilde q(0)+ \left(-\frac{Q^2}{H^2}+\frac{H}{F^2} \right) \tilde h(0)\right)=  \tilde h(0)-\frac{2Q}{H^2}  \tilde q(0)+2\frac{Q^2}{H^3}\tilde h(0),
\end{align*}
which is recognized as the linearized equation \eqref{linearizedeq} at $\lambda=0$ with inhomogeneous forcing term $(H',Q')$, which is exactly the $X$-periodic kernel of \eqref{linearizedeq}.
Further, differentiating the interior profile equation \eqref{smooth} with respect to $H_s$ one finds this system is also solved by
$$-\frac{1}{c'(H_s)}(\d_{H_s}H,\d_{H_s}Q)$$
and hence this function differs from $(\tilde{h}(0),\tilde{q}(0))$ by a constant multiple of  $(H',Q')$.  Consequently, from \eqref{e:factored} we find 
for $|\lambda|\ll 1$ and uniformly in $\xi\in[-\pi/X,\pi/X]$ that
%ENDCHANGED
\begin{align*}
&-c'(H_s)\hat{\Delta}(\lambda,\xi)\\
&\ =\det \begin{pmatrix}
\{ H\} &\ds-\lambda\int_{x_-}^{x_-+X}\hspace{-1em}\d_{H_s}H\,+\,i\xi X(c\d_{H_s}H-\d_{H_s}Q)(x_-)\\[1em]
-\{H-\frac{Q^2}{H^2}\}
&\lambda\{(-\frac{Q^2}{H^2}+\frac{H}{F^2})\d_{H_s}H+(\frac{2Q}{H}-c)\d_{H_s}Q\}\\
&\ -\,\lambda c'(H_s)\{Q\}\,+\,(e^{i\xi X}-1)c'(H_s)(H-\frac{Q^2}{H^2})(x_-)
\end{pmatrix}+O(|\lambda|\,(|\lambda|+|\xi|))\\
&\ =\lambda\Big(\{ H\}\Big\{\Big(-\frac{Q^2}{H^2}+\frac{H}{F^2}\Big)\d_{H_s}H+\Big(\frac{2Q}{H}-c\Big)\d_{H_s}Q\Big\}
-c'(H_s)\,c\,\{H\}-\Big\{H-\frac{Q^2}{H^2}\Big\}\int_{x_-}^{x_-+X}\hspace{-1em}\d_{H_s}H\Big)\\
&\ \ +(e^{i\xi X}-1)\left(\Big\{H-\frac{Q^2}{H^2}\Big\}\Big(-c'(H_s)H(x_-)+q_0'(H_s)\Big)
+c'(H_s)\{ H\}\Big(H-\frac{Q^2}{H^2}\Big)(x_-)\right)+O(|\lambda|\,(|\lambda|+|\xi|))
\end{align*}
%as $\lambda\to0$, uniformly in $\xi$, where we have used $Q=cH-q_0$.
where we have used the fact that $Q=cH-q_0$.  This verifies the expansion in Proposition \ref{factor}.  In particular,
it follows that, using the notation in \eqref{fact2}, $\alpha_0\in\RM$ and $\alpha_1\in\RM i$ are explicitly determined.
%ENDCHANGED

%CHANGED-MJ trying to be more precise
Next, we aim to prove Proposition \ref{serreprop} by comparing the above expansions with the dispersion relation from the Whitham system \eqref{whit}.  To this end, 
%To compare with the dispersion relation from the Whitham system, 
observe that, with notation from Section~\ref{s:Whitham},
%ENDCHANGED
$$
A^0=\begin{pmatrix}\ds
-\frac{\d_{H_s}X}{X^2}\int_{x_-}^{x_-+X}\hspace{-1em}H
+\frac{1}{X}\int_{x_-}^{x_-+X}\hspace{-1em}\d_{H_s}H&\vline&\ds
-\frac{\d_{H_-}X}{X^2}\int_{x_-}^{x_-+X}\hspace{-1em}H+\frac{\{H\}}{X}\d_{H_-}x_-\\
\ds+\frac{\{H\}}{X}\d_{H_s}x_-+H(x_-+X)\,\d_{H_s}X
&\vline&\ds
+H(x_-+X)\,\d_{H_-}X\\
\hline\\[-1.2em]\ds
-\frac{\d_{H_s}X}{X^2}&\vline&\ds-\frac{\d_{H_-}X}{X^2}\end{pmatrix}
$$
and
$$
A^1=c\,A^0
+\begin{pmatrix}\ds\frac{c'(H_s)}{X}\int_{x_-}^{x_-+X}\hspace{-1em}H-q_0'(H_s)&0\\
\ds\frac{c'(H_s)}{X}&0\end{pmatrix}\,.
$$
Therefore, a direct calculation yields
\begin{align*}
&\det\left(\left(\lambda-c\frac{e^{i\xi X}-1}{X}\right)A^0+\frac{e^{i\xi X}-1}{X} A^1\right)\\
&\qquad=\det
\begin{pmatrix}\ds
\frac{\lambda}{X}\int_{x_-}^{x_-+X}\hspace{-1em}\d_{H_s}H
+\lambda\frac{\{H\}}{X}\d_{H_s}x_-&\ds
\lambda\frac{\{H\}}{X}\d_{H_-}x_-+\lambda\,\frac{H(x_-+X)}{X}\,\d_{H_-}X\\\ds
+\lambda\,\frac{H(x_-+X)}{X}\,\d_{H_s}X
-\frac{e^{i\xi X}-1}{X} q_0'(H_s)&\\[1em]\ds
-\lambda\frac{\d_{H_s}X}{X^2}
+\frac{e^{i\xi X}-1}{X}\,\frac{c'(H_s)}{X}
&\ds-\lambda\frac{\d_{H_-}X}{X^2}\end{pmatrix}\\
&\qquad=\frac{\lambda}{X^3}\Big[
\lambda
\left(\{H\}\,\big(\d_{H_s}X\,\d_{H_-}x_-\,-\,\d_{H_-}X\,\d_{H_s}x_-\big)
-\d_{H_-}X\,\int_{x_-}^{x_-+X}\hspace{-1em}\d_{H_s}H\right)\\
&\qquad\qquad+(e^{i\xi X}-1)\left(\d_{H_-}X\,q_0'(H_s)
-c'(H_s)\,\{H\}\,\d_{H_-}x_--\,H(x_-+X)\,\,\d_{H_-}X\,\,c'(H_s)\right)\Big]\,.
\end{align*}
%CHANGED-MJ prose
To compare this to the expansion of $\Delta(\lambda,\xi)$ above, we note that by differentiating 
%Besides, by differentiating 
%ENDCHANGED
the defining equation for $X$, one finds
\begin{align*}
\left(H-\frac{Q^2}{H^2}\right)(x_-+X)\,\d_{H_-}X
&=-\left\{H-\frac{Q^2}{H^2}\right\}\d_{H_-}x_-\,,\\
\left(H-\frac{Q^2}{H^2}\right)(x_-+X)\,\d_{H_s}X&=
-\left\{H-\frac{Q^2}{H^2}\right\}\d_{H_s}x_-
+c'(H_s)\ c\{H\}\\
&\qquad-\left\{\left(\frac{H}{F^2}-\frac{Q^2}{H^2}\right)\d_{H_s}H+\left(\frac{2Q}{H}-c\right)\d_{H_s}Q\right\}
\end{align*}
so that, in particular,
\begin{align*}
&\left(H-\frac{Q^2}{H^2}\right)(x_-+X)\,\big(\d_{H_s}X\,\d_{H_-}x_-\,-\,\d_{H_-}X\,\d_{H_s}x_-\big)\\
&\qquad=
\d_{H_-}x_-\left(c'(H_s)\ c\{H\}
-\left\{\left(\frac{H}{F^2}-\frac{Q^2}{H^2}\right)\d_{H_s}H+\left(\frac{2Q}{H}-c\right)\d_{H_s}Q\right\}\right)
\end{align*}
and
\begin{align*}
&\left(H-\frac{Q^2}{H^2}\right)(x_-+X)\,\left(\d_{H_-}X\,q_0'(H_s)
-c'(H_s)\,\{H\}\,\d_{H_-}x_--\,H(x_-+X)\,\,\d_{H_-}X\,\,c'(H_s)\right)\\
&\qquad=
-\d_{H_-}x_-\,\left(\left\{H-\frac{Q^2}{H^2}\right\}\,(q_0'(H_s)-\,H(x_-+X)\,c'(H_s))+c'(H_s)\,\{H\}\,\left(H-\frac{Q^2}{H^2}\right)(x_-+X)\right)\\
&\qquad=
-\d_{H_-}x_-\,\left(\Big\{H-\frac{Q^2}{H^2}\Big\}\Big(-c'(H_s)H(x_-)+q_0'(H_s)\Big)
+c'(H_s)\{ H\}\Big(H-\frac{Q^2}{H^2}\Big)(x_-)\right)\,.
\end{align*}
Therefore as $\lambda\to0$, uniformly in $\xi\in[-\pi/X/\pi/X]$, we find that
\[
\Delta(\lambda,\xi)
\,=\,
\frac{X^3\left(H-\frac{Q^2}{H^2}\right)(x_-+X)}{c'(H_s)\,\d_{H_-}x_-}
\det\left(\left(\lambda-c\frac{e^{i\xi X}-1}{X}\right)A^0+\frac{e^{i\xi X}-1}{X}A^1\right)
+O(|\lambda|^2\,(|\lambda|+|\xi|))\,.
\]
This calculation proves Proposition \ref{serreprop}, thus rigorously justifying the formal Whitham modulation system \eqref{whit} as a predictor of low-frequency
stability for periodic traveling wave solutions of the inviscid St. Venant equation \eqref{sv_intro}.
%This proves both Propositions~\ref{factor} and~\ref{serreprop}.

\subsection{Low-frequency stability boundaries}\label{s:lf}

During the computations in the previous section, we found that the Evans-Lopatinsky determinant can be decomposed for sufficiently small $\lambda$, uniformly in
$\xi\in[-\pi/X,\pi/X]$ as
$$
\Delta(\lambda,\xi)=\lambda\,\hat{\Delta}(\lambda,\xi)
$$
with
$$
\hat{\Delta}(\lambda,\xi)=
\lambda\d_\lambda\hat{\Delta}(0,0)
+\frac{e^{i\xi X}-1}{i\,X}\d_\xi\hat{\Delta}(0,0)
+O(|\lambda|(|\lambda|+|\xi|))\,.
$$
In particular, observe that if $\d_\xi\hat{\Delta}(0,0)=0$ then $0$ is actually a double root of $\Delta(\cdot,\xi)$ for any $\xi$. 
This suggests that when $\d_\xi\hat{\Delta}(0,0)$ changes sign, a loop of spectrum, parametrized by $\xi$, may switch from one half-plane to the other.  Consequently,
in the absence of additional instabilities, the vanishing of $\d_\xi\hat{\Delta}(0,0)$ may indicate a transition from stability to instability (or vice versa) of the underlying periodic solution.  
Later on we shall offer another partial support to 
this scenario by proving that the parity of the number of real positive subharmonic --- i.e. corresponding to $\xi=\pi/X,-\pi/X$ --- eigenvalues changes when this occurs.

%CHANGED-MJ added footnote, added "may" language.
Likewise, if $\d_\lambda\hat{\Delta}(0,0)$ were changing sign somewhere then this may indicate a change in the parity of the number of real positive co-periodic --- i.e. corresponding to $\xi=0$ --- eigenvalues\footnote{Such a parity change would follow immediately provided one can determine, say, the sign of $\Delta(\lambda,0)$ for large $\lambda\in\RM$.}. 
However, as discussed in the introduction, we have checked numerically that $\d_\lambda\hat{\Delta}(0,0)$ never vanishes over the entire range of existence of periodic traveling 
wave solutions of \eqref{sv_intro}.  Motivated by this numerical observation, we now restrict our theoretical discussion to the case when  $\d_\lambda\hat{\Delta}(0,0)\neq 0$.
%CHANGED

From the above expansion of $\Delta(\lambda,\xi)$,  a direct computation shows that the non-trivial small root of $\Delta(\cdot,\xi)$ expands for $|\xi|\ll 1$ as 
\begin{equation}\label{second}
\lambda(\xi)=-i\alpha\xi-\beta\xi^2+O(|\xi|^3),
\end{equation}
with $\alpha$ and $\beta$ real numbers\footnote{The reality of $\alpha,\beta$ follows by the explicit expansion of $\hat\Delta(\lambda,\xi)$ in the previous section.} given by
$$
\alpha=\frac{\d_\xi\hat{\Delta}(0,0)}{i\d_\lambda\hat{\Delta}(0,0)}
$$
and
$$
\beta=\alpha\,\left(-\frac12X-\frac12\alpha\frac{\d^2_\lambda\hat{\Delta}(0,0)}{\d_\lambda\hat{\Delta}(0,0)}-i\frac{\d^2_{\lambda\xi}\hat{\Delta}(0,0)}{\d_\lambda\hat{\Delta}(0,0)}\right)
=:\alpha\gamma\,.
$$
In particular, a transition instability/stability of small eigenvalues associated with small Floquet numbers necessarily  takes place whenever, by changing parameters of wave,
the parameter $\beta$ changes sign.
% when by changing parameters of the wave $\beta$ changes sign a transition instability/stability of small eigenvalues associated with small Floquet numbers takes place. 
For this reason we use the vanishing of $\beta$ as an indicator of a low-frequency stability boundary.  Furthermore, since $\beta=\alpha\,\gamma$, 
we distinguish between two types of potential low-frequency stability boundaries:
\begin{enumerate}
\item{low-frequency stability boundary I: when $\alpha$ vanishes, and in this case we know that a full loop of spectrum passes through the origin;}
\item{low-frequency stability boundary II: when $\alpha$ does not vanish but $\gamma$ changes sign, and we at least know that the curvature at the origin of the non trivial small root changes sign.}
\end{enumerate}

We now explain how to carry out the computations of the quantities $\d_\lambda\hat{\Delta}(0,0)$, $\alpha$ and $\gamma$.  First we remind the reader that all coefficients in $A^0$ and $A^1$ from Section~\ref{s:Whitham} may be explicitly computed by simple but tedious algebraic manipulations. In particular, both $\d_\xi\hat{\Delta}(0,0)$ and $\d_\lambda\hat{\Delta}(0,0)$ have explicit expressions in terms of $H_-$, $H_s$ and $F$. The expression of the latter is both long and quite daunting, therefore we omit it and resort to numerical evaluation to determine its sign. Let us only mention that its ugliest parts involve $\ell$ and $\ell\overline{\underline H}$.

The evaluation of $\d_\xi\hat{\Delta}(0,0)$ provides a nicer form,
\begin{align}\label{alpha_1}
&\frac{\d_\xi\hat{\Delta}(0,0)}{iX}=
\Big\{H-\frac{Q^2}{H^2}\Big\}\Big(H_--\frac{q_0'(H_s)}{c'(H_s)}\Big)
-\{ H\}\Big(H_--\frac{(c(H_s)H_--q_0(H_s))^2}{H_-^2}\Big)\\\nonumber
&=
\frac{H_s^2\,\{\underline H\}}{F^2(F+1)\underline H_-^2\underline H_+^2}\Big[
\left(F^2\underline H_-^2\underline H_+^2
-2(F+1)\underline H_-\underline H_+
+(\underline H_-+\underline H_+)
\right)\left((F+1)\underline H_--3\right)\\\nonumber
&\qquad\qquad\qquad\qquad
-(F+1)\underline H_+^2\left(F^2\underline H_-^3-(F+1)^2\underline H_-^2
+2(F+1)\underline H_--1\right)\Big]\\\nonumber
&=\frac{H_s^2\,\{\underline H\}}{F^2(F+1)\underline H_-^2\underline H_+^2}\Big[
F^3\underline H_-^2\underline H_+^2
-2F^2\underline H_-\underline H_+(\underline H_+-\underline H_-)\\\nonumber
&\qquad\qquad\qquad\qquad+F(3\underline H_-^2\underline H_+^2-4\underline H_-\underline H_+(\underline H_-+\underline H_+)+7\underline H_-\underline H_++\underline H_-^2+\underline H_+^2)\\\nonumber
&\qquad\qquad\qquad\qquad
+\underline H_-^2\underline H_+^2
-2\underline H_-\underline H_+(\underline H_-+\underline H_+)
+\underline H_-^2+\underline H_+^2
+7\underline H_-\underline H_+
-3(\underline H_-+\underline H_+)
\Big]\,.
\end{align}
Once the sign of $\d_\lambda\hat{\Delta}(0,0)$ is known, determining the sign of $\alpha$ is thus reduced to the study of the sign of a third order polynomial in $F$ with coefficients given as polynomials in $\underline H_-$ and $\underline H_+$, thus as functions of $\underline{H}_-$. This task may be achieved by exact algebraic computations. In particular, by root solving symbolically in MATLAB we have obtained the low-frequency stability boundary~I expounded in Figure~\ref{fig_low_boundaries}(a).   While we have restricted to the case when $H_s=1$ in Figure \ref{fig_low_boundaries}(a), we recall that 
conclusions for general values of $H_s$ can be deduced thanks to the scaling invariance % but as already pointed out this is sufficient to draw  conclusions for general  values of $H_s$ by a scaling argument: see 
\eqref{scaleinv}.

In contrast, the sign of $\gamma$ is not directly related to the first-order\footnote{Yet we expect that it could be related to a suitable second-order system as in \cite{NR,JNRZ-RD-W,JNRZ13,KR}.} modulation system \eqref{whit}, %from Section~\ref{s:Whitham} 
and it is unclear to us whether $\gamma$ may be obtained in closed form.  However, the only missing piece in obtaining a closed form expression for $\gamma$ is knowledge of $\d_\lambda(\tilde h,\tilde q)(\cdot;0)=\tfrac12\d_\lambda^2(h,q)(\cdot;0)$ (with notation from Section~\ref{s:comparison}). By differentiating twice defining equations for $(h,q)(\cdot;\lambda)$, we find that $\d_\lambda(\tilde h,\tilde q)(x_s;0)=0$ and that
\begin{align*}
&\tilde h(0)+\d_x(-c \d_\lambda\tilde h(0)+ \d_\lambda\tilde q(0))=0,\\
&\tilde q(0)+\d_{x}\left( \left(\frac{2Q}{H}-c \right)\d_\lambda\tilde q(0)+ \left(-\frac{Q^2}{H^2}+\frac{H}{F^2} \right) \d_\lambda\tilde h(0)\right)=  \d_\lambda\tilde h(0)-\frac{2Q}{H^2} \d_\lambda\tilde q(0)+2\frac{Q^2}{H^3}\d_\lambda\tilde h(0)
\end{align*}
with
$$
(\tilde h,\tilde q)(\cdot;0)
=-\frac{1}{c'(H_s)}\,(\d_{H_s}H,\d_{H_s}Q)+\frac{3(F+2)}{2(F+1)(F-2)}\sqrt{H_s}\,(H',Q')\,.
$$
In the foregoing we have also used \eqref{eigen-normalization} to determine the multiple of $(H',Q')$ in $(\tilde h,\tilde q)(\cdot;0)$. To determine $\gamma$ we have indeed used  a numerical approximation of $\d_\lambda(\tilde h,\tilde q)(\cdot;0)$ obtained by solving the above singular ODE Cauchy problem. Our numerical outcome is that $\gamma$ vanishes only near the lower boundary of the existence domain, that is, the potential low-frequency stability boundary II actually sits between the low-frequency stability boundary I and the lower boundary of the existence domain. As $F$ increases the low-frequency stability boundary II gets even closer to the lower boundary of the existence domain. See Figure~\ref{fig_low_boundaries}(b) for an enlarged figure of the low-frequency stability boundary II near $F=2$.

\begin{figure}[htbp]
\begin{center}
\includegraphics[scale=.6]{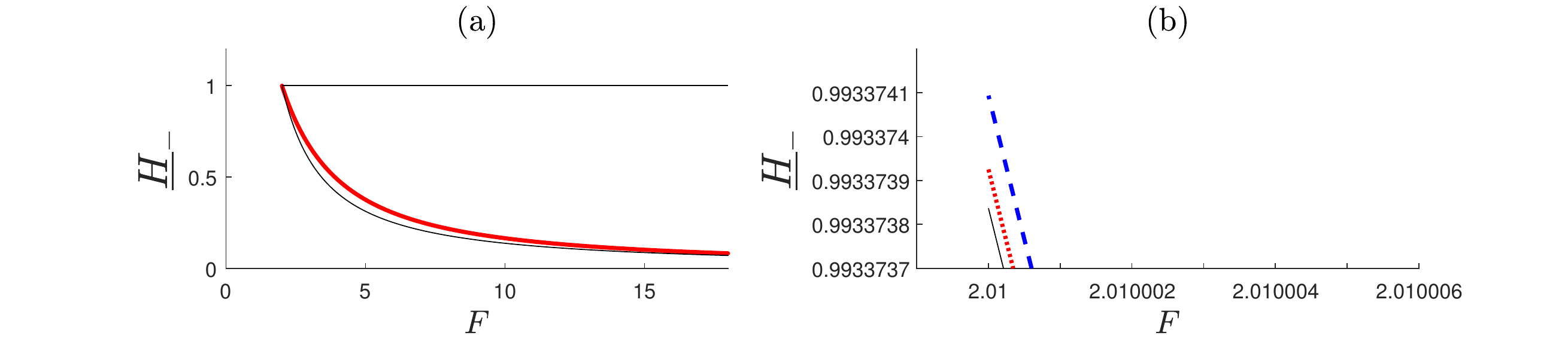}
\end{center}
%\caption{(a) low-frequency stability boundary I with $H_s=1$; (b) low-frequency stability boundary II}
\caption{(a) A plot of the limits of the domain of existence, as well as the low-frequency stability boundaries (recall the labeling scheme from Figure \ref{fig_scheme}).  As the low-frequency
stability boundary II is nearly indistinguishable from the lower existence boundary, we show in (b) an enlargement of (a) near $F=2$.  In (b), we are showing
the lower existence boundary, the low-frequency stability boundary II, and the mid-frequency stability boundary.}
\label{fig_low_boundaries}\end{figure}

\section{High-frequency analysis}

In this section we analyze the behavior of solutions of to the interior eigenvalue problems \eqref{linearizedeq}-\eqref{eigen-normalization} when $|\lambda|\to\infty$.
These conclusions will then be used to deduce high- and 
%KZ: right?
%mean-frequency 
medium-frequency
stability indices. 

Though detailed computations depend on specificities of \eqref{sv_intro}, the method of our analysis applies to general hyperbolic systems of balance laws \eqref{ab}. Indeed, the tenet of the argument is that expansions when $|\lambda|\to\infty$ perturb from purely hyperbolic modes (of the system without source term) and  that the first correction to this picture encodes how the relaxation damps these modes. Note that this occurs even when, as here, the relaxation is degenerate in the sense that it appears only in one equation of the system. In particular the condition on the index $\mathbf{I}$ introduced below is analogous to the averaged slope condition of the viscous case that was shown to always hold in \cite{RZ}.

\subsection{High-frequency expansion}
As in Section~\ref{s:application}, we begin by changing the unknown from $(h,q)$ to $W:=(w_1,w_2)=(ch-q,h)$, so that the system \eqref{linearizedeq} becomes
$$
\left(p'(H)-\frac{q_0^2}{H^2}\right)\,W'
=(\lambda\,B_0+B_1)\,W,
$$
where
\begin{align}
B_0&=\begin{pmatrix}0&p'(H)-\frac{q_0^2}{H^2}\\1&-\frac{2q_0}{H}\end{pmatrix}\nonumber\\
B_1&=\begin{pmatrix}0&0\\
-\d_qr(H,Q)-\left(\frac{2q_0}{H}\right)'
&-\left(p'(H)-\frac{q_0^2}{H^2}\right)'+\d_hr(H,Q)+c\d_qr(H,Q)\end{pmatrix},\label{e:B1}
\end{align}
and normalization \eqref{eigen-normalization} becomes
$$
W(x_s)=\frac{(F-2)}{2(F+1)}\sqrt{H_s}\begin{pmatrix}
2\lambda\sqrt{H_s}\\
\left(\lambda+\frac{2}{3}\frac{F+1}{\sqrt{H_s}}\right)F\end{pmatrix}
$$

Now, note that $B_0=P_0\,\Gamma\,P_0^{-1}$ with
$$
\Gamma=\begin{pmatrix}\mu_+&0\\0&\mu_-\end{pmatrix}\,,\qquad
\mu_{\pm}=-\frac{q_0}{H}\pm\sqrt{p'(H)}\,,
$$
and
\begin{align}
P_0&=\begin{pmatrix}\frac{q_0}{H}+\sqrt{p'(H)}&\frac{q_0}{H}-\sqrt{p'(H)}\\
1&1\end{pmatrix}\label{e:P0}\\
P_0^{-1}&=\frac{1}{2\sqrt{p'(H)}}
\begin{pmatrix}1&\sqrt{p'(H)}-\frac{q_0}{H}\\
-1&\sqrt{p'(H)}+\frac{q_0}{H}\end{pmatrix}\nonumber\,.
\end{align}
By setting $W_0=P_0^{-1}W$, the above system is transformed into
$$
\left(p'(H)-\frac{q_0^2}{H^2}\right)\,W_0'
=(\lambda\,\Gamma+C_0)\,W_0
$$
with
\be
C_0=P_0^{-1}B_1P_0-\left(p'(H)-\frac{q_0^2}{H^2}\right)P_0^{-1}(P_0)'\,.\label{e:C0}
\ee
Observe that, by direct calculations similar to those in Section~\ref{s:application}, we have 
$$
\Gamma(x_s)=\begin{pmatrix}
0&0\\
0&-\frac{2\sqrt{H_s}}{F}\end{pmatrix}
\qquad\textrm{and}\qquad
C_0(x_s)=\begin{pmatrix}
0&0\\
\frac43\frac{(F+1)}{F}&0\end{pmatrix}
$$
and that the normalization \eqref{eigen-normalization} becomes
$$
W_0(x_s)=
\frac{F(F-2)}{4(F+1)}
\begin{pmatrix}1&0\\
-1&\frac{2\sqrt{H_s}}{F}\end{pmatrix}
\begin{pmatrix}
2\lambda\sqrt{H_s}\\
\left(\lambda+\frac{2}{3}\frac{F+1}{\sqrt{H_s}}\right)F\end{pmatrix}
=\frac{F(F-2)}{4(F+1)}
\begin{pmatrix}2\lambda\sqrt{H_s}\\
\frac43(F+1)
\end{pmatrix}\,.
$$

We now set 
$$
P_1(\cdot;\lambda)=I_2+\frac1\lambda\tilde P_1\,,\qquad 
\tilde P_1=\begin{pmatrix}0&-\frac{(C_0)_{1,2}}{\mu_+-\mu_-}\\
\frac{(C_0)_{2,1}}{\mu_+-\mu_-}&0\end{pmatrix}
$$
to ensure commutator relations
$$
[\lambda\Gamma,P_1(\cdot;\lambda)]
=[\Gamma,\tilde P_1(\cdot;\lambda)]
=-\begin{pmatrix}0&(C_0)_{1,2}\\
(C_0)_{2,1}&0\end{pmatrix}\,.
$$
%Then take $\lambda$ large enough to enforce that $P_1(\cdot;\lambda)$ is point-wise invertible and set $W_1=P_1^{-1}W_0$. The interior spectral system then becomes
For $|\lambda|$ sufficiently large, $P_1(\cdot;\lambda)$ is point-wise invertible.  For such $\lambda$, setting $W_1=P_1^{-1}W_0$ transforms the interior spectral system into
$$
\left(p'(H)-\frac{q_0^2}{H^2}\right)\,W_1'
=\left(\begin{pmatrix}\lambda\,\mu_++(C_0)_{1,1}&0\\
0&\lambda\,\mu_-+(C_0)_{2,2}\end{pmatrix}
+\frac1\lambda C_1(\cdot;\lambda)\right)\,W_1
$$
with
$$
C_1(\cdot;\lambda)
=-P_1^{-1}\tilde P_1 \begin{pmatrix}(C_0)_{1,1}&0\\
0&(C_0)_{2,2}\end{pmatrix}+P_1^{-1}C_0\tilde P_1
-\left(p'(H)-\frac{q_0^2}{H^2}\right)P_1^{-1}(\tilde P_1)'\,.
$$
Note that $C_1$ vanishes at $x_s$ and that the normalization \eqref{eigen-normalization} now reads
$$
W_1(x_s)
=\frac{F(F-2)}{4(F+1)}
\begin{pmatrix}1&0\\
-\frac1\lambda\frac23\frac{(F+1)}{\sqrt{H_s}}&1\end{pmatrix}
\begin{pmatrix}2\lambda\sqrt{H_s}\\\frac43(F+1)\end{pmatrix}
=\frac{F(F-2)}{4(F+1)}\begin{pmatrix}2\lambda\sqrt{H_s}\\0\end{pmatrix}\,.
$$

In order not to suggest spurious singularities, we introduce the following smooth factorizations:
\begin{align}
\mu_+&=\left(p'(H)-\frac{q_0^2}{H^2}\right)\tilde\mu_+\,,
&&&
(C_0)_{1,1}&=\left(p'(H)-\frac{q_0^2}{H^2}\right)\tilde \gamma_+\,,\label{e:tilde}\\
(C_0)_{2,2}&=\left(p'(H)-\frac{q_0^2}{H^2}\right)\tilde \gamma_-\,,&&&
C_1&=\left(p'(H)-\frac{q_0^2}{H^2}\right)\tilde C_1\,.\nonumber
\end{align}
The above computations suggests that $W_1(x)$ behaves to leading order as
$$
\frac{F(F-2)}{4(F+1)}\begin{pmatrix}\ds
2\lambda\sqrt{H_s}\,e^{\int_{x_s}^x(\lambda\,\tilde\mu_+(y)+\tilde \gamma_+(y))d y}\\0\end{pmatrix}
$$
as $|\lambda|\to\infty$. To prove this claim, we introduce 
$$
\widetilde W_1 (x)=e^{-\int_{x_s}^x(\lambda\,\tilde\mu_+(y)+\tilde \gamma_+(y))d y}W_1(x)
$$
and note that $\widetilde W_1$ satisfies the fixed point equation $\widetilde W_1=\Phi(\widetilde W_1)$, where $\Phi$ is defined by 
\begin{align*}
&\Phi(\widetilde W)(x)
=\frac{F(F-2)}{4(F+1)}\begin{pmatrix}2\lambda\sqrt{H_s}\\0\end{pmatrix}\\
&+\frac1\lambda
\int_{x_s}^x
e^{-\int_y^x(\lambda\,\tilde\mu_+(z)+\tilde \gamma_+(z))d z}
\begin{pmatrix}1&0\\
0&e^{-\lambda\int_y^x\frac{\mu_+-\mu_-}{p'(H)-q_0^2\,H^{-2}}(z)\,d z-
\int_y^x(\tilde \gamma_+-\tilde \gamma_-)(z)\,dz}\end{pmatrix}
\tilde C_1(y)\,\widetilde W (y)\,d y\,.
\end{align*}
Since both functions
$$
\frac{\mu_+-\mu_-}{p'(H)-q_0^2\,H^{-2}}(x)
\qquad\textrm{and}\qquad
\tilde\mu_+(x)
$$
have the same sign as $(x-x_s)$ over $[0,X]$, it follows that $\Phi$ is a Lipschitz function on $C^0([0,X];\CM^2)$ with Lipschitz constant $K/|\lambda|$ for some constant $K>0$ uniform on $\lambda$ when $\Re(\lambda)\geq-\theta$, with $\theta\in [0,(F-2)/2\sqrt{H_s})$ held fixed, and $\lambda$ is sufficiently large. In particular, under such conditions on $\lambda$ the function
$\Phi$ is a strict contraction so that, in particular, for any function $\widetilde W\in C^0([0,X];\CM^2)$
$$
\|\widetilde W_1-\widetilde W\|_{L^\infty(0,X)}
\leq \frac{1}{1-K\,|\lambda|^{-1}}\|\widetilde W-\Phi(\widetilde W)\|_{L^\infty(0,X)}
$$
where $K$ is a uniform constant. Thus with the same constant $K$ we have the estimates
$$
\|\widetilde W_1\|_{L^\infty(0,X)}
\leq \frac{|\lambda|\sqrt{H_s}}{1-K\,|\lambda|^{-1}}
\frac{F(F-2)}{2(F+1)}
$$
and
$$
\left\|\widetilde W_1-\frac{F(F-2)}{4(F+1)}\begin{pmatrix}2\lambda\sqrt{H_s}\\0\end{pmatrix}\right\|_{L^\infty(0,X)}
\leq \frac{K}{|\lambda|-K}\frac{|\lambda|\sqrt{H_s}}{1-K\,|\lambda|^{-1}}
\frac{F(F-2)}{2(F+1)}\,.
$$

By undoing change of variables $(h,q)=(w_2,cw_2-w_1$, $(w_1,w_2)^T=W=P_0P_1W_1$, we have proved the following proposition.

\begin{proposition}\label{hflem} Let $\theta\in [0,(F-2)/2\sqrt{H_s})$. The solution $(h,q)$ to \eqref{linearizedeq}-\eqref{eigen-normalization} satisfies
\be\label{hfapprox}
e^{-\int_{x_s}^x(\lambda\,\tilde\mu_+(y)+\tilde \gamma_+(y))d y}\begin{pmatrix}h\\q\end{pmatrix}(x)=
\lambda\sqrt{H_s}\,\frac{F(F-2)}{2(F+1)}\begin{pmatrix}1\\\frac{Q(x)}{H(x)}-\sqrt{p'(H(x))}\end{pmatrix}+O(1)
\ee
as $|\lambda|\to\infty$ with $\Re(\lambda)\geq-\theta$, where 
$$\tilde\mu_+=\left(\sqrt{P'(H)}+\frac{q_0}{H}\right)^{-1}$$
and $\tilde \gamma_+$ is defined through \eqref{e:B1}-\eqref{e:P0}-\eqref{e:C0}-\eqref{e:tilde}.
\end{proposition}

\subsection{Co-periodic and subharmonic real-frequency instability indices}\label{s:mfi}

Sending $\Re(\lambda)\to\infty$ in \eqref{lopatinskii}, we recover through the above proposition the results in \cite{N2} that potential unstable frequencies have bounded real parts. This is essentially necessary to local well-posedness, or, in other words, Hadamard stability, as proved in \cite{N3}.
We capture this in the following result.

\begin{proposition}\label{lopprop}
Uniformly in $\xi\in[-\pi/X,\pi/X]$, the Evans-Lopatinsky determinant defined in \eqref{lopatinskii} satisifes
\be\label{lopdef}
\Delta(\lambda,\xi)\sim
\gamma_0\lambda^2\,e^{\int_{x_s}^X(\lambda\,\tilde\mu_+(y)+\tilde \gamma_+(y))d y}\,
\ee
as $\Re(\lambda)\to\infty$, where the constant $\gamma_0$ is defined explicitly as
$$
\gamma_0:=\sqrt{H_s}\,\frac{F(F-2)}{2(F+1)}
\{H\}\,\left(\frac{\sqrt{H_+}}{F}+\frac{q_0}{H_+}\right)^2. 
$$
In particular there is an upper 
	%bounder 
bound on real parts of unstable frequencies.
\end{proposition}

\begin{proof}
For $\Re(\lambda)\to \infty$, the second column in \eqref{lopatinskii} is dominated by
the exponentially large value at $x=X$, whereas, the first column of the Evans-Lopatinsky determinant is dominated by $(\lambda\{H\},\lambda\{Q\})^T$. Combining these observations, we obtain \eqref{lopdef} with
$$
\gamma_0=\det\begin{pmatrix}
\{H\}&-c+\frac{Q(X)}{H_+}-\frac{\sqrt{H_+}}{F}\\
c\{H\}&\frac{H_+}{F^2}-\frac{Q(X)^2}{H_+^2}+\left(\frac{2Q(X)}{H_+}-c\right)\left(\frac{Q(X)}{H_+}-\frac{\sqrt{H_+}}{F}\right)\end{pmatrix}
$$
which is seen to coincide with the above value by direct expansion of $Q(X)=cH_+-q_0$.
\end{proof}

The above proposition may be used to derive co-periodic and subharmonic real-frequency instability indices.  Indeed, note that from the uniqueness of solutions$(h,q)$ of 
\eqref{linearizedeq}-\eqref{eigen-normalization} it follows that 
$(h,q)$ is real whenever $\lambda$ is real and hence we have that $\Delta(\lambda,0)$ and $\Delta(\lambda,\pi/X)$ are real for $\lambda\in\RM$.  
By Proposition~\ref{lopprop} we deduce that both $\Delta(\lambda,0)$ and $\Delta(\lambda,\pi/X)$ are positive when $\lambda\in\RM$ is positive and sufficiently large. 
Using the already established expansions
\begin{align*}
\Delta(\lambda,0)&=\,\lambda^2\d_\lambda\hat\Delta(0,0)\,+O(|\lambda|^3)\\
\Delta\left(\lambda,\frac{\pi}{X}\right)&=\,-2\lambda\frac{\d_\xi\hat\Delta(0,0)}{i\,X}\,+O(|\lambda|^2),
\end{align*}
valid as $\lambda\to0$, we deduce by the intermediate value theorem that the sign of $\d_\lambda\hat\Delta(0,0)$ determines modulo two the number of co-periodic - that is, associated with $\xi=0$ - real unstable frequencies, and that the sign of $i\d_\xi\hat\Delta(0,0)$ determines modulo two the number of subharmonic - that is, associated with $\xi=\pi/X$ - real unstable frequencies. 
Incidentally, note that if $\d_\lambda\hat\Delta(0,0)$ is positive (as suggested by our numerical observations) then by \eqref{second} the
quantity $\Delta(0,\pi/X)$ has the opposite sign of $\alpha$.  Moreover, recall that both $\d_\lambda\hat\Delta(0,0)$ and $\d_\xi\hat\Delta(0,0)$ are directly related to the Whitham modulation system and are hence explicitly computable: see Section \ref{s:comparison} above.

Taken all together, the above considerations prove Theorem~\ref{lfthm}.

\subsection{High-frequency stability index} \label{s:hfi}
Using the result of Proposition \ref{hflem} without assuming that $\Re(\lambda)$ is large, we now study high-frequency stability.  
In this case, we cannot use exponential growths to discard parts of the expansions as in Proposition \ref{lopprop} and, consequently, the expansion \eqref{lopdef} is modified into
\begin{align}
\Delta(\lambda,\xi)
&=
\lambda^2\,e^{-\lambda\int_0^{x_s}\tilde\mu_++\int_{x_-}^X\tilde \gamma_++i\xi X}
\gamma_0\\
&\qquad\times\left(e^{\lambda\int_0^{X}\tilde\mu_+-i\xi X}\left(1+O\left(|\lambda|^{-1}\right)\right)
-\left(\frac{\frac{\sqrt{H_-}}{F}+\frac{q_0}{H_-}}{\frac{\sqrt{H_+}}{F}+\frac{q_0}{H_+}}\right)^2
e^{-\int_0^X\tilde \gamma_+}+O\left(|\lambda|^{-1}\right)\right)
\nonumber
\end{align}
uniformly in $\xi\in[-\pi/X,\pi/X]$ when $|\lambda|\to\infty$ with $\Re(\lambda)\geq-\theta$, $\theta$ being held fixed in $[0,(F-2)/2\sqrt{H_s})$. Defining the high-frequency index
\begin{equation}
\mathbf{I}:=\left(\frac{\frac{\sqrt{H_-}}{F}+\frac{q_0}{H_-}}{\frac{\sqrt{H_+}}{F}+\frac{q_0}{H_+}}\right)^2
e^{-\int_0^X\tilde \gamma_+}
\end{equation}
where $\tilde \gamma_+$ is derived from \eqref{e:B1}-\eqref{e:P0}-\eqref{e:C0}-\eqref{e:tilde}, we obtain the following characterization of high-frequency spectrum.

\begin{proposition}\label{hfprop}
If $\mathbf{I}\leq e^{-\frac{F-2}{2\sqrt{H_s}}\int_0^{X}\tilde\mu_+}$, then for any $\theta\in[0,(F-2)/2\sqrt{H_s})$ there exists a $R_0>0$ such that the system
\eqref{linearizedeq}-\eqref{eigen-normalization} has no spectrum in the set
$$\{\,\lambda\,:\,\Re(\lambda)\geq-\theta\ \textrm{and}\ |\lambda|\geq R_0\}\,.$$
Conversely, if $\mathbf{I}> e^{-\frac{F-2}{2\sqrt{H_s}}\int_0^{X}\tilde\mu_+}$ then for any 
$$\theta\in\left(\frac{\ln(\mathbf{I})}{\int_0^{X}\tilde\mu_+},\frac{F-2}{2\sqrt{H_s}}\right)$$ there exists a $R_0>0$ such that the spectrum in
$$\{\,\lambda\,:\,\Re(\lambda)\geq-\theta\ \textrm{and}\ |\lambda|\geq R_0\}$$
consists exactly of two curves, made of simple Floquet eigenvalues and locally parametrized by $\xi$, asymptoting to the vertical line with equation 
$$
\Re(\lambda)\,=\,\frac{\ln(\mathbf{I})}{\int_0^{X}\tilde\mu_+}\,.
$$
\end{proposition}

\begin{proof}
This follows directly through Rouch\'e's theorem from a comparison of zeros $\lambda$ of
$$
\lambda^{-2}\,e^{\lambda\int_0^{x_s}\tilde\mu_+-\int_{x_-}^X\tilde \gamma_+-i\xi X}
\gamma_0^{-1}\,\Delta(\lambda,\xi)
$$
with those of
$$
e^{\lambda\int_0^{X}\tilde\mu_+-i\xi X}-\mathbf{I}\,.
$$
uniformly in $\xi$ in the aforementioned zone of $\lambda$s.
\end{proof}

Observe now that Theorem~\ref{hfthm} is a simple corollary of Proposition~\ref{hfprop}. 

\medskip

In order to compute the high-frequency index , observe that the integral in the definition of the index $\mathbf{I}$ may be written as an integral between $H_-$ and $H_+$ of a rational function of $h$ whose denominator is easily factored so that an exact formula for $\mathbf{I}$ may be algebraically determined. Yet, we omit it here since the resulting formula is both very long and not very instructive.

Instead, we fix $H_s=1$ and evaluate numerically the high-frequency stability index $\mathbf{I}=\mathbf{I}(F,\underline{H}_-)$.
%CHANGED-MJ reworded 
%The numerical outcome is that in our zone of investigation it is seen to be identically $<1$, across the entire parameter-range of existence indicating high-frequency stability of all waves. 
The outcome of this numerical computation is the observation that, across the entire parameter-range of existence, one has $\mathbf{I}(F,\underline{H}_-)<1$, indicating high-frequency 
stability of all periodic traveling wave solutions of \eqref{sv_intro}.
%ENDCHANGED.
See Figure~\ref{highindex} for a depiction of the
typical behavior of the high-frequency stability index as parameters $\underline{H}_-$ and $F$ are varied.

\begin{figure}[htbp]
\begin{center}
\includegraphics[scale=.6]{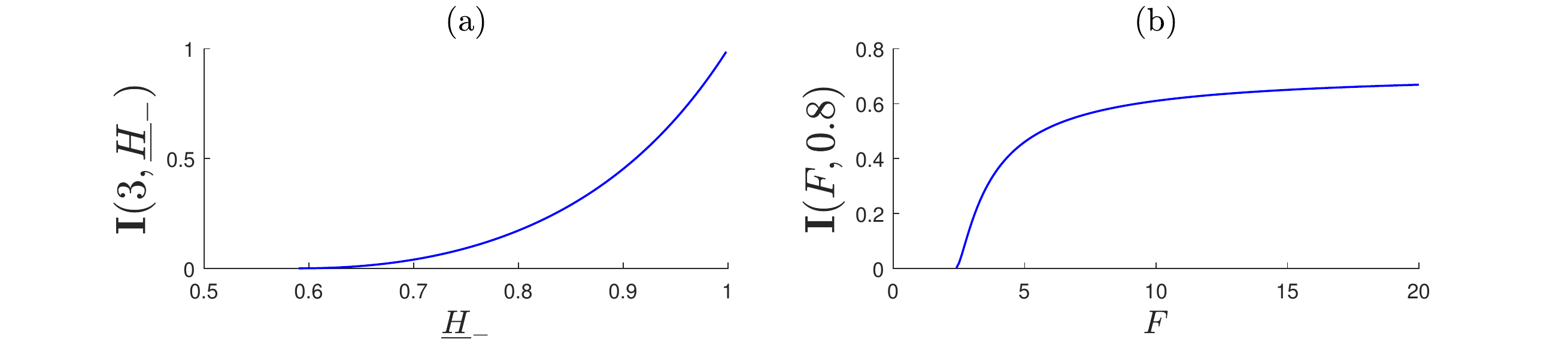}
\end{center}
\caption{Numerical evaluation of $\mathbf{I}(F,\underline{H}_-)$ for (a) $F=3$, $\underline{H}_-\in(H_{hom}(3),1]$ and for (b) $\underline{H}_-=0.8$, $F\in[2.4,20]$.}
\label{highindex}
\end{figure}

\section{Numerical investigations and stability diagram}\label{s:numerics}
In this section, we describe how we have carried out our numerical investigations, providing complete stability diagrams expounded in the introduction in Figure~\ref{fig_rosetta} and  Figure~\ref{fig_compare}. While some details on the implementation are provided in Appendix~\ref{s:computations}, our focus here is on methodology.

The general framework introduced in Section~\ref{s:evans-lop} and the exact factorization from Section~\ref{s:Lowboundary} lead to the problem of finding roots of $\hat{\Delta}(\cdot,\xi)$ 
for any $\xi\in[-\pi/X,\pi/X]$, where $\hat{\Delta}$ is as defined in \eqref{e:factored}.

Let us first explain how to compute $(h,q)$ (or its reduced version $(\tilde h,\tilde q)$) from equations \eqref{linearizedeq}-\eqref{eigen-normalization} and, in particular, 
how to deal with the presence of the singular point at $x=x_s$.  At the numerical level, we follow the pattern of the existence proof provided in Section~\ref{s:general}, that is, in a small neighborhood of the singular point we carry out a specific, asymptotic treatment of the problem and then solve outward from the boundary of the neighboring zone with a standard ODE solver. Near the sonic point, we use the analyticity of the sought solution to approximate it by a finite series whose coefficients are determined from the finite-order recursion relation imposed by the equation. At this stage, since coefficients of the eigen equation and $H'$ itself are expressed in terms of $H$, it is convenient to use a truncated version of an expansion in terms of $H(x)-H_s$ instead of $x-x_s$, that is, to approximate $(h,q)(x)$ with
a finite sum of the form
$$
\sum_{n=0}^{N-1} (H(x)-H_s)^n\,(a_n,b_n)
$$
where $(a_n,b_n)$ are determined recursively. Moreover the convergence of the corresponding series is readily deduced by applying the general theory in Section~\eqref{s:general} to the equation obtained from \eqref{linearizedeq} after the change of variable $x\mapsto H(x)$. Recurrence relations are obtained and solved by using a numerical symbolic solver. To use a relatively small number $N$ of terms in the expansion, the size of the neighborhood of the sonic point on which the expansion is used is chosen adaptively, so as to ensure reasonable accuracy of the power series expansion. This scheme turns out to be very robust.

Once the above approximations of $(h,q)$ has been achieved, our computations follow the by now classical road map to investigate stability of periodic waves from Evans function computations. See in particular \cite{BJNRZ1} for comparison and for some implementation details omitted here.

Explicitly, to determine stability, we fix $0<r<R$ with $r\ll1$ and $R\gg1$ and examine the presence of spectrum associated to \eqref{linearizedeq}-\eqref{eigen-normalization} within the set
$$
\Omega(r,R):=\{\,\lambda\,:\, \Re(\lambda)>0,r<|\lambda|<R\,\}\,.
$$
To this end, we compute for any $\xi\in[-\pi/X,\pi/X]$ a winding number for the associated reduced periodic Evans-Lopatinsky function $\hat{\Delta}(\cdot,\xi)$, i.e. we numerically compute the 
complex contour integral
$$
n(\xi;\Omega):=\frac{1}{2\pi i}\oint_{\d\Omega}\frac{\d_{\lambda}\hat\Delta(\lambda,\xi)}{\hat\Delta(\lambda,\xi)}~d\lambda
=\frac{1}{2\pi i}\oint_{\d\Omega}\d_\lambda\arg\left(\hat\Delta(\lambda,\xi)\right)~d\lambda\,.
$$
Since $\hat\Delta(\cdot,\xi)$ is complex analytic in $\lambda$, it is clear that this counts its number of zeros of $\hat{\Delta}(\cdot,\xi)$ within the set $\Omega$.  
To capture the structure of the full spectrum, this winding number must be computed for a suitably fine discretization of $\xi\in[-\pi/X,\pi/X]$.
Observe that a significant amount of redundant calculation is spared by noting that\footnote{We keep the $x$-notation for clarity but recall that actually we perform numerical computations after the change of variable $x\mapsto H(x)$ has been performed.} 
\begin{align*}
\hat{\Delta}(\lambda,\xi)
&=\det \begin{pmatrix}
\{ H\} &\ds-\lambda\int_{x_-}^{x_-+X}\hspace{-0.5em}\tilde h\,-\,(c \tilde h- \tilde q)(x_-)\\[1em]
\{\lambda Q-(H-\frac{Q^2}{H^2})\}
&\lambda\{(-\frac{Q^2}{H^2}+\frac{H}{F^2})\tilde h+(\frac{2Q}{H}-c) \tilde q\}
\,+\,\lambda\{Q\}\\
&\,+\,(H-\frac{Q^2}{H^2})(x_-)
\end{pmatrix}\\
&+e^{i\xi X}
\det \begin{pmatrix}
\{ H\} &\ds (c \tilde h- \tilde q)(x_-)\\[1em]
\{\lambda Q-(H-\frac{Q^2}{H^2})\}
&-\lambda\left(-\frac{Q^2}{H^2}+\frac{H}{F^2})\tilde h+(\frac{2Q}{H}-c) \tilde q\right)(x_-)\\
&\,-\,(H-\frac{Q^2}{H^2})(x_-)
\end{pmatrix}
\end{align*}
and that the computation of the two determinants above is independent of $\xi$.  For illustration, typical graphs of how the contour $\d\Omega(r,R)$ maps under $\hat\Delta(\cdot,\xi)$ 
are given in Figure~\ref{winding}.  We have used the above method to verify 
%KZ
%mean-frequency stability, 
medium-frequency stability, 
for $2.3\le F\le 19$, discretizing uniformly into $6000$ points the area of parameters corresponding to waves having passed our low-frequency stability tests, choosing $r=0.01$, $R=400$ and $\xi$ in a $1000$-point uniform mesh of the interval $[-\pi/X,\pi/X]$. From these calculations was determined an additional 
%KZ
%``mean-frequency stability boundary" 
``medium-frequency stability boundary" 
which crosses the low-frequency stability boundary I at about $F=16.3$; see the illustration in Figure \ref{fig_scheme}.

\begin{figure}[htbp]
\begin{center}
\includegraphics[scale=.5]{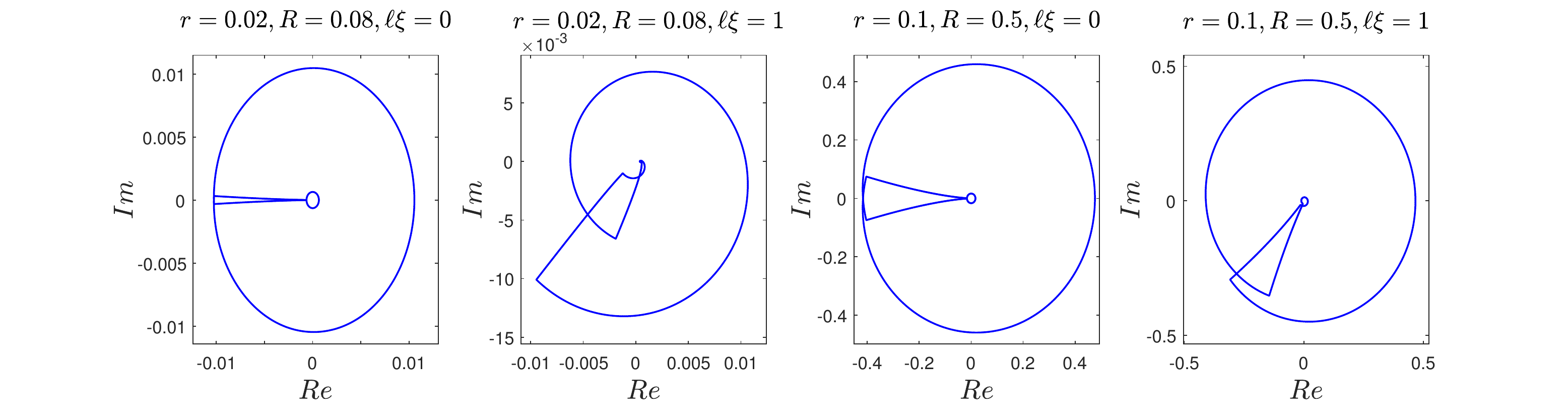}
\end{center}
\caption{Images of the contour $\d\Omega(r,R)$ under the function $\hat\Delta(\cdot,\xi)$ with $F=8$, $\underline{H}_-=0.22$ (and $H_s=1$, hence $\ell=X$). 
The calculated winding numbers are $0,1,0,0$, respectively.}
\label{winding}
\end{figure}

Both to improve accuracy by bisecting on $\underline H_-$ and to shed some light on the transition to instability, we have refined numerically all of the aforementioned stability boundaries. 
Concerning the low-frequency stability boundaries, we have tracked as a function of the Bloch frequency $\xi$ the small root of $\hat\Delta(\cdot,\xi)$ by running a classical root solver. Consistently with our theoretical analysis, we have observed that:
\begin{enumerate}
\item[i.] As $\underline{H}_-$ decreases and crosses the low-frequency stability boundary I, the spectra first appears to be a ``circle" on the right half plane (indicating instability),
gradually shrink to the origin (corresponding to $\alpha=0$), and finally grow to be a circle on the left half plane (indicating low-frequency stability). See Figure~\ref{spectrum}.(a).
\item[ii.] As $\underline{H}_-$ decreases and crosses the low-frequency stability boundary II, the spectra  near the origin initially curves into the left half plane (corresponding
to $\beta>0$), gradually becomes more and more vertical, and finally curves to lie on the right half plane (corresponding to $\beta<0$).  See Figure~\ref{spectrum}.(b).
\end{enumerate}
\begin{figure}[htbp]
\begin{center}
\includegraphics[scale=0.6]{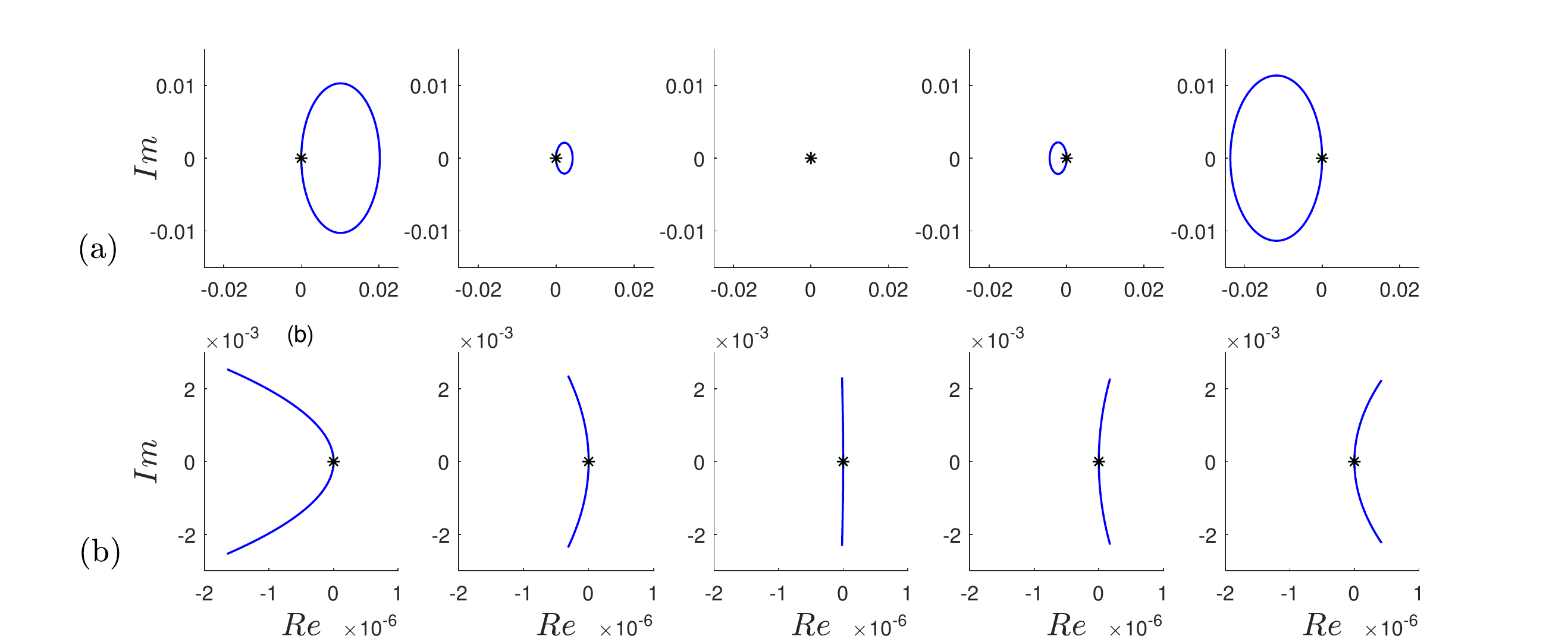}
\end{center}
\caption{(a) Plots of the spectrum near the origin with $H_s=1$ when $\underline{H}_-$ crosses the low-frequency stability boundary I at  $F=2.492779325091594,\;\underline{H}_-=0.806451612903226$; (b) Plots of the spectrum near the origin with $H_s=1$ when $\underline{H}_-$ crosses the low-frequency stability boundary II at $F=2.5,\;\underline{H}_-=0.745329985201548$.}
\label{spectrum}
\end{figure}

Whereas the above is merely a consistency check for the low-frequency stability boundaries, the computation of the 
%KZ: medium? or mid?
%mean-frequency 
medium-frequency
transition is more instructive. To detect and trace it we have again used a root solver, this time near roots that barely crossed the imaginary axis. The behavior of these critical spectral curves seems to depend on $F$.  Indeed, for $2<F\le 12.2$, the curves connect back to the origin as $\xi$ is varied, while for $12.3\le F\le 16.3$ the curves
cross real axis at about $-1.5$: see Figure~\ref{medfrespectrum} (a) and (b).  Furthermore, for $F\approx 2.74$ there appear to be multiple (local) curves that touch the imaginary axis away from the origin,
making the tracking of this boundary potentially difficult.
For example, as $F$ decreases through $2.74$ the location of the purely imaginary points away from the origin in the spectrum changes from
a pair farther away from the origin to a pair nearer the origin; see Figure~\ref{medfrespectrum}(c)-(e).  It is this switching between critical branches of spectrum that is responsible for the sharp edge of the stability boundary in Figure~\ref{fig_rosetta}(d).

\begin{figure}[htbp]
\begin{center}
\includegraphics[scale=.5]{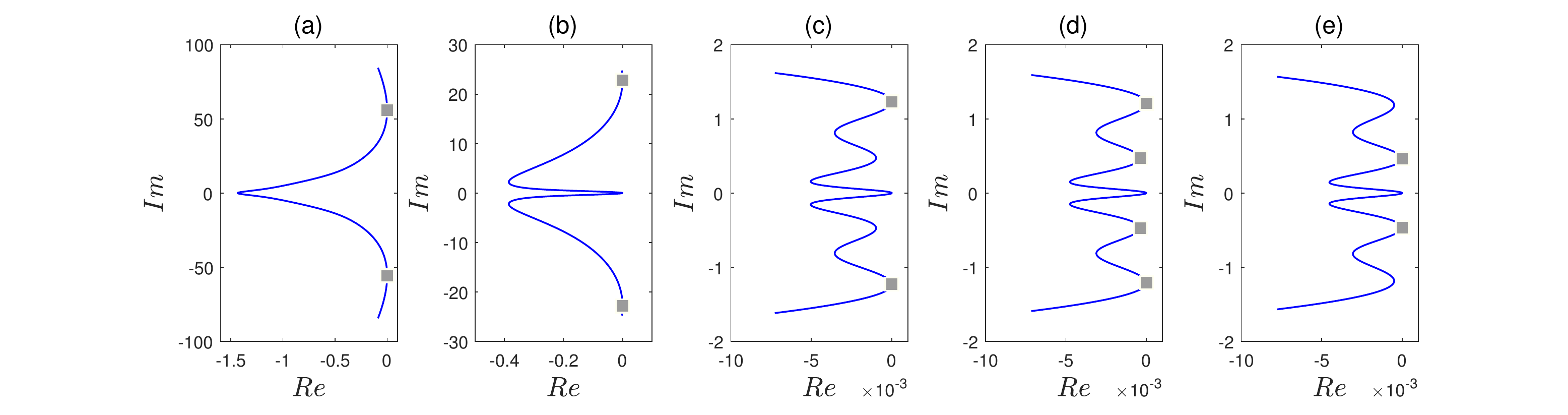}\\
\end{center}
\caption{Spectral plots for points on the 
%KZ
%mean-frequency stability 
mid-frequency boundary, with the purely imaginary (non-zero) points of the spectrum indicated with a small black box. Everywhere $H_s=1$. (a) $F=16.3, \underline{H}_-=0.0920122990378320$; (b) $F=8,\underline{H}_-=0.178726395676078$; (c) $F=2.75 \underline{H}_-=0.658783869495317$; (d) $F=2.74, \underline{H}_-=0.661884455727258$; (e) $F=2.73 \underline{H}_-=0.665012235912442$.  Note that panes (c)-(e) show a jump in the imaginary part of the critical purely imaginary points as $F$ decreases through $F\approx 2.74$. }
\label{medfrespectrum}
\end{figure}

\appendix

%MR
\section{Local solvability}\label{s:local}
In this appendix we treat solvability in Sobolev spaces or analytic setting of a class of singular ODEs. Our analysis significantly differs from, but obviously also overlaps with,
the classical treatment of regular-singular points of ODEs, as found in \cite{C}. We perform the required estimates first for a ``toy" model equation, and then for the general case at hand.

\subsection{Model problems}\label{s:model}
Starting with the principal singular part, we are led to consider the model constant-coefficient scalar equation
\be\label{modeleq}
xv'+a_0 v= g,
\ee
on a bounded interval $I\subset\RM$ containing $0$, where here $a_0\in\CM$ is a constant. 
%CHANGED-MJ wasn't sure what "consistent solving" meant... think I captured the spirit of what was trying to be said.
%We want to develop a consistent solving, allowing $a_0$ to have arbitrary large real part.
Our goal here is to study the continuous solvability of the equation \eqref{modeleq} within different function classes and, in particular, allowing $a_0$ to have arbitrary large real part.

For exposition, we begin with the $L^p$-theory, $1<p\leq\infty$. In this case, for each $g\in L^p(I)$ and $\Re(a_0)>1/p$ there is at most one solution of \eqref{modeleq} 
since $x\mapsto |x|^{-a_0}$ does not belong to $L^p(I)$ on either side of zero.  We show now that under the same condition there is indeed exactly one solution in $L^p(I)$, 
depending continuously on $g\in L^p(I)$.   The claimed solution is given explicitly by
\be\label{special}
v(x)= \int_0^x \left(\frac{|y|}{|x|}\right)^{a_0} g(y)\,\frac{dy}{y}
=\int_0^1 t^{a_0} g(x\,t)\,\frac{dt}{t}\,.
\ee
Note that the condition $\Re(a_0)>1/p$ ensures that the map $y\mapsto |y|^{a_0}/y$ belongs to $L^{p'}(I)$, where $p'$ is Lebesgue-conjugate to $p$, i.e. $1/p+1/p'=1$, since then $p'(\Re(a_0)-1)>-1$. In particular, \eqref{special} makes sense\footnote{We warn the reader, however, that the corresponding pessimistic pointwise bound $|v(x)|\lesssim |x|^{-1/p}$ incorrectly suggests that $v\notin L^p(I)$.} whenever $x\neq0$.  Moreover, since $|x\,v(x)|\lesssim |x|^{1-1/p}\to0$ as $x\mapsto0$, one readily checks that \eqref{special} provides a distributional solution to \eqref{modeleq}. 
The continuity of the solution map follows from the elementary estimate
$$
\|v\|_{L^p(I)}
\leq\int_0^1 t^{\Re(a_0)} \|g(\cdot\,t)\|_{L^p(I)}\,\frac{dt}{t}
\leq \|g\|_{L^p(I)} \int_0^1 t^{\Re(a_0)-1/p}\,\frac{dt}{t}=\frac{1}{\Re(a_0)-1/p}\,\|g\|_{L^p(I)}\,,
$$
which proves the claim.

Extending to $W^{k,p}$-solvability, $k\in\NN^*$, $1<p\leq\infty$, we expect to relax the condition on $a_0$. 
The fastest way to check this claim is to observe that at this level of regularity \eqref{modeleq} may be equivalently written as
$$
x\,(v^{(k)})'+(a_0+k)\,v^{(k)}= g^{(k)}
$$
(to be solved with $v^{(k)}\in L^p(I)$), supplemented by initial data constraints
$$
(a_0+\ell) v^{(\ell)}(0)= g^{(\ell)}(0)\,,\quad0\leq \ell\leq k-1\,.
$$
As a result, when $\Re(a_0)>-k+1/p$ any $W^{k,p}$-solution satisfies
$$
\|v^{(k)}\|_{L^p(I)}
\leq\frac{1}{\Re(a_0)+k-1/p}\,\|g^{(k)}\|_{L^p(I)}\,,
$$
and the set of solutions is a singleton if $a_0\notin\{-(k-1),\cdots,0\}$, and a line parametrized by $v^{(-a_0)}(0)$ otherwise.

Finally, we consider the solvability of \eqref{modeleq} within the class of analytic functions on $I$.  Given an analytic function $g$ on $I$, then as above we find
that if $a_0\notin -\NN$ there is at most one solution of \eqref{modeleq} while if $a_0\in -\NN$ then any solution, if it exists, is determined only up to a multiple of $x\mapsto x^{-a_0}$. 
For simplicity, here we only discuss the case when $\Re(a_0)>-1$. Since $a_0\,v(0)=g(0)$ is a necessary condition for solvability, when $a_0=0$ we assume moreover that $g(0)=0$ and
otherwise choose $v(0)$ such that this constraint is satisfied (uniquely when $a_0\neq0$, arbitrarily otherwise).  In this case, it is easy to check that the $W^{1,\infty}$-formula
$$
v(x)\,=\,v(0)+\int_0^1 t^{a_0} \frac{g(x\,t)-g(0)}{t}\,dt
$$
actually defines an analytic solution on $I$ whenever $g$ is analytic on $I$.  In particular, the above solution formula may be extended to $x$ lying in a complex convex neighborhood $B$ of $0$,
defining a holomorphic extension satisfying the continuity estimate
$$
\max_{z\in B} \frac{|v(z)-v(0)|}{|z|^N} \leq \frac{1}{\Re(a_0)+N}\,\max_{z\in B} \frac{|g(z)-g(0)|}{|z|^N},
$$
where here $N$ is chosen so that $0$ is a root of $g-g(0)$ of order at least $N$. 

\

Having considered the scalar model \eqref{modeleq} above and with actual singular problem deriving from \eqref{sv_intro} in mind, 
we now consider the continuous solvability of a
%replace \eqref{modeleq} with a 
triangular constant-coefficient system of the form
\be\label{modelsys}
\left\{
\begin{array}{rcl}
v_1'&=&g_1\\
x\,v_2'&=&L_0v_1\,-\,a_0\,v_2\,+\,g_2,
\end{array}\right.
\ee
considered again on a bounded interval $I\subset\RM$ containing $0$, where here 
%$L_0,a_0\in\CM$ are constants,
$L_0\in\mathcal{M}_{1,n-1}(\CM)$ is a matrix, $a_0\in\CM$ is a constant,
$v_1$ is $\RM^{n-1}$-valued and $v_2$ is scalar-valued.  Given $1<p\leq\infty$ and given $g_1,g_2\in L^p(I)$, we find that if
%Choose $1<p\leq\infty$. Then provided 
$\Re(a_0)>1/p$ then the set of $L^p$-solutions of \eqref{modelsys} is an $(n-1)$-dimensinoal space, parametrized by $v_1(0)$, with solutions satisfying the continuity estimate
\begin{align*}
\|v_1\|_{L^p(I)}&\leq |v_1(0)|\,|I|^{1/p}+\|g_1\|_{L^p(I)}\,|I|\\
\|v_2\|_{L^p(I)}&\leq \frac{1}{\Re(a_0)-1/p}\,
\left[|L_0|\,|v_1(0)|\,|I|^{1/p}+|L_0|\|g_1\|_{L^p(I)}\,|I|+\|g_2\|_{L^p(I)}\right]\,.
\end{align*}
Likewise when $k\in\NN^*$ and $g_1,g_2\in W^{k,p}(I)$, we find that if 
\begin{equation}\label{sobolevcond}
\Re(a_0)>-k+1/p\,,\qquad 
a_0\notin\{-(k-1),\cdots,-1\}\qquad\text{and}\qquad
(a_0,L_0)\quad\text{is non-zero,} 
\end{equation}
then the set of $W^{k,p}(I)$-solutions is again an $(n-1)$-dimensional space, parametrized by values $(v_1(0),v_2(0))\in\mathbb{R}^n$ satisfying the condition
$$
0=L_0v_1(0)\,-\,a_0\,v_2(0)\,+\,g_2(0)
$$
with solutions satisfying the continuity estimate
\begin{align*}
\|v_1^{(k)}\|_{L^p(I)}&\leq |g_1^{(k-1)}(0)|\,|I|^{1/p}+\|g_1^{(k)}\|_{L^p(I)}\,|I|\\
\|v_2^{(k)}\|_{L^p(I)}&\leq \frac{1}{\Re(a_0)-1/p}\,
\left[|L_0|\,|g_1^{(k-1)}(0)|\,|I|^{1/p}+|L_0|\|g_1^{(k)}\|_{L^p(I)}\,|I|+\|g_2^{(k)}\|_{L^p(I)}\right]\,,
\end{align*}
and data at $x=0$ being determined via the relations
$$
v_1^{(\ell)}(0)=g_1^{(\ell-1)}(0)\,, 
\qquad
v_2^{(\ell)}(0)=\frac{1}{a_0+\ell}\left[L_0g_1^{(\ell-1)}(0)+g_2^{(\ell)}(0)\right],\quad
1\leq \ell\leq k-1\,.
$$

Finally, we consider the solvability of \eqref{modelsys} within the class of analytic functions on $I$ under the assumptions that
$$
\Re(a_0)>-1\,,\qquad\text{and}\qquad
(a_0,L_0)\quad\text{is non-zero.} 
$$
To this end, fix $(v_1(0),v_2(0))$ such that
$$
0=L_0v_1(0)\,-\,a_0\,v_2(0)\,+\,g_2(0)\,.
$$
Then an easy calculation shows there exists a unique analytic solution to \eqref{modelsys} starting from $(v_1(0),v_2(0))$ and its complex extension to a complex convex neighborhood of zero $B$ 
satisfying the continuity estimate
\begin{align*}
\max_{z\in B} \frac{|v_1(z)-v(0)|}{|z|^{N+1}} &\leq \frac{1}{N+1}\,\max_{z\in B} \frac{|g_1(z)|}{|z|^N}\\
\max_{z\in B} \frac{|v_2(z)-v(0)|}{|z|^{N+1}} &\leq \frac{1}{\Re(a_0)+N+1}\,
\left[\frac{|L_0|}{N+1}\,\max_{z\in B} \frac{|g_1(z)|}{|z|^N}+\max_{z\in B} \frac{|g_2(z)-g_2(0)|}{|z|^{N+1}}\right]\,,
\end{align*}
provided $(g_1,g_2)$ are holomorphic on $B$ and $0$ is a root of $g_1$ of order at least $N$ and of $g_2-g(0)$ of order at least $N+1$.

\subsection{The general case}\label{s:general}
We now replace \eqref{modelsys} with an inhomogeneous system in the form
\be\label{gensys}
\bp v_1'\\x\,v_2'\ep=\bp \tilde{A}&C\\ L & -a \ep \bp v_1\\v_2\ep+
\bp g_1\\g_2\ep
\ee
with smooth coefficients and show that the problem keeps the same structure as above, with $a_0:=a(0)$ and $L_0:=L(0)$. 

Concerning Sobolev solving, it stems from the estimates on the model system \eqref{modelsys} through a contraction argument that 
\eqref{gensys} is continuously solvable in $W^{k,p}(I)$ provided \eqref{sobolevcond} holds and 
%this is true provided 
$I=(-\eps,\eps)$ with $\eps$ sufficiently small (depending only on coefficients). This may then be extended to an arbitrary bounded interval containing $0$ by relying on the classical regular ODE theory to solve outside a small ball about 
the singular point $x=0$.

As for analytic solving when
$$
\Re(a(0))>-1\,,\qquad\text{and}\qquad
(a(0),L(0))\quad\text{is non-zero}, 
$$
with initial data $(v_1(0),v_2(0))$ satisfying the condition
$$
0=L(0)v_1(0)\,-\,a(0)\,v_2(0)\,+\,g_2(0)\,,
$$
we again only need to examine analyticity in a neighborhood of $0$. The latter also follows from estimates of the previous subsection through Picard iteration. Indeed if $(g_1,g_2)$ is holomorphic on $B$, a complex convex neighborhood of zero, the above estimates yield that, setting $(v_1,v_2)=\sum_{N\in\NN} (v_1^{(N)},v_2^{(N)})$ with $(v_1^{(0)},v_2^{(0)})$ constant equal to $(v_1(0),v_2(0))$,
there exists a constant $K$ depending only on coefficients (and blowing up when $\Re(a(0))\to-1$) such that
$$
\max_{z\in B} \frac{|(v_1^{(1)},v_2^{(1)})(z)|}{|z|}\leq K\,\,
\left[\|(v_1(0),v_2(0))\|+\max_{z\in B}|g_1(z)|+\max_{z\in B} \frac{|g_2(z)-g_2(0)|}{|z|}\right]
$$ 
and, for any integer $N\geq1$
$$
\max_{z\in B} \frac{|(v_1^{(N+1)},v_2^{(N+1)})(z)|}{|z|^{N+1}}\leq \frac{K}{N+1}\,\max_{z\in B} \frac{|(v_1^{(N)},v_2^{(N)})(z)|}{|z|^N}
$$
%with $K$ a constant depending only on coefficients (and blowing up when $\Re(a(0))\to-1$), thus, for any $N\geq1$, 
Thus, any integer $N\geq 1$ solutions of \eqref{gensys} satisfy the estimate
$$
\max_{z\in B} \frac{|(v_1^{(N)},v_2^{(N)})(z)|}{|z|^N}\leq \frac{K^N}{N!}\,\,
\left[\|(v_1(0),v_2(0))\|+\max_{z\in B}|g_1(z)|+\max_{z\in B} \frac{|g_2(z)-g_2(0)|}{|z|}\right]\,.
$$ 
This is sufficient to deduce the claimed convergence, providing an analytic solution of \eqref{gensys} provided the $g_j$ are themselves analytic.

\subsection{Application to \eqref{sv_intro}}\label{s:application}

We now demonstrate that the above general theory applies to the spectral problem associated with our roll waves of \eqref{sv_intro}, yielding that local solvability in $H^1$
for the system \eqref{linearizedeq}-\eqref{eigen-normalization} occurs whenever the spectral parameter $\lambda$ satisfies
$$
\Re(\lambda)\,>\,-\frac14\,\frac{F-2}{\sqrt{H_s}}
$$ 
as well as analytic solvability when 
$$
\Re(\lambda)\,>\,-\frac12\,\frac{F-2}{\sqrt{H_s}}\,.
$$ 

We begin by rewriting the inhomogeneous version of \eqref{linearizedeq} as
%The starting point is 
\begin{align*}
\lambda\,h+(q-ch)'&=f_1\\
\lambda\,q+\left(\left(p'(H)-\frac{Q^2}{H^2}\right)\,h+\left(\frac{2Q}{H}-c\right)\,q\right)'&=\d_hr(H,Q)\,h+\d_qr(H,Q)\,q\,+\,f_2
\end{align*}
where 
\[
p(h)=\frac{h^2}{2F^2}\qquad\textrm{and}\qquad r(h,q)=h-\frac{|q|\,q}{h^2}
\]
and where $f_1$ and $f_2$ are given functions.
By changing the unknown from $(h,q)$ to $(w_1,w_2)=(ch-q,h)$, the above system becomes
\begin{align*}
w_1'&=\lambda w_2-f_1\\
\left(p'(H)-\frac{q_0^2}{H^2}\right)w_2'&=
\left(-\left(p'(H)-\frac{q_0^2}{H^2}\right)'+\d_hr(H,Q)+c\d_qr(H,Q)-\frac{2\lambda\,q_0}{H}\right)w_2\\
&\quad+\left(\lambda-\d_qr(H,Q)-\left(\frac{2q_0}{H}\right)'\right)w_1+f_2+\left(\frac{2q_0}{H}-c\right)f_1
\end{align*}
where we have introduced the constant $q_0$ such that $Q-cH=-q_0$: see \eqref{UHrel}.

Now, let $x_s$ denote the unique point in $(0,X)$ where $H(x_s)=H_s$, where $H_s$ is such that $p'(H_s)-\frac{q_0^2}{H_s^2}=0$.  Recalling \eqref{smooth}, we see that differentiating the profile equation 
$$
\left(p'(H)-\frac{q_0^2}{H^2}\right)\,H'=r(H,Q)\,,\qquad Q-cH=-q_0\,,
$$
and evaluating at $x_s$ yields
$$
-\left(p'(H)-\frac{q_0^2}{H^2}\right)'(x_s)+\d_hr(H(x_s),Q(x_s))+c\d_qr(H(x_s),Q(x_s))\,=\,0\,.
$$
In particular,  factoring 
$$
\left(-p'(H)-\frac{q_0^2}{H^2}\right)(x)\,=\,(x-x_s)\,d(x)
$$
with $d$ nonvanishing on $(0,X)$, dividing the second resolvent equation by $d$, translating $x_s$ to $0$ and defining $(v_1,v_2)$, $(g_1,g_2)$ accordingly, we see the above inhomogeneous
system takes abstract general form \eqref{gensys} with
$$
a(0)\,=\,\frac{1}{d(x_s)}\,\frac{2\lambda\,q_0}{H(x_s)}
\,=\,\frac{2\lambda\,q_0}{H_s\,\left(p''(H_s)+\frac{2q_0^2}{H^3}\right)\,H'(x_s)}
\,=\,\frac{2\lambda\sqrt{H_s}}{F-2}
$$
and
$$
L(0)\,=\,\frac{1}{d(x_s)}\,\left(\lambda-\d_qr(H(x_s),Q(x_s))+\frac{2q_0}{H_s^2}H'(x_s)\right)
\,=\,\frac{F}{F-2}\left(\lambda+\frac{2}{3}\frac{F+1}{\sqrt{H_s}}\right)
$$
where we have used formulas computed in Section~\ref{s:existence}, $q_0=H_s^{3/2}/F$, $c=H_s^{1/2}(1+1/F)$ and $H'(x_s)=F(F-2)/3$. 
Notice that $(a(0),L(0))$ cannot vanish and that, since $F>2$, the conditions $\Re(a(0))>-1/2$ and $\Re(a(0))>-1$ correspond, respectively, to the announced constraints for $H^1$ and analytic solvability.

As an outcome of the above analysis, note that when 
$$
\Re(\lambda)\,>\,-\frac12\,\frac{F-2}{\sqrt{H_s}},
$$
taking $(f_1,f_2)\equiv0$ above implies the set of analytic solutions of \eqref{linearizedeq} can be parametrized by data at the sonic point $(h(x_s),q(x_s))$, which may be chosen arbitrarily provided it satisfies the condition
$$
-\frac{2\lambda\,q_0}{H_s}\,h(x_s)+\left(\lambda-\d_qr(H(x_s),Q(x_s))+\frac{2q_0}{H_s^2}H'(x_s)\right)(ch(x_s)-q(x_s))\,=\,0.
$$
Note the above condition may be written more explicitly as
$$
\left(\lambda\sqrt{H_s}(F-1)+\frac23(F+1)^2\right)\,h(x_s)
-\left(\lambda+\frac{2}{3}\frac{F+1}{\sqrt{H_s}}\right)F\,q(x_s)\,=\,0\,.
$$ 

\section{Computational framework}\label{s:computations}

\subsection{Computational environment} In carrying out our numerical investigations, we have used a Lenovo laptop with 8GB memory and a quad core AMD processor with 1.9GHz processing speed and a 2009 Mac Pro with 16GB memory and two quad-core Intel processors with 2.26 GHz processing speed for coding and debugging. The main parallelized computation is done in the compute nodes of IU Karst, a high-throughput computing cluster. It has $228$ compute nodes. Each node is an IBM NeXtScale nx360 M4 server equipped with two Intel Xeon E5-2650 v2 8-core processors and with 32 GB of RAM and 250 GB of local disk storage. 

\subsection{Computational time}
The following computational times are times elapsed in a single processor of IU Karst.
\begin{table}[htbp]
\centering
\label{table1}
\begin{tabular}{|l|l|l|l|l|l|l|}
\hline
   \diaghead{DDDDDD}
{$\lambda,\xi$}{$F,\underline{H}_-$}      & $3,\underline{H}_{hom}(3)+10^{-5}$ & $3,0.8$ & $8,\underline{H}_{hom}(8)+10^{-5}$ & $8,0.8$ & $16,\underline{H}_{hom}(16)+10^{-5}$ & $16,0.8$ \\ \hline
$0.001,0$ & 0.04s                              & 0.02s   & 0.09s                              & 0.01s   & 0.02s                                & 0.01s    \\ \hline
$1,0$     & 0.08s                              & 0.02s   & 0.14s                              & 0.01s   & 0.08s                                & 0.01s    \\ \hline
$1000,0$  & 6.06s                              & 2.95s   & 2.11s                               & 0.13s   & 1.25s                                & 0.03s    \\ \hline
\end{tabular}
\caption{Times to compute a single Evans-Lopatinsky determinant.}
\end{table}	
\begin{table}[htbp]
\centering
\small
\label{table2}
\begin{tabular}{|l|l|l|l|l|l|l|}
\hline
$F,\underline{H}_-$ & $3,\underline{H}_{hom}(3)+10^{-5}$ & $3,0.8$ & $8,\underline{H}_{hom}(8)+10^{-5}$ & $8,0.8$  & $16,\underline{H}_{hom}(16)+10^{-5}$ & $16,0.8$ \\ \hline
Time                &                                  92974s  &      9055s   &                                 24373s   & 173s  &   14002s  & 112s  \\ \hline
Stability           &                      stable            &       unstable  &                               unstable     & unstable &                             unstable         & unstable \\ \hline
\end{tabular}
\caption{Times using winding numbers to classify points as stable/unstable. In all situations, $r=0.01,R=400$ and $\ell\xi\in[-\pi:\frac{2\pi}{1000}:\pi]$ where the small half circle is divided uniformly into $1000$ pieces, the two straight line segments are each divided uniformly into $2000$ pieces, and the large half circle is divided uniformly into $2000$ pieces. The split Evans-Lopatinsky determinant is then evaluated at $7000$ points on the contour.}
\end{table}

%TODO: review this fyi:
%
%\section{Notes}\label{s:notes}
%We have omitted the following topics from the paper as originally planned:
%
% Asymptotics (near-onset, large- and small-amplitude, large-$F$).
%
%\medskip
%
%{it Rationale} (KZ): the paper is already sufficiently complicated, and centered around the scalar second-order formulation used in numerics and explicit computations. Also, though interesting, these do not seem to add much to the physical conclusions made here, which are quite amazing.
%Better therefore to keep focused, and push these additional topics to a further paper if desired.

\bibliographystyle{alphaabbr}
\bibliography{Ref-roll}

\end{document}